\documentclass[10pt]{amsart}

\usepackage[utf8]{inputenc}
\usepackage{amsthm}
\usepackage{amsmath}
\usepackage{amscd}
\usepackage[mathscr]{euscript}
\usepackage{tikz-cd}
\usepackage{url}
\usepackage{hyperref}
\usepackage{amssymb}
\subjclass[2020]{14E08, 14D06, 14M25}

\newtheorem{theorem}{Theorem}
\numberwithin{theorem}{section}
\newtheorem{proposition}[theorem]{Proposition}
\newtheorem{lemma}[theorem]{Lemma}
\newtheorem{corollary}[theorem]{Corollary}
\theoremstyle{definition} 
\newtheorem{definition}[theorem]{Definition}
\theoremstyle{remark}

\numberwithin{theorem}{section}
\numberwithin{equation}{section}
\providecommand{\C}{\mathbb C} 

\providecommand{\Q}{\mathbb Q}
\providecommand{\RR}{\mathbb R}

\providecommand{\ZZ}{\mathbb Z}

\providecommand{\charac}{\operatorname{char}}

\providecommand{\supp}{\operatorname{Supp}}

\providecommand{\val}{\operatorname{val}}

\providecommand{\Spec}{\operatorname{Spec}}

\providecommand{\A}{\mathbb A}

\providecommand{\PP}{\mathbb P}
\providecommand{\Gm}{\mathbb{G}_m}

\providecommand{\pr}{\mathrm{pr}}
\providecommand{\Gr}{\mathrm{Gr}}
\providecommand{\HH}{\mathscr{H}}

\providecommand{\id}{\mathrm{id}}
\providecommand{\VOL}{\mathrm{Vol}_{\mathrm{sb}}}
\providecommand{\Rt}{\mathscr{R}}
\providecommand{\Kt}{\mathscr{K}}
\providecommand{\spe}{\operatorname{sp}}
\providecommand{\bdd}{\operatorname{bdd}}

\providecommand{\SB}{\mathrm{SB}}

\providecommand{\OO}{\mathscr O}

\providecommand{\ker}{\operatorname{ker}}
\title{Stable rationality of hypersurfaces of mock toric variety II}
\author{Taro Yoshino}
\date{\today}
\address{Graduate School of Mathematical Sciences, The University of Tokyo, 3-8-1 Komaba,
Meguro-ku, Tokyo, 153-8914, Japan}
\email{yotaro@ms.u-tokyo.ac.jp}

\begin{document}

\begin{abstract}
    In recent years, there has been a development in approaching rationality problems through motivic methods (cf. [Kontsevich--Tschinkel'19], [Nicaise--Shinder'19], [Nicaise--Ottem'21]). 
    
    This method requires the explicit construction of degeneration families of curves with favorable properties. 
    While the specific construction is generally difficult, [Nicaise-–Ottem'22] combines combinatorial methods to construct degeneration families of hypersurfaces in toric varieties and mentions the stable rationality of a very general hypersurface in projective spaces. 
    
    In this paper, we substitute mock toric varieties for toric varieties and we prove the following theorem from the motivic method: If a very general hypersurface of degree $d$ in $\mathbb{P}^{2n-5}_\mathbb{C}$ is not stably rational, then a very general hypersurface of degree $d$ in $\mathrm{Gr}_\mathbb{C}(2, n)$ is not stably rational. 
\end{abstract}

\maketitle
\tableofcontents
\section{Introduction}
 \subsection{Motivic method for the rationality problem}
    The rationality problem is one of the central problems in algebraic geometry. 
    Recently, there exists motivic approach for the rationality problem(cf. \cite{KT19}, \cite{NO21}, \cite{NS19}). 
    
    Let $K$ be a field.  
    Two schemes X and Y of finite type over $K$ are called stably birational if $X\times \PP^m_K$ is birational to $Y\times \PP^n_K$ for some non-negative integers $m$ and $n$. 
    If a $K$-variety $X$ is stably birational to $\Spec(K)$, then we call $X$ is stably rational. 
    In particular, a rational $K$-variety $X$ is stably rational. 
    Let $\SB_K$ denote the set of stable birational equivalence classes $\{X\}_\mathrm{sb}$ of integral $K$-schemes $X$ of finite type, and $\ZZ[\SB_K]$ be the free abelian group on $\SB_K$.    
    Moreover, for any variety $X$ over $K$, it holds that $X$ is stably rational over $K$ if and only if $\{X\}_\mathrm{sb} = \{\Spec(K)\}_\mathrm{sb}$ by definition. 
    In \cite{NO21}, they constructed a ring morphism as follows: 
    \begin{proposition}\label{prop: NONO21}\cite[Lemma.3.3.5]{NO21}
        Let $k$ be an algebraically closed field with $\charac(k) = 0$ and let $\Kt$ denote the field of the Puiseux series over $k$. 
        Then there exists a unique ring morphism $\VOL\colon \ZZ[\SB_{\Kt}]\rightarrow \ZZ[\SB_k]$ such that for every strictly toroidal proper $\Rt$-scheme $\mathscr{X}$ with smooth generic fiber $X = \mathscr{X}_{\Kt}$, we have 
        \[
            \VOL(\{X\}_{\mathrm{sb}}) = \sum_{E\in \mathcal{S}(\mathscr{X})}(-1)^{\mathrm{codim}(E)}\{E\}_{\mathrm{sb}}
        \], 
        where $\{E\}$ is a stratification of $\mathscr{X}_k$. 
    \end{proposition}
    At this point, we omit the detailed definition of a strictly toroidal model(See Definition \ref{NONO21}). 
    Still, it is worth mentioning that if $\mathscr{X}$ is flat over $\Spec(\Rt)$ such that the generic fiber is smooth over $\Spec(\Kt)$ and the closed fiber is a reduced simple normal crossing divisor of $\mathscr{X}$, then $\mathscr{X}$ is a strictly toroidal model. 
    We remark that such ring morphisms appeared in these references(cf. \cite{KT19}, \cite{NO21}, \cite{NS19}), although the details of the forms differ. 
    In application, for a $\Kt$-variety $X$, if there exists a strictly toroidal proper $\Rt$-scheme $\mathscr{X}$ with smooth generic fiber $X = \mathscr{X}_{\Kt}$ and $\VOL(\{X\}_{\mathrm{sb}})\neq \{\Spec(k)\}_\mathrm{sb}$, then $X$ is not stably rational over $\Kt$. 

    In this article, we show the following theorem by using the above ring morphism: 
    \begin{theorem}[See Theorem \ref{thm: d in Grassmanian}]
            If a very general hypersurface of degree $d$ in $\PP_\C^{2n-5}$ is not stably rational, then a very general hypersurface of degree $d$ in $\Gr_\C(2, n)$ is not stably rational.
    \end{theorem}
    In particular, the following corollary holds from \cite[Corollary 1.2]{S19} immediately:  
    \begin{corollary}[See Theorem \ref{cor: log bound}]
        If $n\geq 5$ and $d \geq 3 + \log_2(n-3)$, then a very general hypersurface of degree $d$ of $\Gr_\C(2, n)$ is not stably rational.         
    \end{corollary}
    \subsection{The difficulty of motivic method}
    While this method of determination is quite simple, in practice, there are several difficult problems. 
    For example, we have the following three problems : 
    \begin{itemize}
        \item [1.] To construct a ``good'' proper model with a smooth generic fiber. 
        \item [2.] To enumerate strata in the stratification of the model. 
        \item [3.] To determine the stable birational equivalence classes (or birational equivalence classes) for each stratum.
    \end{itemize}
    Previous work \cite{NO22} successfully overcomes these three problems as follows(cf. \cite[Theorem. 3.14]{NO22}):
    \begin{itemize}
        \item [1.] They focused on a hypersurface $Y$ of an algebraic torus. 
        This hypersurface has a Tropical compactification $\overline{Y}$ (cf.\cite{Tev}); in particular, they were compactified as a hypersurface of a proper toric variety $X$. 
        If $Y$ has a good condition("Newton nondegenerate"), then the toric resolution of $X$ induces the resolution of $\overline{Y}$. 
        This property is applied to the construction of a good model. 
        \item [2.] The ambient space is a toric variety, and its decomposition of torus orbits induces the stratification of this hypersurface. 
        Moreover, each stratum is also a hypersurface of an algebraic torus, and we can calculate the defining Laurant polynomial of each stratum combinatorially.  
        \item [3.] They construct the strictly toroidal model of a quartic 5-fold, which has a stratum $E$ such that $E$ is birational to a very general quartic double fourfold. In particular, it is not stably rational (cf.\cite{HPT19}). 
        In addition to this, they showed that a very general quartic 5-fold is not stably rational(\cite[Cor. 5.2]{NO22}). 
    \end{itemize}
    It is natural to consider applying the method of \cite{NO22} to mention the rationality of hypersurfaces in other rational varieties. 
    For example, Grassmanian varieties are. 
    Grassmanian varieties have a Schubert decomposition; in particular, they have an affine space as a dense open subspace. 
    Moreover, an affine space has an algebraic torus as a dense open subspace. 
    Thus, it seems straightforward to construct a strictly toroidal model of the variety, which is birational to the hypersurface of the Grassmanian variety. 
    However, this does not work well because the defining Laurant polynomial of the hypersurface is not Newton nondegenerate. 
    The necessity of this condition arises from the fact that in this method, singularities of this hypersurface can only be resolved through a toric resolution of the ambient toric variety. 
    For example, suppose this hypersurface has a singular point in the dense open subset of the ambient toric variety. 
    In that case, we cannot resolve the singularity of it in this case. 
    
    \subsection{Mock toric varieties}
    In the previous subsection, we indicate that there is a limit to constructing concretely the strict toroidal models of some varieties from toric varieties. 
    In this article, we substitute mock toric varieties for toric varieties and prove the main theorem \ref{thm: d in Grassmanian}. 
    Mock toric varieties are introduced by the author of this paper in \cite[Definition 3.1]{Y23}. 
    From \cite[Section 3 \& 4]{Y23}, mock toric varieties inherit some same properties as toric varieties. 
    In particular, if some conditions hold, we can construct proper strictly toroidal models of hypersurfaces of mock toric varieties concretely as these of toric varieties(See \cite[Proposition 7.7]{Y23}). 
    Moreover, we show that we can construct the birational proper strict toroidal models of general hypersurfaces of $\Gr_\Kt(2, n)$(See Proposition \ref{prop:grassman-mock}) from mock toric varieties, and we show the main theorem. 

    Consequently, the proof strategy is the same as in \cite{NO22}, with the difference that the scope has been extended from hypersurfaces of toric varieties to those of mock toric varieties. 
    The key point of this proof is the result of the stable birational equivalence class of a very general hypersurface in a projective space(\cite[Corollary. 4.2]{NO22}). 
    We remark that for any integer $m\geq 3$, we can't find these models of hypersurfaces of $\Gr_\Kt(m, n)$ related to mock toric varieties.
    \subsection{The rationality problem for hypersurfaces of Grassmanian varieties}
        We discuss previous work on the rationality of hypersurfaces in $\Gr_\C(2, n)$. 
        We fix a Pl\"{u}cker embedding $\Gr_\C(2, n)\hookrightarrow \PP^{n(n-1)/2 - 1}_\C$, let $X$ denote $\Gr_\C(2, n)$, 
        and $X_{(r)}$ denote the intersection of $X$ and hypersurfaces of degree = $r$ in $\PP^{n(n-1)/2 - 1}_\C$. 
        Because $K_X = \OO_X(-n)$, a general $X_{(r)}$ is a Fano variety for $1\leq r\leq n$. 
        We itemize the rationality of $X_{(r)}$ as follows:
        \begin{itemize}
            \item $X_{(1)}$ is rational(\cite[Theorem 2.2.1.]{Xu11}). 
            \item $X_{(2)}\subset \Gr_\C(2, 4)$ is complete intersection of 2 quadrics in $\PP^5_\C$, so it is a rational. 
            \item A very general $X_{(3)}\subset \Gr_\C(2, 4)$ is not stably rational(\cite[Theorem 1.1]{HT19})
            \item $X_{(2)}\subset \Gr_\C(2, 5)$ is a Gushel-Mukai 5-fold. 
            Moreover, it is rational(\cite[Proposition 4.2]{DK18}).
            \item Ottem showed that a very general $X_{(3)}\subset \Gr_\C(2, 5)$ is not stably rational in his unpublished paper. 
            \item We can show that a very general $X_{(4)}\subset \Gr_\C(2, 5)$ is not stably rational from \cite[Theorem 7.1.2]{Simen20} and \cite[Corollary 4.2.2]{NO22}. 
        \end{itemize}
    
    \subsection{Outline of the paper}
        This paper is organized as follows. 
        In Section 2, we construct the formulation of a stable birational volume of hypersurfaces of mock toric varieties from \cite[Proposition 7.7]{Y23}. 
        In Section 3, we define an analog of Newton polytope for mock toric varieties and consider Newton nondegeneracy. 
        In general, confirming Newton nondegeneracy of a specific hypersurface takes time and effort. 
        Instead, we apply the formula to a general section of linear systems, following the idea from \cite{NO22}. 
        In Section 4, we apply the formula to the stable rationality problem of a very general hypersurface in $\Gr_\C(2, n)$ of some degrees.  
    \subsection{Acknowledgment}
        The author is grateful to his supervisor, Yoshinori Gongyo, for his encouragement.
\section{Stably birational volume of hypersurfaces of mock toric varieties}
        In this section, we compute the stably birational volume of hypersurfaces of mock toric varieties from the proper strictly toroidal model constructed in \cite[Proposition 7.7]{Y23}. 

        We recall some notations from \cite{NO21} to compute the stable birational volume using the results obtained from \cite{NO21}. 
		\begin{definition}[\cite{NO21}]\label{NONO21} 
            Let $\mathscr{X}$ be a flat and separated $\Rt$-scheme of finite type, and let $x\in\mathscr{X}_k$ be a point. 
            We say that $\mathscr{X}$ is strictly toroidal at $x$ if there exist a toric monoid $S$, an open neighbourhood $U$ of $x$ in $\mathscr{X}$, and a smooth morphism of $\Rt$-schemes $U \rightarrow \Spec(\Rt[S]/(\chi^\omega - t^q))$, where $q$ is a positive rational number, $\omega$ is an element of $S$, and $\chi^\omega$ is a torus invariant monomial in $k[S]$ associated with $\omega$ such that $k[S]/(\chi^\omega)$ is reduced. 
            We call that $\mathscr{X}$ is \textbf{strictly toroidal} if it is so at any $x\in \mathscr{X}_k$. 
            
            Let $\mathscr{X}$ be a strictly toroidal $\Rt$-scheme.  
            Then a \textbf{stratum} $E$ of $\mathscr{X}_k$ is a connected component of the intersection of a non-empty set of irreducible components of $\mathscr{X}_k$. 
            Let $\mathrm{codim}(E)$ denote the codimension of $E$ in $\mathscr{X}_k$, and $\mathcal{S}(\mathscr{X})$ denote the set of all strata in $\mathscr{X}_k$.
        \end{definition}

        Now, we compute the stably birational volume of hypersurfaces of mock toric varieties from the proper strictly toroidal model constructed in \cite[Proposition 7.7]{Y23}. 
        \begin{proposition}\label{prop: neo irreducible computation}
        		We keep the notation in \cite[Proposition 7.7]{Y23}. 
        		Then the following statements follow: 
        		\begin{enumerate}
        			\item[(a)] We assume that $f$ is nondegenerate for $\Delta'$. 
                Let $\tau\in\Delta'_{\spe}\cap {}_f\Delta'$ be a cone and $\{E^{(1)}_{\tau}, E^{(2)}_{\tau}, \ldots, E^{(r_\tau)}_{\tau}\}$ denote all connected components of $H_{W, f}\cap W_{\tau}$. 
                Then for any $1\leq j\leq r_\tau$, $E^{(j)}_{\tau}$ is a dense open subset of $\overline{E^{(j)}_{\tau}}$ and $\dim(E^{(j)}_{\tau}) = \dim(Y)-\dim(\tau)$. 
                Moreover, as a closed subscheme of $\overline{W_{\tau}}$, the following equation holds:
                \[
                    H_{W, f}\cap \overline{W_{\tau}} = \coprod_{1\leq j\leq r_\tau} \overline{E^{(j)}_{\tau}}
                \]
                \item[(b)] We use the notation in (a). 
                Then for any $1\leq j\leq r_\tau$, the following equations hold:
                \begin{align*}
                    \biggl\{E^{(j)}_{\tau}\biggr\}_{\mathrm{sb}} &= \biggl\{\overline{E^{(j)}_{\tau}}\biggr\}_{\mathrm{sb}}\\
                    E^{(j)}_{\tau} &= \overline{E^{(j)}_{\tau}}\cap W_{\tau}
                \end{align*}
                \item[(c)] We use the notation in (a). 
                We assume that the assumption of (a) holds. 
                Let $\mathfrak{I}$ denote a set which consists of all irreducible components of $\HH_k$. 
                Then the following equation holds:
                \[
                    \mathfrak{I} = \coprod_{\substack{\gamma\in \Delta'_{\spe}\cap {}_f\Delta'\\ 
                    \dim(\gamma) = 1}} \biggl\{\overline{E^{(j)}_{\gamma}}\biggr\}_{1\leq j\leq r_\gamma}
                \]
                \item[(d)] We assume that the assumption of \cite[Proposition 7.7(k)]{Y23} holds. 
                Moreover, we assume that $\Delta'$ is compactly arranged. 
                We use the notation in (a) and \cite[Proposition 7.7(j)]{Y23}. 
                Let $1\leq i\leq r$ be an integer and $E\in\mathcal{S}(\HH_i)$ be a stratum. 
                Then there unique exists $\tau\in \Delta'_{\bdd}\cap {}_f\Delta'$ and $1\leq j\leq r_{\tau}$ such that $E = \overline{E^{(j)}_{\tau}}$. 
                \item[(e)] We assume that the assumption in \cite[Proposition 7.7(k)]{Y23} holds and use the notation in (a) and \cite[Proposition 7.7(j)]{Y23}.  
                Let $\tau\in \Delta'_{\bdd}\cap {}_f\Delta'$ be a cone  and $1\leq j\leq r_{\tau}$ be an integer. 
                Then there exists $1\leq i\leq r$ and $E\in\mathcal{S}(\HH_i)$ such that $\overline{E^{(j)}_{\tau}} = E$. 
                \item[(f)] 
                We assume that the assumption of \cite[Proposition 7.7(k)]{Y23} holds. 
                Let $\Delta''$ be a strongly convex rational polyhedral fan in $(N'\oplus\ZZ)_\RR$ such that $\pi^1$ is compatible with the fans $\Delta''$ and $\Delta\times\Delta_!$, and $\supp(\Delta'') = (\pi^1_\RR)^{-1}(\supp(\Delta\times\Delta_!))$. 
                We use the notation in (a) and \cite[Proposition 7.7(j)]{Y23}. 
                Let $W'$ be a mock toric variety induced by $Z^1$ along $\pi^1_*\colon X(\Delta'')\rightarrow X(\Delta\times\Delta_!)$ and $s^1$. 
                Like as in (a), let $\{E^{(j)}_{\sigma}\}_{1\leq j\leq r_\sigma}$ denote all connected components of $H_{W', f}\cap W'_\sigma$ for $\sigma\in\Delta''$. 

                We assume that $\Delta'$ is a refinement of $\Delta''$ and $f$ is fine for $\Delta''$.
                Then $f$ is non-degenerate for $\Delta''$ and the following equation holds:
                \[
                    \sum_{1\leq i\leq r}\VOL\biggl(\bigl\{H^\circ_{Y_{\Kt}, f_i}\bigr\}_{\mathrm{sb}}\biggr) = \sum_{\substack{\sigma\in{\Delta''}_{\bdd}\\ 
                    \sigma\in{}_f\Delta''}}(-1)^{\dim(\sigma)-1}\biggl(\sum_{1\leq j\leq r_\sigma} \bigl\{E^{(j)}_{\sigma}\bigr\}_{\mathrm{sb}}\biggr)
                \]
        		\end{enumerate}
        \end{proposition}
        \begin{proof}
        		We prove the statements from (a) to (f) in order: 
        		\begin{enumerate}
        			\item[(a)] From \cite[Proposition 6.4]{Y23}, there exist 
                $\chi\in k[W_0]^*$ such that $\val^{\tau}_f + \val^{\tau}_{\chi'} \equiv 0$. 
                Let $p^{\tau}$ denote a ring morphism $\Gamma(W(\tau), \OO_{W})\rightarrow \Gamma(W_{\tau}, \OO_{\overline{W_{\tau}}})$ and $f^\tau$ denote $p^{\tau}(\chi f)$. 
                Then from \cite[Proposition 6.11(a)]{Y23}, $H_{W, f}\cap W_{\tau} = H^\circ_{\overline{W_{\tau}}, f^\tau}$. 
                Let $f^\tau = g_1g_2\cdots g_l$ be an irreducible decomposition. 
                We remark that $H_{W, f}\cap W_{\tau}$ is smooth over $k$. 
                Thus, $l = r_{\tau}$ and if it is necessary, we may replace the index such that $E^{(j)}_{\tau} = H^\circ_{\overline{W_{\tau}}, g_j}$ for any $1\leq j\leq r_{\tau}$. 
                Moreover, from \cite[Proposition 6.11(c)]{Y23}, $H_{W, f}\cap \overline{W_{\tau}} = H_{\overline{W_{\tau}}, f^\tau}$ and from \cite[Proposition 6.11(e)]{Y23}, $f^\tau$ is non-degenerate for $\Delta'[\tau]$. 
                Thus, from \cite[Proposition 6.12(c)]{Y23}, the following equation is a connected decomposition of $H_{W, f}\cap \overline{W_{\tau}} = H_{\overline{W_{\tau}}, f^\tau}$:
                \[
                    H_{\overline{W_{\tau}}, f^\tau} = \coprod_{1\leq j\leq r_{\tau}} H_{\overline{W_{\tau}}, g_j}
                \]
                We remark that $\overline{E^{(j)}_{\tau}} = H_{\overline{W_{\tau}}, g_j}$ for any $1\leq j\leq r_{\tau}$. 
                Because $H_{\overline{W_{\tau}}, g_j}$ is a scheme theoretic closure of $H^\circ_{\overline{W_{\tau}}, g_j}$ in $\overline{W_{\tau}}$ and $H^\circ_{\overline{W_{\tau}}, g_j} = H_{\overline{W_{\tau}}, g_j}\cap W_{\tau}$, $E^{(j)}_{\tau}$ is a dense open subset in $\overline{E^{(j)}_{\tau}}$ and $E^{(j)}_{\tau} = \overline{E^{(j)}_{\tau}}\cap W_{\tau}$ for any $1\leq j\leq r_{\tau}$. 
                Furthermore, from the above argument,  $\dim(E^{(j)}_{\tau}) = \dim(W_{\tau}) -1$. 
                Thus, from \cite[Proposition 3.4(d)]{Y23}, we have $\dim(E^{(j)}_{\tau}) = \dim(Y) - \dim(\tau)$. 
                \item[(b)] We have already shown the statement of (b) in the proof of (a).  
                \item[(c)] We use the notation in (a). 
                We can identify with $\HH_k$ and $H_{W, f}\times_{\A^1} \{0\}$. 
                In addition to this, $(H_{W, f}\times_{\A^1}\{0\})\cap W_{\tau} \neq \emptyset$ if and only if $\tau\in\Delta'_{\spe}\cap {}_f\Delta'$. 
                Therefore, $\HH_k$ has the following stratification:
                \begin{align*}
                    \HH_k &= \coprod_{\tau\in\Delta'_{\spe}\cap {}_f\Delta'}(\HH_k\cap W_{\tau})\\
                    &= \coprod_{\tau\in\Delta'_{\spe}\cap {}_f\Delta'}(\coprod_{1\leq j\leq r_{\tau}}E^{(j)}_{\tau})
                \end{align*}
                For any $\tau\in \Delta'_{\spe}\cap {}_f\Delta'$, we have $\tau\not\subset N'_\RR\times\{0\}$. 
                Hence, there exists $\gamma\in\Delta'_{\spe}$ such that $\gamma\preceq\tau$ and $\dim(\gamma) = 1$. 
                From \cite[Proposition 6.11(d)]{Y23}, $\gamma\in{}_f\Delta'$. 
                Moreover, for any $\tau_1$ and $\tau_2\in\Delta'_{\spe}\cap {}_f\Delta'$, $\HH_k\cap W_{\tau_1} \subset \HH_k\cap \overline{W_{\tau_2}}$ if and only if $\tau_2\subset \tau_1$ from \cite[Corollary 3.8]{Y23}. 
                 In conclusion, the following equation holds from (a):
                \[
                    \HH_k = \bigcup_{\substack{\gamma\in\Delta'_{\spe}\cap {}_f\Delta'\\\dim(\gamma) = 1}}(\bigcup_{1\leq j\leq r_{\gamma}}\overline{E^{(j)}_{\gamma}})
                \]
                
                Furthermore, we remark that for any $\tau_1$ and $\tau_2\in\Delta'_{\spe}\cap {}_f\Delta'$, any $1\leq j_1\leq r_{\tau_1}$, and any $1\leq j_2\leq r_{\tau_2}$, if we have $\overline{E^{(j_1)}_{\tau_1}} = \overline{E^{(j_2)}_{\tau_2}}$, then $\tau_1 = \tau_2$ and $j_1 = j_2$. 
                Indeed, because $E^{(j_1)}_{\tau_1} \cap \overline{E^{(j_2)}_{\tau_2}} \neq\emptyset$, we have $W_{\tau_1}\cap\overline{W_{\tau_2}}\neq\emptyset$. 
                Thus, from \cite[Corollary 3.8]{Y23}, $\tau_2\subset \tau_1$. 
                Similarly, $\tau_1\subset \tau_2$ too. 
                From (a), we have $j_1 = j_2$. 
                
                Let $\gamma\in \Delta'_{\spe}\cap {}_f\Delta'$ be a cone such $\dim(\gamma) = 1$ for any $1\leq j\leq r_\gamma$. 
                Then $\dim(\overline{E^{(j)}_{\gamma}}) = \dim(Y) - 1$ from (a) for any $1\leq j\leq r_\gamma$.  
                Thus, from the above remark, the following set is equal to $\mathfrak{I}$:
                \[
                    \coprod_{\substack{\gamma\in\Delta'_{\spe}\cap {}_f\Delta'\\ \dim(\gamma) = 1}}\biggl\{\overline{E^{(j)}_{\gamma}}\biggr\}_{1\leq j\leq r_{\gamma}}
                \]
                \item[(d)] Let $1\leq i'\leq r$ be an integer and $E\in\mathcal{S}(\HH_{i'})$ be a stratum. 
                From \cite[Proposition 7.7(j)]{Y23}, $\HH_k = \coprod_{1\leq i\leq r} (\HH_i)_k$. 
                Thus, all irreducible components of $(\HH_i)_k$ are in $\mathfrak{I}$. 
                Hence, from (c), 
                there exists $\gamma_1, \ldots, \gamma_m\in \Delta'_{\spe}\cap {}_f\Delta'$ and integers $\{j_l\}_{1\leq l\leq m}$, such that $\dim(\gamma_l) = 1$, $1\leq j_l\leq r_{\gamma_l}$ for any $1\leq l\leq m$, and $E$ is a connected component of the following closed subset:
                \[
                    \bigcap_{1\leq l\leq m} \overline{E^{(j_l)}_{\gamma_l}}
                \]
                In particular, from (a), $\cap_{1\leq l\leq m} \overline{W_{\gamma_l}} \neq \emptyset$. 
                Thus, from \cite[Proposition 3.4(a)]{Y23}, there exists $\sigma\in\Delta'$ such that $\gamma_l\preceq\sigma$ for any $1\leq j\leq m$. 
                Hence, from the assumption, there exists $\tau\in\Delta'_{\bdd}$ such that $\gamma_l\preceq\tau$ for any $1\leq l\leq m$. 
                We can take $\tau$ minimal because $\Delta'$ is a fan. 
                Then the following equation holds from the assumption of $\tau$ :
                \[
                    \bigcap_{1\leq l\leq m} \overline{W_{\gamma_l}} = \overline{W_{\tau}}
                \]
                Thus, from (a), $\HH_k\cap \overline{W_{\tau}}\neq \emptyset$. 
                In particular, from \cite[Proposition 3.4(a)]{Y23} and \cite[Proposition 6.11(d)]{Y23}, $\tau\in{}_f\Delta'$. 
                Moreover, the following equation holds:
                \begin{align*}
                    \HH_k\cap \overline{W_{\tau}}&= \bigcap_{1\leq l\leq m} (\HH_k\cap\overline{W_{\gamma_l}})\\
                    &= \coprod_{\{\{h_l\}_{1\leq l\leq m}\mid1\leq h_l\leq r_{\gamma_l}\}} (\bigcap_{1\leq l\leq m} \overline{E^{(h_l)}_{\gamma_l}})
                \end{align*}
                Thus, $E$ is also a connected component of $\HH_k\cap \overline{W_{\tau}}$. 
                Therefore, from (a), there exists $1\leq j\leq r_{\tau}$ such that $E = \overline{E^{(j)}_{\tau}}$. 
                
                We already have shown the uniqueness in the proof of (c)
                \item[(e)] Let $\tau\in\Delta'_{\bdd}\cap{}_f\Delta'$ be a cone and $1\leq j\leq r_{\tau}$ be an integer. 
                Let $\gamma_1, \ldots, \gamma_m$ be all rays of $\tau$. 
                Because $\tau\in\Delta'_{\bdd}$, we have $\gamma_1, \ldots, \gamma_m\in\Delta'_{\bdd}$. 
                Moreover, from \cite[Proposition 6.11(d)]{Y23}, $\gamma_1, \ldots, \gamma_m\in{}_f\Delta'$. 
                Because $\gamma_1, \ldots, \gamma_m$ are all rays of $\tau$, the following equation holds:
                \[
                    \bigcap_{1\leq l\leq m} \overline{W_{\gamma_l}} 
                    = \overline{W_{\tau}}
                \]
                Thus, like as the argument in proof of (d), the following equation holds:
                \begin{align*}
                    \HH_k\cap \overline{W_{\tau}}&= \bigcap_{1\leq l\leq m} (\HH_k\cap\overline{W_{\gamma_l}})\\
                    &= \coprod_{\{\{h_l\}_{1\leq l\leq m}\mid 1\leq h_l\leq r_{\gamma_l}\}} (\bigcap_{1\leq l\leq m}\overline{E^{(h_l)}_{\gamma_l}})
                \end{align*}
                Hence, there exists integers $(j_l)_{1\leq l\leq m}$ such that $1\leq j_l\leq r_{\gamma_l}$ for any  $1\leq l\leq m$ and $\overline{E^{(j)}_{\tau}}$ is a connected component of the following closed subset:
                \[
                    \bigcap_{1\leq l\leq m} \overline{E^{(j_l)}_{\gamma_l}}
                \]
                From (m), for any $1\leq l\leq m$, $\overline{E^{(j_l)}_{\gamma_l}}$ is an irreducible component of $\HH_k$. 
                Now, $\overline{E^{(j)}_{\tau}}$ is non-empty, so there unique exists $1\leq i\leq r$ such that each $\overline{E^{(j_l)}_{\gamma_l}}$ is an irreducible componenst of $(\HH_i)_k$ from \cite[Proposition 7.7(j)]{Y23}. 
                This shows that $\overline{E^{(j)}_{\tau}}\in\mathcal{S}(\HH_i)$. 
                \item[(f)] We can regard $W$ as a mock toric variety induced by $W'$ along $X(\Delta')\rightarrow X(\Delta'')$. 
                Thus, from \cite[Proposition 6.14(e)]{Y23}, $f$ is non-degenerate for $\Delta''$. 
                
                From (d), (e), and the remark in the proof of (c), we can check that the following equation holds:
                \[
                    \coprod_{1\leq i\leq r}\mathcal{S}(\HH_i) = \coprod_{\tau\in\Delta'_{\bdd}\cap {}_f\Delta'}\{\overline{E^{(j)}_{\tau}}\}_{1\leq j\leq r_{\tau}}
                \]
                We remark that $\{H_{Y_\Kt, f_i}\}_{\mathrm{sb}} = \{H^\circ_{Y_\Kt, f_i}\}_{\mathrm{sb}}$ for any $1\leq i\leq r$. 
                Thus, from (a), (b), \cite[Proposition 7.7(k)]{Y23}, and Proposition \ref{prop: NONO21}, the following equation holds:
                \[
                    \sum_{1\leq i\leq r}\VOL\biggl(\bigl\{H^\circ_{Y_{\Kt}, f_i}\bigr\}_{\mathrm{sb}}\biggr) = \sum_{\substack{\tau\in\Delta'_{\bdd}\\ 
                    \tau\in{}_f\Delta'}}(-1)^{\dim(\tau)-1}\biggl(\sum_{1\leq j\leq r_{\tau}} \bigl\{E^{(j)}_{\tau}\bigr\}_{\mathrm{sb}}\biggr)
                \]
                
                Let $\tau\in\Delta'_{\spe}\cap{}_f\Delta'$ be a cone.  
                Then there exists unique $\sigma\in\Delta''$ such that $\tau^\circ\subset\sigma^\circ$ because $\Delta'$ is a refinement of $\Delta''$. 
                Because $\tau\in\Delta'_{\spe}$, we have $\sigma\in\Delta''_{\spe}$. 
                Moreover, from the proof of \cite[Proposition 6.14(e)]{Y23}, $H_{W, f}\cap W_{\tau}$ is a trivial algebraic torus fibration of $H_{W', f}\cap W'_\sigma$ whose dimension of the fiber is $\dim(\sigma) -\dim(\tau)$ from (a). 
                In particular, $\sigma\in{}_f\Delta''$. 
                Moreover, $r_{\tau} = r_\sigma$ and if it is necessary, we can replace the index such that $E^{(j)}_{\tau}$ is a trivial algebraic torus fibration of $E^{(j)}_{\sigma}$ for any $1\leq j\leq r_\sigma$. 

                Conversely, for any $\sigma\in\Delta''_{\spe}\cap{}_f\Delta''$ and for $\tau\in\Delta'$ such that $\tau^\circ\subset\sigma^\circ$, we have $\tau\in\Delta'_{\spe}\cap{}_f\Delta'$ from Proposition  \cite[Proposition 6.14(e)]{Y23}. 

                Thus, the following equation holds:
                \[
                    \sum_{1\leq i\leq r}\rho\biggl(\bigl\{H^\circ_{Y_{\Kt}, f_i}\bigr\}_
                    {\mathrm{sb}}\biggr) = \sum_{\substack{\sigma\in\Delta''_{\spe}\\ 
                    \sigma\in{}_f\Delta''}}\biggl(\sum_{\substack{\tau\in\Delta'_{\bdd}\\
                    \tau^\circ\subset\sigma^\circ}}(-1)^{\dim(\tau)-1}\biggr)\biggl(\sum_{1\leq j\leq r_{\sigma}} \bigl\{E^{(j)}_{\sigma}\bigr\}_{\mathrm{sb}}\biggr)
                \]
                
                Now, we use the notation in \cite[$\S3$]{NPS16}. 
                Let $\sigma\in\Delta''_{\spe}$ be a cone. 
                We can identify with $N'_\Q$ and $N'_\Q\times\{1\}$, so we can regard $\sigma\cap N'_\Q\times\{1\}$ as a $\Q$-rational polyhedron in $N'_\Q$. 
                Similarly, for any $\tau\in\Delta'$ with $\tau^\circ\subset\sigma^\circ$, we can regard $\tau\cap N'_\Q\times\{1\}$ as a $\Q$-rational polyhedron in $N'_\Q$. 
                Moreover, there exists the following disjoint decomposition of $\Q$-rational polyhedrons because $\Delta'$ is a refinement of $\Delta''$:
                \[
                    \sigma^\circ\cap(N'_\Q\times\{1\}) = \coprod_{\substack{\tau\in\Delta'\\
                    \tau^\circ\subset\sigma^\circ}}\tau^\circ\cap(N'_\Q\times\{1\})
                \]
                There exists Euler characteristic $\chi'$ of definable subsets in $N'_\Q$ from \cite[9.6]{HK06}. 
                From \cite[9.6]{HK06}, $\chi'$ has the additivity, so the following equation holds:
                \[
                    \chi'(\sigma^\circ\cap(N'_\Q\times\{1\})) = \sum_{\substack{\tau\in\Delta'\\
                    \tau^\circ\subset\sigma^\circ}}\chi'(\tau^\circ\cap(N'_\Q\times\{1\}))
                \]
                Furthermore, for any $\tau\in\Delta'_{\spe}$, the following equation holds from \cite[3.7]{NPS16}: 
                \begin{equation*}
                    \chi'(\tau^\circ\cap(N'_\Q\times\{1\}))=
                        \begin{cases}
                            (-1)^{\dim(\tau) -1}& \tau\in\Delta'_{\bdd} \\
                            0                   & \tau\notin\Delta'_{\bdd}\\
                        \end{cases}
                \end{equation*}
                Thus, for any $\sigma\in\Delta''_{\spe}\cap{}_f\Delta''$, there exists the following equation:
                \begin{align*}
                    \sum_{\substack{\tau\in\Delta'_{\bdd}\\
                    \tau^\circ\subset\sigma^\circ}}(-1)^{\dim(\tau)-1}&= \sum_{\substack{\tau\in\Delta'_{\bdd}\\
                    \tau'^\circ\subset\sigma^\circ}}\chi'(\tau^\circ\cap(N'_\Q\times\{1\}))\\
                    &= \sum_{\substack{\tau\in\Delta'_{\spe}\\
                    \tau'^\circ\subset\sigma^\circ}}\chi'(\tau^\circ\cap(N'_\Q\times\{1\}))\\
                    &= \chi'(\sigma^\circ\cap(N'_\Q\times\{1\}))
                \end{align*}
                Similarly, for any $\sigma\in\Delta''_{\spe}$, the following equation holds from \cite[3.7]{NPS16}:
                \begin{equation*}
                    \chi'(\sigma^\circ\cap(N'_\Q\times\{1\}))=
                        \begin{cases}
                            (-1)^{\dim(\sigma) -1} & \sigma\in\Delta''_{\bdd} \\
                            0                     & \sigma\notin\Delta''_{\bdd}\\
                        \end{cases}
                \end{equation*}
                Thus, the following equation holds:
                \[
                    \sum_{1\leq i\leq r}\VOL\biggl(\bigl\{H^\circ_{Y_{\Kt}, f_i}\bigr\}_{\mathrm{sb}}\biggr) = \sum_{\substack{\sigma\in\Delta''_{\bdd}\\ 
                    \sigma\in{}_f\Delta''}}(-1)^{\dim(\sigma)-1}\biggl(\sum_{1\leq j\leq r_{\sigma}} \bigl\{E^{(j)}_{\sigma}\bigr\}_{\mathrm{sb}}\biggr)
                \]
        		\end{enumerate}
        \end{proof}
        In \cite[Proposition 7.7]{Y23} and Proposition\ref{prop: neo irreducible computation}, we assume that $f$ is irreducible. 
        In the following theorem, we extract and summarize the conditions under which the above proposition holds.
        \begin{theorem}\label{thm: general computation}
            We keep the notation in \cite[Proposition 7.7]{Y23}. 
            Let $\Delta''$ be a strongly convex rational polyhedral fan in $(N'\oplus\ZZ)_\RR$ such that $\pi^1$ is compatible with the fans $\Delta''$ and $\Delta\times\Delta_!$, and $\supp(\Delta'') = (\pi^1_\RR)^{-1}(\supp(\Delta\times\Delta_!))$. 
            we assume that
            \begin{enumerate}
                \item A fan $\Delta'$ is a refinement of $\Delta''$ and $f$ is fine for $\Delta''$
                \item A function $f$ is nondegenerate for $\Delta'$. 
                \item A fan $\Delta'$ is compactly arranged, generically unimodular, and specifically reduced. 
                \item Every irreducible component of $H^\circ_{W, f}$ dominates $\Gm^1$. 
                \item A scheme $Z$ is proper over $k$.
            \end{enumerate}
            Let $\{F_\lambda\}_{\lambda\in\Lambda}$ be all connected components of $\HH^\circ$, $W'$ denote a mock toric variety induced by $Z^1$ along $\pi^1_*\colon X(\Delta'')\rightarrow X(\Delta\times\Delta_!)$ and $s^1$, $\{E^{(j)}_{\sigma}\}_{1\leq j\leq r_\sigma}$ denote all connected components of $H_{W', f}\cap W'_\sigma$ for $\sigma\in\Delta''$. 
            Then the following equation holds:
            \[
                \sum_{\lambda\in\Lambda}\VOL\biggl(\bigl\{F_\lambda\bigr\}_{\mathrm{sb}}\biggr) = \sum_{\substack{\sigma\in\Delta''_{\bdd}\\ 
                \sigma\in{}_f\Delta''}}(-1)^{\dim(\sigma)-1}\biggl(\sum_{1\leq j\leq r_\sigma} \bigl\{E^{(j)}_{\sigma}\bigr\}_{\mathrm{sb}}\biggr)
            \]
        \end{theorem}
        \begin{proof}
            From \cite[Proposition 6.14(e)]{Y23}, $f$ is non-degenerate for $\Delta''$. 
            Let $f = h_1\cdots h_n$ be an irreducible decomposition of $f\in k[W^\circ]$. 
            Because $f$ is non-degenerate for $\Delta'$, $h_b$ is non-degenerate for $\Delta'$ too from \cite[Proposition 6.12(b)]{Y23} for any $1\leq b\leq n$. 
            Moreover, from \cite[Proposition 6.12(c)]{Y23}, there exist a following decomposition of $H_{W, f}$:
            \[
                H_{W, f} = \coprod_{1\leq b\leq n} H_{W, h_b}
            \]
            In particular, $\HH^\circ$ can be decomposed as follows:
            \[
                \coprod_{1\leq b\leq n} (H^\circ_{W, h_b}\times_{\Gm}\Spec(\Kt))
            \]
            We remark that $H^\circ_{W, f}$ is smooth over $k$. 
            In particular, it is reduced. 
            Thus, for any $1\leq b\leq n$, $H^\circ_{W, h_b}$ is reduced. 
            Moreover, from \cite[Proposition 6.12(a)]{Y23}, each $h_b$ is fine for $\Delta''$. 
            Therefore, for each $h_b$, we can compute the summation of the stable birational volume of irreducible components of $H^\circ_{W, h_b}\times_{\Gm^1}\Spec(\Kt)$ for the refinement $\Delta'$ of $\Delta''$ like as Proposition \ref{prop: neo irreducible computation}(f). 
           Moreover, for any $\sigma\in\Delta''_{\spe}$, the following equation holds from the proof of \cite[Proposition 6.12(b)]{Y23}:
            \begin{align*}
                H_{W', f}\cap W'_{\sigma} &= \coprod_{1\leq b\leq n} (H_{W', h_b}\cap W'_{\sigma})
            \end{align*}
            Thus, from the above disjoint decomposition, the statement holds. 
        \end{proof}
\section{Mock polytopes of the mock toric variety}
        Let $Z$ be a scheme over $k$ and ($N$, $\Delta$, $\iota$, $\Phi$, $\{N_\varphi\}_{\varphi\in\Phi}$, $\{\Delta_\varphi\}_{\varphi\in\Phi}$) be a mock toric structure of $Z$. 
        In the previous section, we considered the stable birational volume of hypersurfaces of a mock toric variety. 
        However, it isn't easy to apply Theorem\ref{thm: general computation} because of the following three reasons. 
        \begin{itemize}
            \item To find a strongly convex rational polyhedral fan $\Delta''$ that $f$ is fine for $\Delta''$.
            \item To check $f$ is nondegenerate for $\Delta'$.
            \item To find a refinement $\Delta'$ of $\Delta''$ that $\Delta'$ is compactly arranged, generically unimodular, and specifically reduced. 
        \end{itemize}
        For the first point, Collory\cite[Corollary 6.9]{Y23} might be helpful, but calculating the fan associated with $f$ is extremely difficult. 
        Furthermore, for the second point, if an intersection of a hypersurface and a single stratum is not smooth, then the calculation becomes impossible at that stage. 
        Similarly, for the third point, no fan satisfies "specifically reduced" in general. 

        For the first and second points, we adopt the approach already taken in \cite{NO22}. 
        In more detail, consider the finite set consisting of units in $k[Z^\circ]$ and focus on their linear system $D$. 
        These linear systems correspond to what is referred to as the Newton polytope in the argument of toric varieties. 
        In \cite{NO22}, they constructed $\Delta''$ that a general $f\in D$ is fine for $\Delta''$ from Neton polytope associated with $D$. 
        Moreover, they showed that a general $f\in D$ is non-degenerate for $\Delta''$ from Bertini's theorem. 
        In this section, we consider the applicability of the method outlined above to the context of mock toric varieties. 
        In the latter part of the section, we explain ways of avoiding the third issue, which geometrically corresponds to semistable reduction.
        \vskip\baselineskip
        In this section, we use the notation in \cite[Sections 6 \& 7]{Y23} and the following definition. 
        \begin{definition}\label{def: finite unit set}
            We itemize the notation as follows:
            \begin{itemize}
                \item Let $k$ be an infinite algebraically closed field with $\charac(k) = 0$. 
                \item Let $S$ be a finite set.
                \item Let $\mathcal{U}_{Z, S}$ denote a set which consists of all map from $S$ to $k[Z^\circ]^*$. 
                \item We define a relation $\sim$ of $\mathcal{U}_{Z, S}$ as follows:  
                For $u_1, u_2\in\mathcal{U}_{Z, S}$, if there exists $\chi\in k[Z^\circ]^*$ such that $\chi u_1(i) = u_2(i)$ for any $i\in S$, then we define $u_1\sim u_2$. 
                We can check that this relation $\sim$ is an equivalence relation of $\mathcal{U}_{Z, S}$.
                \item For $u\in\mathcal{U}_{Z, S}$, let $[u]$ denote the equivalence class of $u$ along the relation $\sim$. 
                \item For $u\in\mathcal{U}_{Z, S}$, let $V_u$ denote a $k$-lenear subspace of $k[Z^\circ]$ generated by $\{u(i)\}_{i\in S}$, and $d_u$ denote $\dim_k(V_u)$. 
            \end{itemize}
        \end{definition}
        The following proposition is immediately evident from Definition\ref{def: finite unit set}
        \begin{proposition}\label{prop: dimension of polytope}
            We use the notation in Definition \ref{def: finite unit set}.  
            Let $u_1$ and $u_2\in \mathcal{U}_{Z, S}$ be maps. 
            If $u_1\sim u_2$, then $d_{u_1} = d_{u_2}$. 
        \end{proposition}
        \begin{proof}
            For any $\chi$, the $k$-linear endmorphism of $k[Z^\circ]$ defined by multiplication of $\chi$ is an isomorphism. 
            Thus, $\dim(V_{u_1}) = \dim(V_{u_2})$. 
        \end{proof}
        In the following definition, we define a mock polytope, an analog of a Newton polytope. 
        \begin{definition}\label{def:mock polytope}
            We use the notation in Definition \ref{def: finite unit set}.  
            Let $u\in\mathscr{U}_{Z, S}$ be a map. 
            Then we call the tuple $(Z, S, u)$ a mock polytope of $Z$. 
        \end{definition}
        In Section 6, we construct the map $\val_f\colon \supp(\Delta)\rightarrow \RR$ for $f\in k[Z^\circ]$. 
        From now on, we construct the map $\val_u\colon \supp(\Delta)\rightarrow \RR$ for $u\in \mathcal{U}_{Z, S}$.
        \begin{definition}\label{def:val function of mock polytope}
            Let $(Z, S, u)$ be a mock polytope of $Z$. 
            Let $\val_u$ denote a map from $\supp(\Delta)$ to $\RR$ defined as $\val_u(v) = \min_{i\in S} \{\val_{u(i)}(v)\}$ for $v\in\supp(\Delta)$. 
            For a subset $A\subset\supp(\Delta)$, let $\val^A_u$ denote the restriction $\val_u|_A$. 
        \end{definition}
        We consider the basic properties of $\val_u$ in the following proposition. 
        \begin{proposition}\label{prop: VoPP1}
           We use the notation in Definition \ref{def:val function of mock polytope}. 
           Then the following statement holds:
           \begin{enumerate}
               \item[(a)] The map $\val_u$ is continuous and is preserved by the scalar product of non-negative real numbers. 
               \item[(b)] Let $\sigma\in\Delta$ be a cone and $\varphi\in\Phi$ be an element such that $\sigma\in\Delta_\varphi$. 
               Let $\sigma_\varphi$ denote $(q_\varphi)_\RR(\sigma)$. 
               Then there exists $g_\varphi\in k[M_\varphi]$ such that the following equation holds:
               \[
                    \val^{\sigma}_u = \val^{\sigma_\varphi}_{g_\varphi}\circ(q_\varphi)_\RR|_{\sigma}
               \]
           \end{enumerate}
        \end{proposition}
        \begin{proof}
            We prove the statements from (a) to (b). 
            \begin{enumerate}
                \item[(a)] For each $i\in S$, $\val_{u(i)}$ is a continuous map and is preserved by the scalar product of non-negative real numbers from \cite[Proposition 6.7(a)]{Y23}. 
                Thus, $\val_u$ also has these properties from the definition of $\val_u$. 
                \item[(b)] 
                From \cite[Proposition 6.3]{Y23}, for each $i\in S$, there exists $\omega_i\in M_\varphi$ such that $\val^\sigma_{u(i)} = \val^\sigma_{\iota^*_\varphi(\chi^{\omega_i})}$. 
                Let $T$ denote a subset $\{\omega_i\mid i\in S\}\subset M_\varphi$. 
                Let $g_\varphi$ denote $\sum_{\omega\in T} \chi^{\omega}\in k[M_\varphi]$. 
                Then for any $v\in \sigma\cap N$, we have $q_\varphi(v)(g_\varphi) = \min_{\omega\in T}\{q_\varphi(v)(\chi^\omega)$\} because the definition of the torus invariant valuation on toric varieties. 
                This shows that for any $v\in\sigma\cap N$, there exists the following equation:
                \begin{align*}
                    \val^{\sigma_\varphi}_{g_\varphi}\circ(q_\varphi)_\RR(v)&= q_\varphi(v)(g_\varphi)\\
                    &= \min_{\omega\in T}\{q_\varphi(v)(\chi^\omega)\}\\
                    &= \min_{i\in S}\{q_\varphi(v)(\chi^{\omega_i})\}\\
                    &= \min_{i\in S}\{\val_Z(v)(\iota^*_\varphi(\chi^{\omega_i}))\}\\
                    &= \min_{i\in S}\{\val^\sigma_{\iota^*_\varphi(\chi^{\omega_i})}(v)\}\\
                    &= \min_{i\in S}\{\val^\sigma_{u(i)}(v)\}\\
                    &= \val^\sigma_u(v)
                \end{align*}
                Thus, from (a) and \cite[Lemma 8.13]{Y23}, the statement holds.
            \end{enumerate}
        \end{proof}
        Similar to \cite[Proposition 6.8]{Y23}, $\val_u$ provides a refinement of $\Delta$.
        We define the notation in Proposition \ref{def: refinement by mock polytope} and prove that the set of cones is a refinement of $\Delta$ in Proposition \ref{prop: VoPP2}.  
        \begin{definition}\label{def: refinement by mock polytope}
            Let $(Z, S, u)$ be a mock polytope and $\sigma\in\Delta$ be a cone. 
            \begin{itemize}
                \item Let $\Omega^\sigma_u$ be a set defined as follows:
                \[
                    \Omega^\sigma_u = \{(\chi, b)\in k[Z^\circ]^*\times \ZZ_{>0}\mid\val_u(v)\leq\frac{1}{b}\val_\chi(v)\quad(\forall v\in \sigma)\}
                \]
                \item For $(\chi, b)\in \Omega^\sigma_u$, let $C^\sigma_u(\chi, b)$ be a set defined as follows:
                \[
                    C^\sigma_u(\chi, b) = \{v\in\sigma\mid\val_u(v)=\frac{1}{b}\val_\chi(v)\}
                \]
            \end{itemize}
        \end{definition}
        \begin{proposition}\label{prop: VoPP2}
            We keep the notation in Definition \ref{def: refinement by mock polytope}.  
            Then the following statements follow:
            \begin{enumerate}
                \item[(a)] The set $\{C^\sigma_u(\chi, b)\}_{(\chi, b)\in\Omega^\sigma_u}$ is a strongly convex rational polyhedral fan and a refinement of $\sigma$. 
                \item[(b)] Let $\sigma$ and $\tau\in\Delta$ be cones  such that $\tau\preceq\sigma$. 
                Then $\Omega^\sigma_u\subset\Omega^\tau_u$. 
                Moreover, $C^\sigma_u(\chi, b)\cap\tau = C^\tau_u(\chi, b)$ for any $(\chi, b)\in\Omega^\sigma_u$. 
                In particular, $C^\tau_u(\chi, b)\preceq C^\sigma_u(\chi, b)$.  
                \item[(c)] The following set $\Delta_u$ is a strongly convex rational polyhedral fan in $N_\RR$ and a refinement of $\Delta$:
                \[
                    \Delta_u = \bigcup_{\sigma\in\Delta}\{C^\sigma_u(\chi, b)\}_{(\chi, b)\in\Omega^\sigma_u}
                \]
                \item[(d)] For any $(\chi, b)\in\Omega^\sigma_u$, there exists $\chi'\in k[Z^\circ]^*$ such that $\val_u(v) = \val_{\chi'}(v)$ for any $v\in C^\sigma_u(\chi, b)$. 
                \item[(e)] Let $u_1$ and $u_2\in \mathscr{U}_{Z, S}$ be maps. 
                If $u_1\sim u_2$, then $\Delta_{u_1} = \Delta_{u_2}$.
            \end{enumerate}
        \end{proposition}
        \begin{proof}
            We prove the statements from (a) to (e) in order. 
            \begin{enumerate}
                \item[(a)] From Proposition \ref{prop: VoPP1}(b), there exists $f\in k[Z^\circ]$ such that $\val^\sigma_u = \val^\sigma_f$. 
                This shows that $\Omega^\sigma_u = \Omega^\sigma_f$ and for any $(\chi, b)\in \Omega^\sigma_u$, we have $C^\sigma_u(\chi, b) = C^\sigma_f(\chi, b)$. 
                Thus, from \cite[Proposition 6.8(c)]{Y23}, the statement holds. 
                \item[(b)] We can check the statement likewise \cite[Proposition 6.8(d)]{Y23}. 
                \item[(c)] We can check the statement likewise \cite[Proposition 6.8(e)]{Y23}. 
                \item[(d)] We use the notation in the proof of (a).  
                Then $(\chi, b)\in \Omega^\sigma_f$ and $C^\sigma_u(\chi, b) = C^\sigma_f(\chi, b)$. 
                Thus, from \cite[Proposition 6.8(f)]{Y23}, there exists $\chi'\in k[Z^\circ]^*$ such that $\val_f(v) = \val_{\chi'}(v)$ for any $v\in C^\sigma_f(\chi, b)$. 
                Therefore, the statement holds.
                \item[(e)] Let $\eta\in k[Z^\circ]^*$ be a unit such that $u'(i) = \eta u(i)$ for any $i\in S$. 
                Then for any $i\in S$, $\val_{u'(i)} = \val_{\eta} + \val_{u(i)}$ from \cite[Proposition 6.7(b)]{Y23}. 
                This shows that $\val_{u'} = \val_\eta + \val_u$. 
                Let $\sigma\in\Delta$ be a cone. 
                Thus, for any $(\chi, b)\in k[Z^\circ]^*\times \ZZ_{>0}$, we can check that $(\chi, b)\in\Omega^\sigma_u$ if and only if $(\chi\eta^b, b)\in\Omega^\sigma_{u'}$. 
                Moreover, in this case, $C^\sigma_{u'}(\chi\eta^b, b) = C^\sigma_u(\chi, b)$. 
                Therefore, the set $\{C^\sigma_u(\chi, b)\}_{(\chi, b)\in\Omega^\sigma_u}$ is equal to the set $\{C^\sigma_{u'}(\chi', b')\}_{(\chi', b')\in\Omega^\sigma_{u'}}$. 
            \end{enumerate}
        \end{proof}
        From $u\in \mathscr{U}_{Z, S}$, we can define the linear system $D\subset \Gamma(Z^\circ, \OO_Z)$. 
        In Proposition\ref{prop: Bertini for fine}, we show that $\Delta_f = \Delta_u$ for a general $f\in D$. 
        The following lemma is needed to show Proposition\ref{prop: Bertini for fine}. 
        \begin{lemma}\label{lem: preparation for Fine}
            Let $(Z, S, u)$ be a mock polytope. 
            Let $(a_i)_{i\in S}\in k^{|S|}$ be a vector. 
            Let $f\in k[Z^\circ]$ be a function defined as $f = \sum_{i\in S}a_iu(i)$. 
            Then the following statements follow:
            \begin{enumerate}
                \item[(a)] Let $v\in \supp(\Delta)\cap N$ be an element.  
                Then for general $(a_i)_{i\in S}\in k^{|S|}$, we have $\val_f(v) = \val_u(v)$. 
                \item[(b)] Let $\tau\in\Delta_u$ be a cone and $v_0\in\tau^\circ$ be an element. 
                For any $(a_i)_{i\in S}\in k^{|S|}$, if $\val_u(v_0) = \val_f(v_0)$, then $\val^\tau_u = \val^\tau_f$. 
            \end{enumerate}
        \end{lemma}
        \begin{proof}
            We prove the statements from (a) to (b).
            \begin{enumerate}
                \item[(a)] For $n\in\ZZ$, let $E(v,n)$ denote a subset of $k^{|S|}$ defined as follows:
                \[
                    E(v, n) = \{(a_i)_{i\in S}\in k^{|S|}\mid\val_Z(v)(\sum_{i\in S}a_iu(i))\geq n\}
                \]
                Because $\val_Z(v)$ is a valuation which is trivial on $k$, $E(v, n)$ is a linear subspace of $k^{|S|}$. 
                In particular, $E(v, \val_u(v)) = k^{|S|}$ because of the property of a valuation and the definition of $\val_u$. 

                On the other hand, from the definition of $\val_u$, there exists $j\in S$ such that $\val_u(v) = \val_{u(j)}(v)$. 
                Let $e_j$ denotes $(\delta_{i, j})_{i\in S}\in k^{|S|}$, where $\delta_{\cdot, \cdot}$ is Kronecker's delta. 
                Thus, $e_j\notin E(v, \val_u(v) + 1)$. 
                
                Therefore, $E(v, \val_u(v)+1)$ is a strict linear sub space of $k^{|S|}$. 
                In particular, for general $(a_i)_{i\in S}\in k^{|S|}$, $\val_f(v) = \val_u(v)$.
                \item[(b)] From the definition of $\Delta_u$, there exists $\sigma\in\Delta$ and $(\chi, b)\in \Omega^\sigma_u$ such that $\tau = C^\sigma_u(\chi, b)$.
                From Proposition \ref{prop: VoPP2}(d), there exists $\eta\in k[Z^\circ]^*$ such that $\val^\tau_u = \val^\tau_\eta$. 
                In particular, $\val^\tau_u$ is a linear map over $\tau$ from \cite[Proposition 6.7(d)]{Y23}. 
                For any $v\in\supp(\Delta)\cap N$, we have $\val_u(v)\leq \val_f(v)$ because of the definition of $\val_u$ and the property of valuations. 
                Both functions $\val_u$ and $\val_f$ are continuous and are preserved by the scalar product of non-negative real numbers from \cite[Proposition 6.7(a)]{Y23} and Proposition \ref{prop: VoPP1}(a). 
                Thus $\val^\tau_u\leq \val^\tau_f$ because $\tau$ is a rational polyhedral cone. 
                Moreover, from \cite[Proposition 6.7(c)]{Y23}, $\val^\sigma_f$ is an upper convex function. 
                In particular, $\val^\tau_f$ is an upper convex function because $\tau\subset\sigma$. 
                Thus, $\val^\tau_f - \val^\tau_u$ is a non-negative valued upper convex function. 
                On the other hand, from the assumption, $\val^\tau_u(v_0) = \val^\tau_f(v_0)$. 
                Therefore, $\val^\tau_u = \val^\tau_f$ because $v_0\in\tau^\circ$.  
            \end{enumerate}
        \end{proof}
        \begin{proposition}\label{prop: Bertini for fine}
             We keep the notation in Lemma \ref{lem: preparation for Fine}.  
             Then $\val_u = \val_f$ and $\Delta_f = \Delta_u$ for general $(a_i)_{i\in S}\in k^{|S|}$.
        \end{proposition}
        \begin{proof}
            If $\val_u = \val_f$, then $\Delta_u = \Delta_f$ from the definition of $\Delta_u$ and $\Delta_f$. 
            Thus, it is enough to show that $\val_u = \val_f$ for general $(a_i)_{i\in S}\in k^{|S|}$. 
            Because $|\Delta|$ is finite, $|\Delta_u|$ is also finite from the proof of Proposition \ref{prop: VoPP2}(a). 
            Thus, it is enough to show that for any $\tau\in\Delta_u$, we have $\val^\tau_u = \val^\tau_f$ for general $(a_i)_{i\in S}\in k^{|S|}$. 
            Let $\tau\in\Delta_u$ be a cone, and $v_0\in \tau^\circ\cap N$ be an element. 
            Then from Lemma \ref{lem: preparation for Fine}(a), $\val_f(v_0) = \val_u(v_0)$ for general $(a_i)_{i\in S}\in k^{|S|}$. 
            Thus, from Lemma \ref{lem: preparation for Fine}(b), $\val^\tau_u = \val^\tau_f$ for such $(a_i)_{i\in S}\in k^{|S|}$. 
        \end{proof}
        From $u\in \mathscr{U}_{Z, S}$, we construct a refinement $\Delta_u$ of $\Delta$. 
        Let $Z_u$ denote the mock toric variety induced by $Z$ along $X(\Delta_u)\rightarrow X(\Delta)$. 
        In Definition\ref{def: polytope for orbit}, we define new mock polytopes for each $\tau\in\Delta_u$. 
        They are important to check the structure of the intersection of a hypersurface and a stratum of $Z_u$. 
        \begin{definition}\label{def: polytope for orbit}
            Let $(Z, S, u)$ be a mock polytope, $Z_u$ be a mock toric variety induced by $Z$ along $X(\Delta_u)\rightarrow X(\Delta)$, $\tau\in\Delta_u$ be a cone, $Z^\tau_u$ denote a mock toric variety $\overline{(Z_u)_\tau}$, $p^\tau$ denote a ring morphism $\Gamma(Z_u(\tau), \OO_{Z_u})\rightarrow \Gamma((Z_u)_{\tau}, \OO_{Z^\tau_u})$ associated with the closed immersion $(Z_u)_\tau\hookrightarrow Z_u(\tau)$, and $S^\tau$ denote a subset $\{i\in S\mid val^\tau_u = \val^\tau_{u(i)}\}$ of $S$. 
            
            From Proposition \ref{prop: VoPP2} (d), there exists $\chi\in k[Z^\circ]^*$ such that $\val^\tau_u = \val^\tau_\chi$. 
            Let $u^\tau_\chi$ denote a map from $S^\tau$ to $k[(Z^\tau_u)^\circ]$ defined as $u^\tau_\chi(i) = p^\tau(\chi^{-1}u(i))$ for $i\in S^\tau$. 
        \end{definition}
        The following proposition is immediately evident from Definition\ref{def: polytope for orbit}. 
        \begin{proposition}\label{prop: well-def of polytope}
            We keep the notation in Definition \ref{def: polytope for orbit}. Then the following statements follow:
            \begin{enumerate}
                \item[(a)] The tuple $(Z^\tau_u, S^\tau, u^\tau_\chi)$ is a mock polytope of $Z^\tau_u$. 
                \item[(b)] For any $\chi'\in k[Z^\circ]^*$ such that $\val^\tau_u = \val^\tau_{\chi'}$, we have $u^\tau_\chi\sim u^\tau_{\chi'}$.
                \item[(c)] Let $u'\in\mathscr{U}_{Z, S}$ be a map such that $u\sim u'$. 
                From Proposition \ref{prop: VoPP2}(e), $\Delta_u = \Delta_{u'}$, in particular, $\tau\in\Delta_{u'}$. 
                Then $\{i\in S\mid \val^\tau_{u'} = \val^\tau_{u'(i)}\}$ is equal to 
                $S^\tau$. 
                \item[(d)] We use the notation in (c). 
                Let $\chi'\in k[Z_0]^*$ be a unit such that $\val^\tau_{u'} = \val^\tau_{\chi'}$. 
                Then $u^\tau_\chi\sim u'^\tau_{\chi'}$. 
            \end{enumerate}
        \end{proposition}
        \begin{proof}
            we prove the statements from (a) to (d) in order. 
            \begin{enumerate}
                \item[(a)] 
                It is enough to show that for any $i\in S^\tau$, we have $u^\tau_\chi(i)\in k[(Z^\tau_u)^\circ]^*$. 
                Let $i\in S^\tau$ be an element. 
                From the definition of $S^\tau$, we have $\val^\tau_{\chi^{-1}u(i)} \equiv 0$. 
                Thus, from \cite[Proposition 4.8]{Y23}, $\chi^{-1}u(i)\in \Gamma(Z_{u}(\tau), \OO_{Z_u})^*$. 
                Hence, $p^\tau(\chi^{-1}u(i))\in k[(Z^\tau_u)^\circ]^*$ too. 
                \item[(b)] From the assumption of $\chi$ and $\chi'$, $\val^\tau_{\chi'^{-1}\chi} \equiv 0$. 
                As the argument in the proof of (a), $\chi'^{-1}\chi\in \Gamma(Z_{u}(\tau), \OO_{Z_u})^*$. 
                Let $\eta$ denote $p^\tau(\chi'^{-1}\chi)\in k[(Z^\tau_u)^\circ]^*$. 
                Then for any $i\in S^\tau$, we have $\eta u^\tau_\chi(i) = u^\tau_{\chi'}(i)$. 
                This shows that $u^\tau_\chi \sim u^\tau_{\chi'}$. 
                \item[(c)] Let $\chi_0 \in k[Z^\circ]^*$ be a unit such that $\chi_0u(i) = u'(i)$ for any $i\in S$. 
                Then $\val_{u(i)} + \val_{\chi_0} = \val_{u'(i)}$ from \cite[Proposition 6.7(b)]{Y23}.
                Thus, $\val_{u} + \val_{\chi_0} = \val_{u'}$ from the definition of $\val_u$ and $\val_{u'}$. 
                Therefore, for any $i\in S$, $\val^\tau_u = \val^\tau_{u(i)}$ if and only if $\val^\tau_{u'} = \val^\tau_{u'(i)}$. 
                This shows that $\{i\in S\mid \val^\tau_{u'} = \val^\tau_{u'(i)}\}$ is equal to $S^\tau$. 
                \item[(d)] Let $\chi_0 \in k[Z^\circ]^*$ be a unit such that $\chi_0u(i) = u'(i)$ for any $i\in S$. 
                In the proof of (c), we have already shown that $\val_{u} + \val_{\chi_0} = \val_{u'}$. 
                Thus, $\val^\tau_{\chi\chi_0} = \val^\tau_{u'}$ from \cite[Proposition 6.7(b)]{Y23}. 
                Therefore from (b), $u'^\tau_{\chi\chi_0}\sim u'^\tau_{\chi'}$. 
                On the other hand, for any $i\in S$, $(\chi\chi_0)^{-1}u'(i) = \chi^{-1}u(i)$. 
                Thus, $u^\tau_{\chi} = u'^\tau_{\chi\chi_0}$. 
                This shows that $u^\tau_{\chi}\sim u'^\tau_{\chi'}$. 
            \end{enumerate}
        \end{proof}
        In Definition\ref{def: polytope for orbit}, we choose $\chi$ to define $S^\tau$ for each $\tau\in\Delta_u$. 
        However, in Proposition\ref{prop: well-def of polytope}, we showed that the definition of $S^\tau$ is independent of the choice of $\chi$. 
        Moreover, we showed that $u^\tau_\chi$ is the same up to the equivalence class. 
        Now, we define a new notation for the previously established notation, which takes into account the equivalent classes. 
        In practice, the shifts by multiplication of units are irrelevant when considering the linear system. 
        \begin{definition}\label{def: decomposition of mock polytope}
            Let $Z$ be a mock toric variety, $S$ be a nonempty finite set, and $[u]$ be an equivalent class of $\mathscr{U}_{Z, S}$. 
            \begin{itemize}
                \item 
                We call the tuple $(Z, S, [u])$ an abstract mock polytope of $Z$. 
                \item Let $\Delta_{[u]}$ denote $\Delta_u$. 
                From Proposition \ref{prop: VoPP2}(e), it is well-defined. 
                Let $Z_{[u]}$ denote a mock toric variety $Z_u$. 
                \item  
                Let $d_{[u]}$ denote $d_u$. 
                From Proposition \ref{prop: dimension of polytope}, it is well-defined. 
                \item We use the notation in Definition \ref{def: polytope for orbit}. 
                Let $\tau\in\Delta_u$ be a cone. 
                We call $(Z^\tau_u, S^\tau, [u^\tau_\chi])$ an abstract mock toric polytope of $Z^\tau_u$ induced by $(Z, S, [u])$ associated with $\tau$. 
                The definition is independent of the choice of $u\in[u]$ and $\chi\in k[Z^\circ]^*$ from Proposition \ref{prop: well-def of polytope}(b), (c), and (d). 
                Let $[u^\tau]$ denote $[u^\tau_\chi]$.  
                \item For an abstract mock polytope $(Z, S, [u])$ of $Z$, we call the following set the decomposition of $(Z, S, [u])$:
                \[
                    \{(Z^\tau_u, S^\tau, [u^\tau])\}_{\tau\in\Delta_{[u]}}
                \]
            \end{itemize}
        \end{definition}
        In the following lemma, we can compute the intersection of a hypersurface and a stratum of $Z_u$ for a general section in the linear system defined by $[u]$.  
        \begin{lemma}\label{lemma: polytope and stratification}
            We keep the notation in Definition \ref{def: decomposition of mock polytope}. 
            Let $(Z, S, [u])$ be an abstract mock polytope of $Z$, $(a_i)_{i\in S}\in k^{|S|}$ be a vector, $\tau\in\Delta_{[u]}$ be a cone, $f$ denote $\sum_{i\in S}a_iu(i)\in k[Z^\circ]$, and $f^\tau$ denote $\sum_{i\in S^\tau}a_iu^\tau(i)\in k[(Z^\tau_u)^\circ]$. 
            We assume that $\val_f = \val_{u}$. 
            Then $H_{Z_u, f}\cap (Z_u)_\tau = H^\circ_{Z^\tau_u, f^\tau}$. 
        \end{lemma}
        \begin{proof}
            From the assumption, $\Delta_{[u]} = \Delta_f$.
            In particular, $f$ is fine for $\Delta_{[u]}$ from  \cite[Corollary 6.9]{Y23}. 
            There exists $\chi\in k[Z^\circ]^*$ such that $\val^\tau_\chi = \val^\tau_u$ from Proposition \ref{prop: VoPP2}(d), in particular, $\val^\tau_{\chi^{-1}f}\equiv 0$. 
            We can assume that $u^\tau(i) = p^\tau(\chi^{-1}u(i))$ for any $i\in S^\tau$ from Proposition \ref{prop: well-def of polytope}(b). 
            From \cite[Proposition 4.8]{Y23}, $\chi^{-1}f\in \Gamma(Z_{u}(\tau), \OO_{Z_{u}})$. 
            Let $g$ denote $p^\tau(\chi^{-1}f)$. 
            Then from \cite[Proposition 6.11(a)]{Y23}, $H_{Z_u, f}\cap (Z_u)_\tau = H^\circ_{Z^\tau_u, g}$. 
            Thus, it is enough to show that for any $i\notin S^\tau$, we have $p^{\tau}(\chi^{-1}u(i)) = 0$. 
            Let $i\in S$ be an element such that $i\notin S^\tau$. 
            Then from Lemma \ref{lem: preparation for Fine} (b), for any $v_0\in\tau^\circ$, we have $\val_{u}(v_0)<\val_{u(i)}(v_0)$. 
            In particular, $\val_{Z}(v_0)(\chi^{-1}u(i))>0$ for any $v_0\in\tau^\circ\cap N$. 
            From \cite[Proposition 4.7(b)]{Y23}, such $\val_Z(v_0)$ is a valuation of a valuation ring over $Z_{[u]}$ which dominates the local ring at the generic point of $(Z_u)_\tau$. 
            This shows that the ideal of $\Gamma(Z_{u}(\tau), \OO_{Z_{u}})$ associated with the closed subscheme $(Z_u)_\tau$ of $Z_{u}(\tau)$ contains 
            $\chi^{-1}u(i)$. 
            In particular, $p^\tau(\chi^{-1}u(i)) = 0$. 
        \end{proof}
        As mentioned at the beginning of this section, the following proposition indicates that a general section in the linear system defined by $[u]$ is non-degenerate for $\Delta_u$.
        \begin{proposition}\label{prop:non-empty and Bertini}
            Let $(Z, S, [u])$ be an abstract mock polytope of $Z$, $(a_i)_{i\in S}\in k^{|S|}$ be a vector, and $f$ denote $\sum_{i\in S}a_iu(i)\in k[Z^\circ]$. 
            Then the following statements follow: 
            \begin{enumerate}
                \item[(a)] If $d_{[u]} = 1$, then $H^\circ_{Z, f} = \emptyset$ for general $(a_i)_{i\in S}\in k^{|S|}$.
                \item[(b)] We assume that $k$ is an uncountable field and $d_{[u]}\geq 2$. 
                Then for general $(a_i)_{i\in S}\in k^{|S|}$, we have $H^\circ_{Z, f} \neq \emptyset$. 
                \item[(c)] Let $\tau\in\Delta_{[u]}$ be a cone. 
                We assume that $k$ is an uncountable field and $d_{[u^\tau]}\geq 2$. 
                Then for general $(a_i)_{i\in S}\in k^{|S|}$, we have $\Delta_f = \Delta_{[u]}$ and $H_{Z_{[u]}, f}\cap Z_{[u], \tau} \neq \emptyset$. 
                \item[(d)] Let $\Delta^{\geq 2}_{[u]}$ denote the following subset of $\Delta_{[u]}$:
                \[
                    \Delta^{\geq 2}_{[u]} = \{\tau\in\Delta_{[u]}\mid d_{[u^\tau]}\geq 2\}
                \]
                Then $\Delta^{\geq 2}_{[u]}$ is a sub fan of $\Delta_{[u]}$.
                \item[(e)] For general $(a_i)_{i\in S}\in k^{|S|}$, $H^\circ_{Z, f}$ is smooth over $k$. 
                \item[(f)] For general $(a_i)_{i\in S}\in k^{|S|}$, $f$ is non-degenerate for $\Delta_{[u]}$. 
            \end{enumerate}
        \end{proposition}
        \begin{proof}
            We prove the statements from (a) to (f) in order. 
            \begin{enumerate}
                \item[(a)] From the assumption, there exists $\chi\in k[Z^\circ]^*$ and $(c_i)_{i\in S}\in (k^*)^{|S|}$ such that $\sum_{i\in S}a_iu(i) = (\sum_{i\in S}a_ic_i)\chi$. 
                Because $\chi$ is a unit of $k[Z^\circ]$, for general $(a_i)_{i\in S}\in k^{|S|}$, we have $H^\circ_{Z_{[u]}, f} = \emptyset$. 
                \item[(b)] 
                Let $h$ denote a linear morphism  from $k^{|S|}$ to $k[Z^\circ]^*$ defined as $h((a_i)_{i\in S}) = \sum_{i\in S}a_iu(i)$ for $(a_i)_{i\in S}\in k^{|S|}$. 
                We remark that for any $\chi\in k[Z^\circ]^*$, $h^{-1}(k\cdot\chi)$ is a strict linear subspace of $k^{|S|}$ from the assumption of $d_{[u]}$. 
                Let $H$ be a closed subscheme of $\A^{|S|}_k\times Z^\circ_{u}$ defined by  $\sum_{i\in S}x_iu(i)\in k[Z^\circ][x_i]$, where $\{x_i\}_{i\in S}$ are the coordinate functions of $\A^{|S|}_k$.  
                Let $\theta\colon H\rightarrow\A^{|S|}_k$ denote the first projection. 
                From Chevalley's theorem, $\theta(H)$ is a constractible set of $\A^{|S|}_k$. 
                We show that $\theta$ is a dominant map by a contradiction. 
                We assume that $\theta$ is not dominant. 
                Then there exists a non-empty open subset $V$ of $\A^{|S|}_k$ such that $V\cap\theta(H) = \emptyset$. 
                This shows that for any $(a_i)_{i\in S}\in V(k)$, we have $\sum_{i\in S}a_iu(i)\in k[Z^\circ]^*$. 
                Thus, from the remark, $V(k)$ is covered by strict linear subspaces $\{h^{-1}(k\cdot\chi)\}_{\chi\in k[Z^\circ]^*}$ of $k^{|S|}$. 
                On the other hand, a group $k[Z^\circ]^*/k^*$ is a free abelian group of finite rank, so $|\{h^{-1}(k\cdot\chi)\}_{\chi\in k[Z^\circ]^*}|$ is countable. 
                However, $V(k)$ can not be covered by countable strict closed subsets of $\A^{|S|}_k$ because $k$ is an uncountable field. 
                It is a contradiction. 
                \item[(c)] 
                Let $f^\tau$ denote $\sum_{i\in S^\tau} a_iu^\tau(i)\in k[(Z^\tau_u)^\circ]$. 
                From Proposition \ref{prop: Bertini for fine} and Lemma\ref{lemma: polytope and stratification}, $H_{Z_u, f}\cap (Z_u)_\tau = H^\circ_{Z^\tau_u, f^\tau}$ for general $(a_i)_{i\in S}\in k^{|S|}$. 
                Let $g$ denote $\sum_{i\in S^\tau}b_iu^\tau(i)\in k[(Z^\tau_u)^\circ]$ for $(b_i)_{i\in S^\tau}\in k^{|S^\tau|}$.  
                Thus, from (b) and the assumption of $d_{[u^\tau]}$, we have $H^\circ_{Z^\tau_u, g} \neq \emptyset$ for general $(b_i)_{i\in S^\tau}\in k^{|S^\tau|}$. 
                Therefore, $H_{Z_u, f}\cap (Z_u)_\tau \neq \emptyset$ for general $(a_i)_{i\in S}\in k^{|S|}$. 
                \item[(d)] Because $|\Delta_u|$ is finite, for general $(a_i)_{i\in S}\in k^{|S|}$, $f$ is fine for $\Delta_u$ and $H_{Z_u, f}\cap (Z_u)_\tau \neq \emptyset$ for all  $\tau\in\Delta_u$ such that $d_{[u^\tau]} \geq 2$ from Proposition \ref{prop: Bertini for fine} and (c). 
                We remark that for any $\tau\in\Delta_u$ such that $d_{[u^\tau]} = 1$, we have $H_{Z_u, f}\cap (Z_u)_\tau = \emptyset$ from the fact that $f$ is fine for $\Delta_u$, the argument in the proof of (a), and Lemma\ref{lemma: polytope and stratification}. 
                Thus, $\Delta^{\geq 2}_{[u]}$ is equal to ${}_f(\Delta_{[u]})$ for general $(a_i)_{i\in S}\in k^{|S|}$. 
                Therefore, from \cite[Proposition 6.11(d)]{Y23}, $\Delta^{\geq 2}_{[u]}$ is a sub fan of $\Delta_{[u]}$. 
                \item[(e)] For any $i\in S$, $u(i)$ is a unit in $k[Z^\circ]$. 
                In particular, $\{u(i)\}_{i\in S}\subset\Gamma(Z^\circ, \OO_{Z})$ has no base locus in $Z^\circ$. 
                Thus, from Bertini's theorem, $H^\circ_{Z, f}$ is smooth over $k$ for general $(a_i)_{i\in S}\in k^{|S|}$. 
                \item[(f)] Let $f^\tau$ denote $\sum_{i\in S^\tau}a_iu^\tau(i)\in k[(Z^\tau_u)^\circ]$. 
                From Proposition \ref{prop: Bertini for fine} and Lemma\ref{lemma: polytope and stratification}, for general $(a_i)_{i\in S}\in k^{|S|}$, we have $H_{Z_u, f}\cap (Z_u)_\tau = H^\circ_{Z^\tau_u, f^\tau}$. 
                Let $g$ denote $\sum_{i\in S^\tau}b_iu^\tau(i)\in k[(Z^\tau_u)^\circ]^*$ for $(b_i)_{i\in S^\tau}\in k^{|S^\tau|}$.  
                Thus, from (e), for general $(b_i)_{i\in S^\tau}\in k^{|S^\tau|}$, $H^\circ_{Z^\tau_u, g}$ is smooth over $k$. 
                Therefore, for general $(a_i)_{i\in S}\in k^{|S|}$, $H_{Z_u, f}\cap (Z_u)_\tau$ is smooth over $k$. 
                Therefore, because $|\Delta_u|$ is finite, for general $(a_i)_{i\in S}\in k^{|S|}$, $f$ is non-degenerate for $\Delta_{[u]}$. 
            \end{enumerate}
        \end{proof}
        Until now, we have been considering mock toric varieties over $\A^1_k$ and their hypersurfaces, and the same holds in this case. 
        In particular, we show that for a general section, its hypersurface dominates $\A^1_k$ in Proposition\ref{prop: polytope and dominant}.
        \begin{definition}\label{def: model polytope}
        We use the notation in \cite[Proposition 4.14]{Y23}, that in \cite[Section 7]{Y23}, and the following ones:
        \begin{itemize}
            \item Let $(Z_\pi, S, u)$ be a mock polytope of $Z_\pi$. 
            \item Let $\kappa\colon S\rightarrow \ZZ$ be a map. 
            \item Let $u\oplus\kappa\in \mathscr{U}_{Z^1_{\pi^1}, S}$ denote the map defined as $u\oplus\kappa(i) = u(i)t^{\kappa(i)}$ for $i\in S$. 
            We recall that $u\oplus\kappa(i)\in k[Z^1_{\pi^1}]^*$ for any $i\in S$. 
            We call a mock polytope $(Z^1_{\pi^1}, S, u\oplus\kappa)$ of $Z^1_{\pi^1}$ a model of mock polytope of $(Z^\pi, S, u)$ along $\kappa$. 
        \end{itemize}
        \end{definition}
        \begin{proposition}\label{prop: abstract model}
            We keep the notation in  Definition \ref{def: model polytope}.    
            Let $u'\in\mathscr{U}_{Z_\pi, S}$ be a map. 
            If $u\sim u'$, then $u\oplus\kappa \sim u'\oplus\kappa$. 
        \end{proposition}
        \begin{proof}
            Let $\chi\in k[Z_\pi]^*$ be a unit such that $\chi u(i) = u'(i)$ for any $i\in S$. 
            We remark that $\chi\in k[Z^1_{\pi^1}]^*$ too. 
            Thus, $\chi u(i)t^{\kappa(i)} = u'(i)t^{\kappa(i)}$ in $k[Z^1_{\pi^1}]$ for any $i\in S$. 
            This equation shows that $u\oplus\kappa \sim u'\oplus\kappa$. 
        \end{proof}
        In Definition \ref{def: model polytope}, we can define the equivalent class $[u\oplus\kappa]$ of $\mathscr{U}_{Z^1_{\pi^1}, S}$ from an abstract mock polytope $(Z_\pi, S, [u])$ by Proposition \ref{prop: abstract model}. 
        We call that $(Z^1_{\pi^1}, S, [u\oplus\kappa])$  an abstract model mock polytope of $(Z_\pi, S, [u])$ along $\kappa$. 
        \begin{proposition}\label{prop: polytope and dominant}
            We keep the notation in Definition \ref{def: model polytope}.  
            Let $(Z^1_{\pi^1}, S, [u\oplus\kappa])$ be an abstract model mock polytope of $(Z_\pi, S, [u])$ along $\kappa$. 
            Let $(a_i)_{i\in S}\in k^{|S|}$ be a vector and $f$ denote $\sum_{i\in S}a_i(u\oplus\kappa)(i)\in k[Z^1_{\pi^1}]$. 
            If $d_{[u]} \geq 2$, then every irreducible component of $H^\circ_{Z^1_{\pi^1}, f}$ dominates ${\Gm^1}_{,k}$ for general $(a_i)_{i\in S}\in k^{|S|}$. 
        \end{proposition}
        \begin{proof}
            Let $(a_i)_{i\in S}\in k^{|S|}$ be a vector. 
            We assume that one irreducible component of $H^\circ_{Z^1_{\pi^1}, f}$ does not dominate ${\Gm^1}_{,k}$. 
            Because $k[Z^1_{\pi^1}]\cong k[t, t^{-1}][Z_\pi]$, there exists a non-unit polynomial $g\in k[t, t^1]$ such that $f$ can be divide by $g$. 
            Because $k$ is algebraically closed, there exists $\alpha\in k^*$ such that $t-\alpha | g$. 
            In particular, $\sum_{i\in S}\alpha^{\kappa(i)}a_iu(i) = 0\in k[Z_\pi]$.  

            Let $\rho_\kappa\colon{\Gm^1}_{,k}\times\A^{|S|}_{k}\rightarrow \A^{|S|}_{k}$ be an algebraic action of ${\Gm^1}_{,k}$ for $\A^{|S|}_k$ defined as $\rho_\kappa(t_0)(b_i) = (t_0^{-\kappa(i)}b_i)_{i\in S}$ for $t_0\in k^*$ and $(b_i)_{i\in S}\in k^{|S|}$. 
            Let $h$ denote a linear morphism from $k^{|S|}$ to $k[Z^1_{\pi^1}]^*$ defined as $h((c_i)_{i\in S}) = \sum_{i\in S}c_iu(i)$ for $(c_i)_{i\in S}\in k^{|S|}$. 
            From the property of $g$, we have $(\alpha^{\kappa(i)}a_i)_{i\in S}\in\ker(h)$. 
            We can identify with $k^{|S|}$ and $\A^{|S|}_k(k)$. 
            Let $X$ denote a linear subspace of $\A^{|S|}_k$ associated with $\ker(h)\subset k^{|S|}$. 
            On this identification, we can also check that $(a_i)_{i\in S}\in \rho_\kappa({\Gm^1}_{,k}(k))(X(k))$. 
             
            Because $d_{u}\geq 2$, we have $\dim(X)\leq |S| - 2$. 
            Thus, the restriction $\rho_{\kappa}|_{{\Gm^1}_{,k}\times X}\colon {\Gm^1}_{,k}\times X\rightarrow \A^{|S|}_k$ is not dominant. 
            Therefore, for general $(a_i)_{i\in S}\in k^{|S|}$, every irreducible component of $H^\circ_{Z^1_{\pi^1}, f}$ dominates ${\Gm^1}_{,k}$ from Chevalley's theorem.
        \end{proof}
        We use the notation in \cite[Definition 4.11]{Y23}. 
        At the end of this article, we compute the stable birational volume of hypersurfaces in $\Gr_\C(2, n)$. 
        In fact, we haven't found a mock toric structure of $\Gr_\C(2, n)$ yet. 
        Alternatively, it is necessary to construct a mock toric variety of dimension $n-3$. 
        Unfortunately, we were unable to find directly a proper mock variety of dimension $2(n-2) = \dim(\Gr_\C(2, n))$. 
        Instead, we consider a $(n-1)$-dimensional algebraic torus fibration over a $(n-3)$-dimensional proper mock toric variety. 
        Next, by compactifying this $2(n-2)$-fold, we obtain a proper mock toric variety of dimension $2(n-2)$. 
        The two crucial components for the compactification are $\val_{u, \pi}$ and $\Delta_{u, \pi}$, which are constructed from $\{\val_{u(i), \pi}\}_{i\in S}$ defined in \cite[Proposition 6.16]{Y23}.
        \begin{proposition}\label{prop: extend refinement along function}
            We keep the notation in \cite[Definition 4.11]{Y23}. 
            Let $(Z_\pi, S, u)$ be a mock polytope of $Z_\pi$. 
            We use the following notation:
            \begin{itemize}
                \item Let $\val_{u, \pi}$ denote the function from $\pi^{-1}_\RR(\supp(\Delta))$ to $\RR$ defined as follows for $w\in \pi^{-1}_\RR(\supp(\Delta))$: 
                \[
                    \val_{u, \pi}(w) = \min_{i\in S}\{\val_{u(i), \pi}(w)\}
                \]

                We remark that $\val_{u, \pi}$ is continuous and is preserved by the scalar product of non-negative real numbers. 
                Indeed, from \cite[Proposition 6.16]{Y23}, $\val_{u(i), \pi}$ is continuous and is preserved by the scalar product of non-negative real numbers for any $i\in S$. 
                Thus, $\val_{u, \pi}$ is so from the definition of $\val_{u, \pi}$. 
                \item Let $\sigma\in\Delta$ be a cone, $\sigma^\pi$ denote $\pi^{-1}_\RR(\sigma)$, and $\Omega^{\sigma^\pi}_{u, \pi}$ denote the following set:
                \[
                    \Omega^{\sigma^\pi}_{u, \pi} = \{(\chi, b)\in k[Z_\pi]^*\times\ZZ_{>0}\mid val_{u, \pi}(v') \leq \frac{1}{b}\val_{\chi, \pi}(v')\quad(\forall v'\in\sigma^\pi)\}
                \]
                \item For $(\chi, b)\in \Omega^{\sigma^\pi}_u$, let $C^{\sigma^\pi}_{u, \pi}(\chi, b)$ denote the following set:
                \[
                    C^{\sigma^\pi}_{u, \pi}(\chi, b) = \{v'\in \sigma^\pi\mid val_{u, \pi}(v') = \frac{1}{b}\val_{\chi, \pi}(v')\}
                \]
            \end{itemize}
            Then the following statements follow:
            \begin{enumerate}
            	\item[(a)] Let $\varphi\in\Phi$ be an element such that $\sigma\in\Delta_\varphi$, $\sigma_\varphi$ denote $(q_\varphi)_\RR(\sigma)$, and $\sigma^\pi_\varphi$ denotes a rational polyhedral convex cone $(q'_\varphi)_\RR(\sigma^\pi)$ in $(N'/s(N_\varphi))_\RR$ .  
            	Then 
            		\begin{itemize}
            			\item[(I)] $\sigma^\pi_\varphi = (\pi_\varphi)^{-1}_\RR(\sigma_\varphi)$
            			\item[(II)] $(q'_\varphi)_\RR|_{\sigma^\pi}\colon \sigma^\pi\rightarrow \sigma^\pi_\varphi$ is bijective. 
            			\item[(III)] $q'_\varphi|_{\langle\sigma^\pi\rangle\cap N'}\colon \langle\sigma^\pi\rangle\cap N'\rightarrow \langle\sigma^\pi_\varphi\rangle\cap N'/s(N_\varphi)$ is isomorphic. 
            		\end{itemize} 
                \item[(b)] Let $\sigma\in\Delta$ be a cone and $\varphi\in\Phi$ be an element such that $\sigma\in\Delta_\varphi$. 
                Let $\sigma^\pi_\varphi$ denotes $(q'_\varphi)_\RR(\sigma^\pi)$. 
                Then there exists $g_\varphi\in k[M'_\varphi]$ such that $\val^{\sigma^\pi}_{u, \pi} = \val^{\sigma^\pi_\varphi}_{g_\varphi}\circ(q'_\varphi)_\RR|_{\sigma^\pi}$.  
                \item[(c)] Let $g_\varphi\in k[M'_\varphi]$ be a function that satisfies the condition in (b). 
                Then for any $(\chi, b)\in\Omega^{\sigma^\pi}_{u, \pi}$, there exists $\omega\in\Omega^{\sigma^\pi_\varphi}_{g_\varphi}$ such that the following equation holds:
                \[
                    (q'_\varphi)_\RR(C^{\sigma^\pi}_{u, \pi}(\chi, b)) = C^{\sigma^\pi_\varphi}_{g_\varphi}(\omega)
                \]
                Conversely, for any  $\omega\in\Omega^{\sigma^\pi_\varphi}_{g_\varphi}$, there exists $(\chi, b)\in\Omega^{\sigma^\pi}_{u, \pi}$ such that the following equation holds:
                \[
                    C^{\sigma^\pi_\varphi}_{g_\varphi}(\omega) = (q'_\varphi)_\RR(C^{\sigma^\pi}_{u, \pi}(\chi, b))
                \]
                \item[(d)] The set $\{C^\sigma_{u, \pi}(\chi, b)\}_{(\chi, b)\in\Omega^{\sigma^\pi}_{u, \pi}}$ is a rational polyhedral convex fan and a refinement of $\sigma^\pi$.
                \item[(e)] Let $\tau$ be a face of $\sigma$ and $\tau^\pi$ denote $\pi^{-1}_\RR(\tau)$. 
                Then $\Omega^{\sigma^\pi}_{u, \pi}\subset\Omega^{\tau^\pi}_{u, \pi}$. 
                Moreover, for any $(\chi, b)\in\Omega^{\sigma^\pi}_{u, \pi}$, we have $C^{\sigma^\pi}_{u, \pi}(\chi, b)\cap\tau^\pi = C^{\tau^\pi}_{u, \pi}(\chi, b)$. 
                In particular, $C^{\tau^\pi}_{u, \pi}(\chi, b)\preceq C^{\sigma^\pi}_{u, \pi}(\chi, b)$.  
                \item[(f)] The following set $\Delta_{u, \pi}$ is a rational polyhedral convex fan in $N_\RR$ and $\supp(\Delta_{u, \pi}) = \pi^{-1}_\RR(\supp(\Delta))$:
                \[
                    \Delta_{u, \pi} = \bigcup_{\sigma\in\Delta}\{C^{\sigma^\pi}_{u, \pi}(\chi, b)\}_{(\chi, b)\in\Omega^{\sigma^\pi}_{u, \pi}}
                \]
                \item[(g)] For any $(\chi, b)\in\Omega^{\sigma^\pi}_{u, \pi}$, there exists $\chi'\in k[Z_\pi]^*$ such that $\val_{u, \pi}(v') = \val_{\chi', \pi}(v')$  for any $v'\in C^{\sigma^\pi}_{u, \pi}(\chi, b)$. 
                In particular, $\val_{u, \pi}$ is a linear function on $C^{\sigma^\pi}_{u, \pi}$ from \cite[Proposition 6.17]{Y23}. 
                \item[(h)] We use the notation in (f).
                We assume that $\Delta_{u, \pi}$ is a strongly convex rational polyhedral fan in $N'_\RR$. 
                Let $\Delta'$ denote $\Delta_{u, \pi}$ and $W$ denote a mock toric variety induced by $Z$ along $\pi_*\colon X(\Delta')\rightarrow X(\Delta)$ and $s$. 
                Then for the mock polytope $(W, S, u)$, we have $\Delta'_u = \Delta'$. 
            \end{enumerate}
        \end{proposition}
        \begin{proof}
            We prove the statements from (a) to (h) in order. 
            \begin{enumerate}
            	\item[(a)] We prove the statements from (I) to (III) in order. 
            		\begin{itemize}
            			\item[(I)] First, we prove $\sigma^\pi_\varphi \subset (\pi_\varphi)^{-1}_\RR(\sigma_\varphi)$. 
            			Let $v'\in\sigma^\pi$ be an element. 
            			Then $(\pi_\varphi)_\RR\circ(q'_\varphi)_\RR(v') = (q_\varphi)_\RR\circ(\pi_\varphi)_\RR(v')\in\sigma_\varphi$. 
            			Thus, $\sigma^\pi_\varphi = (q'_\varphi)_\RR(\sigma^\pi) \subset (\pi_\varphi)^{-1}_\RR(\sigma_\varphi)$. 
            			Next, we prove $\sigma^\pi_\varphi \supset (\pi_\varphi)^{-1}_\RR(\sigma_\varphi)$. 
            			Let $v''\in(\pi_\varphi)^{-1}_\RR(\sigma_\varphi)$ be an element. 
            			Then $(\pi_\varphi)_\RR(v'')\in (q_\varphi)_\RR(\sigma)$ and there exists $v\in\sigma$ such that $(q_\varphi)_\RR(v) = (\pi_\varphi)_\RR(v'')$. 
            			Let $w\in N'_\RR$ denote $s_\RR(v)$. 
            			Then $(\pi_\varphi)_\RR(v''-(q'_\varphi)_\RR(w)) = (q_\varphi)_\RR(v - \pi_\RR(w)) = 0$ and $v'' - (q'_\varphi)_\RR\in \ker((\pi_\varphi)_\RR)$. 
            			We remark that $\ker(q_\varphi\circ\pi) = \ker(\pi)\oplus s(N_\varphi)$, in particular, $\ker(\pi_\varphi) = q'_\varphi(\ker(\pi))$. 
            			Similarly, $\ker((\pi_\varphi)_\RR) = (q'_\varphi)_\RR(\ker(\pi_\RR))$. 
            			Hence, there exists $u\in\ker(\pi_\RR)$ such that $v'' - (q'_\varphi)_\RR(w) = (q'_\varphi)_\RR(u)$, and hence $v'' = (q'_\varphi)_\RR(u + w)$. 
            			Moreover, $\pi_\RR(u + w) = \pi_\RR(w) = v\in\sigma$, and hence $u + w\in \sigma^\pi$. 
            			Therefore, $v''\in \sigma^\pi_\varphi$. 
            			\item[(II)] It is enough to show that $(q'_\varphi)_\RR|_{\sigma^\pi}$ is injective. 
            			Let $v'_1, v'_2\in\sigma^\pi$ such that $(q'_\varphi)_\RR(v'_1) = (q'_\varphi)_\RR(v'_2)$. 
            			Then $v'_1 - v'_2\in s_\RR((N_\varphi)_\RR)$, and hence there exists $v\in (N_\varphi)_\RR$ such that $v'_1 - v'_2 = s_\RR(v)$ and  $v\in\langle\sigma\rangle\cap (N_\varphi)_\RR$. 
            			Because $\sigma$ is a rational polyhedral convex cone, $\langle\sigma\rangle\cap (N_\varphi)_\RR = 0$ from \cite[Proposition 3.2(b)]{Y23}. 
            			Thus, $v'_1 = v'_2$. 
            			\item[(III)] From the definition of $\sigma^\pi_\varphi$, $(q'_\varphi)_\RR(\langle\sigma^\pi\rangle) = \langle\sigma^\pi_\varphi\rangle$. 
            			Thus, $(q'_\varphi)(\langle\sigma^\pi\rangle\cap N')\subset \langle\sigma^\pi_\varphi\rangle\cap N'/s(N_\varphi)$. 
            			Moreover, from the proof of (II), $q'_\varphi|_{\langle\sigma^\pi\rangle\cap N'}$ is injective. 
            			
            			Hence, we show that $q'_\varphi|_{\langle\sigma^\pi\rangle\cap N'}$ is surjective. 
            			Let $v''\in \langle\sigma^\pi_\varphi\rangle\cap N'/s(N_\varphi)$ be an element. 
            			From (I), $(\pi_\varphi)_\RR(\langle\sigma^\pi_\varphi\rangle)\subset \langle\sigma_\varphi\rangle$, and hence $\pi_\varphi(v'')\in\langle\sigma_\varphi\rangle\cap N/N_\varphi$. 
            			From \cite[Proposition 3.2(d)]{Y23},  there exist $v\in \langle\sigma\rangle\cap N$ such that $\pi_\varphi(v'') = q_\varphi(v)$. 
            			Because $\pi_\varphi(v'' - q'_\varphi(s(v))) = \pi_\varphi(v'') - q_\varphi(\pi(s(v))) = \pi_\varphi(v'') - q_\varphi(v) = 0$, $v'' - q'_\varphi(s(v))\in\ker(\pi_\varphi)$. 
            			From the remark in the proof of (I), there exists $v'\in \ker(\pi)$ such that $q'_\varphi(v') = v'' - q'_\varphi(s(v))$. 
            			Because $v\in \langle\sigma\rangle\cap N$ and $v'\in \ker(\pi)$, $s(v) + v'\in \langle\sigma^\pi\rangle$. 
            			Therefore, $v'' = q'_\varphi(s(v) + v') \in q'_\varphi(\langle\sigma^\pi\rangle\cap N')$. 
            		\end{itemize}
                \item[(b)] From \cite[Proposition 6.17]{Y23}, for each $i\in S$, there exists $\omega_i\in M'_\varphi$ such that $\val^{\sigma^\pi}_{u(i)} = \val^{\sigma^\pi}_{\iota^*_\varphi(\chi^{\omega_i})}$. 
                Let $T$ denote a subset $\{\omega_i\mid i\in S\}$ of $M'_\varphi$ and $g_\varphi$ denote $\sum_{\omega\in T} \chi^{\omega}\in k[M'_\varphi]$. 
                Then for any $v'\in \sigma^\pi\cap N'$, we have $q'_\varphi(v)(g_\varphi) = \min_{\omega\in T}\{q'_\varphi(v)(\chi^\omega)\}$ because the definition of the torus invariant valuation on toric varieties. 
                This shows that for any $v'\in\sigma^\pi\cap N'$,  there exists the following equation:
                \begin{align*}
                    \val_{g_\varphi}\circ(q'_\varphi)_\RR(v')&= q'_\varphi(v')(g_\varphi)\\
                    &= \min_{\omega\in T}\{q'_\varphi(v')(\chi^\omega)\}\\
                    &= \min_{i\in S}\{q'_\varphi(v')(\chi^{\omega_i})\}\\
                    &= \min_{i\in S}\{\val_{Z,\pi}(v')({\iota'}^*_\varphi(\chi^{\omega_i}))\}\\
                    &= \min_{i\in S}\{\val_{{\iota'}^*_\varphi(\chi^{\omega_i}), \pi}(v')\}\\
                    &= \min_{i\in S}\{\val_{u(i), \pi}(v')\}\\
                    &= \val_{u, \pi}(v')
                \end{align*}
                Thus, from \cite[Lemma 8.13]{Y23}, the statement holds. 
                \item[(c)] We can check the statement likewise \cite[Proposition 6.8(b)]{Y23}. 
                In the proof of \cite[Proposition 6.8(b)]{Y23}, we used \cite[Proposition 6.3]{Y23}.
                Instead, we use \cite[Proposition 6.17]{Y23} and (a)-(III). 
                \item[(d)]  We can check the statement likewise \cite[Proposition 6.8(c)]{Y23}.
                \item[(e)]  Let $\omega\in \sigma^\vee\cap M$ be an element such that $\tau = \sigma\cap \omega^\perp$. 
                We remark that $\pi^*(\omega)\in (\sigma^\pi)^\vee\cap M'$ and $\tau^\pi = \sigma^\pi\cap \pi^*(\omega)^\perp$. In particular, $\tau^\pi\preceq\sigma^\pi$. 
                Now, we can check the statement likewise \cite[Proposition 6.8(d)]{Y23}.
                \item[(f)]  We can check the statement likewise \cite[Proposition 6.8(e)]{Y23}.
                \item[(g)]  We can check the statement likewise \cite[Proposition 6.8(f)]{Y23}. 
                \item[(h)] We remark that $\pi$ is compatible with the fans $\Delta'$ and $\Delta$ from (d) and (f). 
                Let $\tau\in \Delta'$ be a cone. 
                Then there exists $\sigma\in\Delta$ and $(\chi, b)\in\Omega^{\sigma^\pi}_{u, \pi}$ such that $\tau = C^{\sigma^\pi}_{u, \pi}(\chi, b)$. 
                From \cite[Proposition 4.12(b)]{Y23}, $\val_W = \val_{Z, \pi}$, in particular,  $\val^\tau_{u(i)} = \val^\tau_{u(i), \pi}$ for any $i\in S$ and $\val^\tau_{\chi} = \val^\tau_{\chi, \pi}$. 
                Hence, $\val^\tau_{u} = \val^\tau_{u, \pi}$ and $(\chi, b)\in \Omega^{\tau}_{u}$. 
                Moreover, from the definition of $C^{\tau}_{u}(\chi, b)$, we have $\tau = C^{\tau}_{u}(\chi, b)$. 
                From \ref{prop: VoPP2}(c), $\{C^\tau_u(\chi', b')\}_{(\chi', b')\in \Omega^\tau_u}$ is a refinment of $\tau$, and hence $\{C^\tau_u(\chi', b')\}_{(\chi', b')\in \Omega^\tau_u}$ is equal to the set of all faces of $\tau$. 
                This shows that $\Delta'_u = \Delta'$. 
            \end{enumerate}
        \end{proof}
        We construct $\Delta_{u, \pi}$ in Proposition \ref{prop: extend refinement along function}. 
        However, the above construction method is not practical in the sense that it is not suitable for computations using a programming code.
        The following proposition indicates that the construction of $\Delta_{u, \pi}$ itself can be achieved through straightforward calculations. 
        \begin{proposition}\label{prop: actual-configuration}
            We keep the notation in Proposition \ref{prop: extend refinement along function}. 
            Let $\sigma\in\Delta$ be a cone. 
            From \cite[Proposition 6.17]{Y23}, there exists $\omega_i\in M'$ for any $i\in S$ such that $\val_{u(i), \pi}(v') = \langle v', \omega_i\rangle$ for any $v'\in\sigma^\pi$. 
            We fix such $\omega_i\in M'$ for any $i\in S$ and let $D^\sigma_{u, \pi}$ be a rational polyhedral convex cone in  $(M'\oplus\ZZ)_\RR$ generated by the following set:
            \[
                \{(\omega', 0)\in(M'\oplus\ZZ)_\RR\mid\omega'\in(\sigma^\pi)^\vee\} \cup \{(\omega_i, 1)\in(M'\oplus\ZZ)_\RR\mid i\in S\}
            \]
            Let $C^\sigma_{u, \pi}$ denote a dual cone of $D^\sigma_{u, \pi}$ in $(N'\oplus\ZZ)_\RR$. 
            We remark that $C^\sigma_{u, \pi}$ is also a rational polyhedral convex cone in $(N\oplus\ZZ)_\RR$. 

            Then the following statements hold:
            \begin{enumerate}
                \item[(a)] The definition of $D^\sigma_{u, \pi}$ is independent of the choice of $\{\omega_i\}_{i\in S}$. 
                \item[(b)] For any $(v', a)\in C^\sigma_{u, \pi}$, we have that $v'\in\sigma^\pi$. 
                \item[(c)] Let $v'\in\sigma^\pi$ be an element and $a\in\RR$ be a real number. 
                Then $\val_{u, \pi}(v') + a \geq 0$ if and only if $(v', a)\in C^\sigma_{u, \pi}$. 
                \item[(d)] Let $\tau\preceq C^\sigma_{u, \pi}$ be a face. 
                We assume that $(0, 1)\notin \tau$. 
                Then for any $(v', a)\in \tau$, we have $a = -\val_{u, \pi}(v')$. 
                \item[(e)] For any $v'\in\sigma^\pi$, there exists a face $\tau\preceq C^\sigma_{u, \pi}$ such that $(v', -\val_{u, \pi}(v'))\in\tau$ and $(0, 1)\notin\tau$. 
                \item[(f)] Let $\Delta^{\sigma, u, \pi}$ denote the following set:
                \[
                    \Delta^{\sigma, u, \pi} = \{\tau\preceq C^\sigma_{u, \pi}\mid (0, 1) \notin\tau\}
                \]
                We remark that $\Delta^{\sigma, u, \pi}$ is a rational polyhedral fan in $(N'\oplus\ZZ)_\RR$. 
                Let $\pr_1$ denote a first projection from $N'\oplus\ZZ$ to $N'$. 
                Then the following subset $\Delta^\sigma_{u, \pi}$ is a rational polyhedral fan in $N'_\RR$ with $\supp(\Delta^\sigma_{u, \pi}) = \sigma^\pi$:
                \[
                    \Delta^\sigma_{u, \pi} = \{(\pr_1)_\RR(\tau)\mid \tau\in\Delta^{\sigma, u, \pi}\}
                \]
                \item[(g)] Let $E^\sigma_{u, \pi}$ denote a convex cone in $(N'\oplus\ZZ)_\RR$ generated by the following set:
                \[
                    \{(v', -\val_{u, \pi}(v'))\mid v'\in\sigma^\pi\}\cup\{(0, 1)\}
                \]
                Then $E^\sigma_{u, \pi}$ is $C^\sigma_{u, \pi}$. 
                \item[(h)] The cone $C^\sigma_{u, \pi}$ is strongly convex if and only if $\Delta^\sigma_{u, \pi}$ is a strongly convex fan in $N'_\RR$. 
                \item[(i)] For any $(\chi, b)\in\Omega^{\sigma^\pi}_{u, \pi}$, there exists $\tau\in\Delta^{\sigma, u, \pi}$ such that $C^{\sigma^\pi}_{u, \pi}(\chi, b) = (\pr_1)_\RR(\tau)$. 
                \item[(j)] Conversely, for any $\tau\in\Delta^{\sigma, u, \pi}$, there exists $(\chi, b)\in\Omega^{\sigma^\pi}_{u, \pi}$ such that $(\pr_1)_\RR(\tau) = C^{\sigma^\pi}_{u, \pi}(\chi, b)$. 
                Besides the statement (i), we have that $\{C^{\sigma^\pi}_{u, \pi}(\chi, b)\}_{(\chi, b)\in\Omega^{\sigma^\pi}_{u, \pi}} = \Delta^\sigma_{u, \pi}$. 
                \item[(k)] Let $\varsigma\in\Delta^\sigma_{u, \pi}$ be a cone. 
                Then there exists $\sigma\in\Delta$ and $\tau\in\Delta^{\sigma, u, \pi}$ such that $\varsigma = (\pr_1)_\RR(\tau)$ from (j). 
                Let $(v', a)\in \tau^\circ$ be an element. 
                The following equation holds. 
                \[
                    \{i\in S\mid \val^{\varsigma}_{u, \pi} = \val^{\varsigma}_{u(i), \pi} \} = \{i\in S\mid \langle (v', a), (\omega_i, 1)\rangle = 0\}
                \]
            \end{enumerate}
        \end{proposition}
        \begin{proof}
            We prove the statement from (a) to (k) in order. 
            \begin{enumerate}
                \item[(a)] Let $i\in S$ be an integer and $\omega'_i\in M'$ be an element such that $\val_{u(i), \pi}(v') = \langle v', \omega'_i\rangle$ for any $v'\in \sigma^\pi$. 
                Then $\omega'_i - \omega_i\in (\sigma^\pi)^\vee$ because $\langle v', \omega'_i - \omega_i\rangle = 0$ for any $v'\in \sigma^\pi$. 
                Thus, $(\omega'_i, 1)\in D^\sigma_{u, \pi}$ from the definition of $D^\sigma_{u, \pi}$. 
                This shows that the definition of $D^\sigma_{u, \pi}$ is independent of the choice of $\{\omega_i\}_{i\in S}$. 
                \item[(b)] Let $(v', a)\in C^\sigma_{u, \pi}$ be an element. 
                For any $\omega'\in(\sigma^\pi)^\vee$, we have $(\omega', 0)\in D^\sigma_{u, \pi}$. 
                Thus, $\langle v', \omega'\rangle = \langle (v', a), (\omega', 0)\rangle\geq 0$ for any $\omega'\in(\sigma^\pi)^\vee$. 
                Therefore, $v'\in\sigma^\pi$. 
                \item[(c)] From the definition of $C^\sigma_{u, \pi}$, we can check that $(v', a)\in C^\sigma_{u, \pi}$ if and only if $\langle v', \omega_i\rangle \geq -a$ for any $i\in S$. 
                Because $v'\in \sigma^\pi$, $\val_{u(i), \pi}(v') = \langle v', \omega_i\rangle$ for any $i\in S$, we can check the statement follows from the definition of $\val_{u, \pi}$.  
                \item[(d)] Let $(v', a)\in\tau$ be an element. 
                Then $v'\in\sigma^\pi$ from (b). 
                From (c), $a\geq -\val_{u, \pi}(v')$. 
                Moreover, $(v', -\val_{u, \pi}(v'))$ and $(0, a +\val_{u, \pi}(v'))\in C^\sigma_{u, \pi}$. 
                Because $\tau\preceq C^\sigma_{u, \pi}$, we have that $(v', -\val_{u, \pi}(v'))$ and $(0, a + \val_{u, \pi}(v'))\in\tau$. 
                Thus, from the assumption of $\tau$, we have that $a = -\val_{u, \pi}(v')$. 
                \item[(e)] Let $\tau\preceq C^\sigma_f$ be a face such that $(v', -\val_{u, \pi}(v'))\in\tau^\circ$. 
                If $\tau$ contains $(0, 1)$, then there exists $\epsilon>0$ such that $(v', -\val_{u, \pi}(v') - \epsilon)\in\tau$ because $(v', -\val_{u, \pi}(v'))\in\tau^\circ$. 
                However, it is a contradiction to (c). 
                Thus, we have that $(0, 1)\notin\tau$. 
                \item[(f)] Because $(\pr_1)_\RR|_{\supp(\Delta^{\sigma, u, \pi})}$ is injective from (d). 
                 From \cite[Lemma 8.1(c)]{Y23}, $\Delta^{\sigma}_{u, \pi}$ is a polyhedral convex fan. 
                A morphism $\pr_1$ is a lattice morphism, and $\Sigma^{\sigma, u, \pi}$ is a rational polyhedral fan so that $\Delta^{\sigma}_{u, \pi}$ is a rational polyhedral convex fan.  
                Moreover, $\supp(\Delta^\sigma_{u, \pi}) = \sigma^\pi$ from (a) and (e). 
                \item[(g)] We can check easily that $E^\sigma_{u, \pi}\subset C^\sigma_{u, \pi}$ from (c). 
                Let $(v', a)\in C^\sigma_{u, \pi}$be an element.
                From (b), $v'\in\sigma^\pi$. 
                Moreover, from (c), $a +\val_{u, \pi}(v')\geq 0$. 
                Thus, $(v', a) = (v', -\val_{u, \pi}(v')) + (0, a +\val_{u, \pi}(v'))\in E^\sigma_{u, \pi}$. 
                \item[(h)]
                Because $(\pr_1)_\RR|_{\supp(\Delta^{\sigma, u, \pi})}$ is injective, the strong convexity of $\Delta^{\sigma, u, \pi}$ is equivalent to that of $\Delta^\sigma_{u, \pi}$ from \cite[Lemma 8.1(b)]{Y23}. 
                Let $\tau$ denote the minimal face of $C^\sigma_{u, \pi}$. 
                We remark that $\tau$ is a linear subspace of $(N'\oplus\ZZ)_\RR$. 
                Because $C^\sigma_{u, \pi}$ does not contain $(0, -1)$, $\tau$ does not contain $(0, 1)$. 
                Thus, $\tau\in \Delta^{\sigma, u, \pi}$ and 
                the statement follows. 
                \item[(i)] 
                Let $(\chi, b)\in\Omega^{\sigma^\pi}_{u, \pi}$ be an element. 
                From \cite[Proposition 6.17]{Y23}, there exists $\omega\in M'$  such that $\val_{\chi, \pi}(v') = \langle v', \omega\rangle$ for any $v'\in\sigma^\pi$. 
                Because $(\chi, b)\in\Omega^{\sigma^\pi}_{u, \pi}$, we have $\langle v', \omega\rangle - b\val_{u, \pi}(v') \geq 0$ for any $v'\in \sigma^\pi$. 
                This indicates that $(\omega, b)\in D^\sigma_{u, \pi}$ from (g). 
                Let $\tau$ denote $C^\sigma_{u, \pi}\cap (\omega, b)^{\perp}\preceq C^\sigma_{u, \pi}$. 
                Then $\tau\in\Delta^{\sigma, u, \pi}$ because $b > 0$. 
                Moreover, we can check that $(\pr_1)_\RR(\tau) = C^{\sigma^\pi}_{u, \pi}(\chi, b)$. 
                \item[(j)] Let $\tau\in\Delta^{\sigma, u, \pi}$ be a cone. 
                Then there exists $(\omega, b)\in D^\sigma_{u, \pi}$ such that $\tau = C^\sigma_{u, \pi}\cap (\omega, b)^{\perp}$. 
                Because $D^\sigma_{u, \pi}$ is rational polyhedral cone and $(0, 1)\notin\tau$, we may assume that $\omega\in M'$ and $b\in\ZZ_{>0}$. 
                Let $\varphi\in\Phi$ be an element such that $\sigma\in\Delta_\varphi$. 
                Then there exists a section $s'$ of $q'_\varphi\colon N'\rightarrow N'/s(N_\varphi)$ such that $\sigma^\pi\subset s'_\RR((N'/s(N_\varphi))_\RR)$ from \cite[Lemma 8.2]{Y23} and the condition (1) and (3) in \cite[Definition 3.1]{Y23}. 
                In particular, $\ker(s'^*)\subset M'\cap (\sigma^\pi)^\perp$. 
                Let $M'_\varphi$ denote the dual lattice of $N'/s(N_\varphi)$. 
                Then there exists the direct summation $M = \ker(s'^*)\oplus (q'_\varphi)^*(M'_\varphi)$. 
                Thus, there exists $\omega_1\in (q'_\varphi)^*(M'_\varphi)$ and $\omega_2\in \ker(s'^*)$ such that $\omega = \omega_1 + \omega_2$. 
                Because $\omega_2\in\ker(s'^*)\subset (\sigma^\pi)^\perp$, we have $(\omega_1, b)\in D^\sigma_{u, \pi}$ and $\tau = C^\sigma_{u, \pi}\cap (\omega_1, b)^{\perp}$ from (a). 
                Let $\chi$ denote $\iota'^*(\chi^{\omega_1})\in k[Z^\pi]^*$. 
                Then $\val_{\chi, \pi}(v') = \langle v', \omega_1\rangle$ for any $v'\in \sigma^\pi$ because $\omega_1\in (q'_\varphi)^*(M'_\varphi)$. 
                Thus, we can check that $(\chi, b)\in \Omega^{\sigma^\pi}_{u, \pi}$ and $(\pr_1)_\RR(\tau) = C^{\sigma^\pi}_{u, \pi}(\chi, b)$. 
                \item[(k)] We recall that $v'\in \sigma^\pi$ from (b) and $a = -\val_{u, \pi}(v')$ from (d). 
                Let $i\in S$ be an integer. 
                Thus, $\langle (v', a), (\omega_i, 1)\rangle = 0$ if and only if $\val_{u, \pi}(v') = \val_{u(i), \pi}(v')$. 
                Moreover, Because $(\pr_1)_\RR|_\tau$ is injective, $v'\in \varsigma^\circ$. 
                Hence, $\val_{u, \pi}(v') = \val_{u(i), \pi}(v')$ if and only if $\val^\varsigma_{u, \pi} = \val^\varsigma_{u(i), \pi}$ from \cite[Proposition 6.17]{Y23} and Proposition \ref{prop: extend refinement along function}(g) and the definition of $\val_{u, \pi}$. 
            \end{enumerate}
        \end{proof}
        In general, the commputation of $D^\sigma_{u, \pi}$ is very difficult because we have to find $\omega_i\in M'$ such that $\val_{u(i), \pi}(v') = \langle v', \omega_i\rangle$ for any $v'\in\sigma^\pi$ and $i\in S$. 
        In fact, the following lemma shows that we can compute $D^\sigma_{u, \pi}$ briefly if $\{u(i)\}_{i\in S}$ is specific. 
        \begin{lemma}\label{lem: supplement-of-valuation3}
        We keep the notation in Proposition \ref{prop: actual-configuration}.
        We assume that for any $i\in S$, there exists $\omega_i\in M'$ and $c_i\in k^*$ such that $u(i) = c_i{\iota'}^*(\chi^{\omega_i})$. 
        Then $D^\sigma_{u, \pi}$ is generated by the following elements: 
        \[
            \{(\omega', 0)\in(M'\oplus\ZZ)_\RR\mid\omega'\in(\sigma^\pi)^\vee\} \cup \{(\omega_i, 1)\in(M'\oplus\ZZ)_\RR\mid i\in S\}
        \]
    \end{lemma}
    \begin{proof}
        From \cite[Lemma 8.13]{Y23} and \cite[Corollary 4.13]{Y23}, $\val_{u(i), \pi}(v') = \langle v', \omega_i\rangle$ for any $v'\in \sigma^\pi$ and $i\in S$. 
        Thus, from Proposition \ref{prop: actual-configuration}(a), the statement holds. 
    \end{proof}
        
        As mentioned at the beginning of this section, the computation of stable birational volume of hypersurfaces of mock toric varieties assumes certain conditions on the fan. 
        The following proposition describes a method to overcome these obstacles, which is geometrically referred to as the semistable reduction.
        \begin{lemma}\label{lem: modified fan}
            We keep the notation in Definition \ref{def: model polytope}.  
            Let $(Z^1_{\pi^1}, S, u\oplus\kappa)$ be a model mock polytope of a mock polytope $(Z_\pi, S, u)$ along $\kappa$. 
            Let $n$ be a positive integer. 
            Let $r_n\colon N'\oplus\ZZ\rightarrow N'\oplus\ZZ$ be a linear morphism defined as $r_n(v, c) = (nv, c)$ for $(v, c)\in N'\oplus\ZZ$. 
            Then the following statements follow:
            \begin{enumerate}
                \item[(a)] $n\val_{u\oplus\kappa, \pi^1} = \val_{u\oplus{n\kappa}, \pi^1}\circ (r_n)_\RR|_{(\pi^1_\RR)^{-1}(\supp(\Delta\times\Delta_!))}$
                \item[(b)] Let $\sigma\in\Delta\times\Delta_!$ be a cone, $\chi\in k[Z_\pi]^*$ be a unit, $a\in \ZZ$ be an integer, and $b\in\ZZ_{>0}$ be a positive integer. 
                Let $\sigma^\pi$ denote $(\pi^1_\RR)^{-1}(\sigma)$. 
                Then $(\chi t^a, b)\in\Omega^{\sigma^\pi}_{u\oplus\kappa, \pi^1}$  if and only if $(\chi t^{an}, b)\in \Omega^{\sigma^\pi}_{u\oplus{n\kappa}, \pi^1}$. 
                In this case, $(r_n)_\RR(C^{\sigma^\pi}_{u\oplus\kappa, \pi^1}(\chi t^a, b))=C^{\sigma^\pi}_{u\oplus{n\kappa}, \pi^1}(\chi t^{an}, b)$. 
                \item[(c)] We use the notation in (b). 
                Then $(\chi t^a, b)\in \Omega^{\sigma^\pi}_{u\oplus{n\kappa}, \pi^1}$ if and only if $(\chi^n t^a, nb)\in \Omega^{\sigma^\pi}_{u\oplus\kappa, \pi^1}$. 
                In this case, $(r_n)_\RR(C^{\sigma^\pi}_{u\oplus\kappa, \pi^1}(\chi^n t^a, bn))=C^{\sigma^\pi}_{u\oplus{n\kappa}, \pi^1}(\chi t^a, b)$. 
                \item[(d)] 
                The map $(r_n)_\RR$ induces a one-to-one correspondence with cones in $(\Delta\times\Delta_!)_{u\oplus\kappa, \pi^1}$ and those in $(\Delta\times\Delta_!)_{u\oplus{n\kappa}, \pi^1}$. 
                Moreover, if $(\Delta\times\Delta_!)_{u\oplus\kappa}$ is strongly convex, then $(\Delta\times\Delta_!)_{u\oplus{l\kappa}}$ is strongly convex for any $l\in\ZZ_{>0}$. 
                \item[(e)] There exists $m\in\ZZ_{>0}$ satisfying the following condition (*):
                \begin{itemize}
                    \item[(*)] For any $l\in \ZZ_{>0}$, there exists a strongly convex rational polyhedral fan $\Delta'_l$ in $(N'\oplus\ZZ)_\RR$ such that $\Delta'_l$ is a refinement of $(\Delta\times\Delta_!)_{u\oplus{ml\kappa}, \pi^1}$, generically unimodular, specifically reduced, and compactly arranged. 
                \end{itemize}
            \end{enumerate}
        \end{lemma}
        \begin{proof}
            We prove the statements from (a) to (e) in order. 
            \begin{enumerate}
                \item[(a)] 
                Let $v' = (v, c)\in (\pi^1_\RR)^{-1}(\supp(\Delta\times\Delta_!))$ be an element. 
                For any $i\in S$, the following equation folds:
                \begin{align*}
                    \val_{u\oplus{n\kappa}(i), \pi^1}\circ r_n(v')&= \val_{u\oplus{n\kappa}(i), \pi^1}\circ r_n(v, c)\\
                    &= \val_{u\oplus{n\kappa}(i), \pi^1}(nv, c)\\
                    &= \val_{u(i)t^{n\kappa(i)}, \pi^1}(nv, c)\\
                    &= \val_{u(i), \pi}(nv) + nc\kappa(i)\\
                    &= n(\val_{u(i), \pi}(v) + c\kappa(i))\\
                    &= n(\val_{u(i)t^{\kappa(i)}, \pi^1}(v, c))\\
                    &= n\val_{u\oplus\kappa(i), \pi^1}(v')\\
                \end{align*}
                Thus, the statement holds from the definition of $\val_{u\oplus n\kappa, \pi^1}$ and  $\val_{u\oplus \kappa, \pi^1}$.
                \item[(b)] We remark that the restriction $(r_n)_\RR|_{\sigma^\pi}$ is an automorphism of $\sigma^\pi$ from the definition of $r_n$. 
                Moreover, for any $\chi\in k[Z_\pi]^*$ and  $a\in\ZZ$, we can check that $\val_{\chi t^{an}, \pi^1} \circ (r_n)_\RR|_{ (\pi^1_\RR)^{-1}(\supp(\Delta\times\Delta_!))} = n\val_{\chi t^{a}, \pi^1}$ from the similar argument in the proof of (a). 
                Thus, we can check that the statement holds. 
                \item[(c)] We remark that for any $\chi\in k[Z_\pi]^*$ and $a\in\ZZ$, we have $n(\frac{1}{n}\val_{\chi^nt^a, \pi^1}) = \val_{\chi t^a, \pi^1}\circ (r_n)_\RR|_{ (\pi^1_\RR)^{-1}(\supp(\Delta\times\Delta_!))}$. 
                Then, we can check that the statement holds from the similar argument in the proof of (b). 
                \item[(d)] The restriction $(r_n)_\RR|_{(\pi^1)^{-1}(\supp(\Delta\times\Delta_!))}$ is an automorphism of $(\pi^1)^{-1}(\supp(\Delta\times\Delta_!))$. 
                Thus, from (b) and (c), $r_n$ induces a one-to-one correspondence with cones in $(\Delta\times\Delta_!)_{u\oplus\kappa, \pi^1}$ and those in $(\Delta\times\Delta_!)_{u\oplus{n\kappa}, \pi^1}$. 
                
                For the latter part of the statement, we can check it because $(r_l)_\RR$ is a linear isomorphism for any $l\in\ZZ_{>0}$. 
                \item[(e)] Let $m\in\ZZ_{>0}$ be a positive integer and $\Delta'$ be a strongly convex rational polyhedral fan in $(N'\oplus\ZZ)_\RR$ which is a refinement of $(\Delta\times\Delta_!)_{u\oplus m\kappa, \pi^1}$. 
                We assume that $\Delta'$ is generically unimodular, specifically reduced, and compactly arranged. 
                For $l\in\ZZ_{>0}$, let $\Delta'_l$ denote the following set:
                \[
                    \Delta'_l = \{(r_l)_\RR(\tau)\mid\tau\in\Delta'\}
                \]
                From \cite[Lemma 8.1(d)]{Y23}, $\Delta'_l$ is a strongly convex rational polyhedral fan in $(N'\oplus\ZZ)_\RR$. 
                Moreover, from (d), $\Delta'_l$ is a refinement of $(\Delta\times\Delta_!)_{u\oplus{ml\kappa}, \pi^1}$. 
                Now, we show that $\Delta'_l$ is also generically unimodular, specifically reduced, and compactly arranged for any $l\in\ZZ_{>0}$. 
                
                From the definition of $r_l$, we can check that $\Delta'_l$ is generically unimodular. 
                
                Let $\tau\in\Delta'$ be a cone. 
                Then from the definition of $r_l$, we can check that $\tau\in\Delta'_{\spe}$ if and only if $(r_l)_\RR(\tau)\in(\Delta'_l)_{\spe}$. 
                Moreover, we can check that $\tau\in\Delta'_{\bdd}$ if and only if 
                $(r_l)_\RR(\tau)\in(\Delta'_l)_{\bdd}$. 
                Therefore, from (d), $\Delta'_l$ is compactly arranged. 

                We remark that $(r_l)_\RR$ induces a one-to-one correspondence with $\gamma\in\Delta'_{\spe}$ such that $\dim(\gamma) = 1$ and those in $\Delta'_l$ from the definition of $r_l$.
                Let $\gamma\in\Delta'_{\spe}$ be a cone such that $\dim(\gamma) = 1$ and 
                $(v, c)\in N'\oplus\ZZ$ be a minimal generator of $\gamma$. 
                From the assumption, $c = 1$. 
                Thus, $(lv, 1)$ is a generator of $(r_l)_\RR(\gamma)$. 
                Hence, $\Delta'_l$ is specifically reduced. 

                From the above argument, it is enough to show that the following statement (**) holds:
                \begin{enumerate}
                    \item[(**)] There exists $m\in\ZZ_{>0}$ and $\Delta'$ be a strongly convex rational polyhedral fan in $(N'\oplus\ZZ)_\RR$ such that $\Delta'$ is a refinement of $(\Delta\times\Delta_!)_{u\oplus m\kappa, \pi^1}$, generically unimodular, specifically reduced, and compactly arranged. 
                \end{enumerate}
                Let $\Delta''$ be an unimodular refinement of $(\Delta\times\Delta_!)_{u\oplus\kappa, \pi^1}$. 
                For a positive integer $m$, let $\Delta''_m$ denote the following set:
                \[
                    \Delta''_m = \{(r_m)_\RR(\tau)\mid\tau\in\Delta''\}
                \]
                From \cite[Proposition 7.4(b)]{Y23}, we can check that $\Delta''_m$ is a strongly convex rational polyhedral fan in $(N'\oplus\ZZ)_\RR$, a refinement of $(\Delta\times\Delta_!)_{u\oplus m\kappa, \pi^1}$, generically unimodular, and compactly arranged for any $m\in\ZZ_{>0}$. 
                Let $\Gamma$ denote the following subset of $\Delta''$:
                \[
                    \Gamma = \{\gamma\in\Delta''_{\spe}\mid \dim(\gamma) = 1\}
                \]
                For $\gamma\in\Gamma$, let $v_\gamma$ denote an element in $N'_\Q$ such that $(v_\gamma, 1)\in\gamma$. 
                Because $|\Gamma|$ is finite, there exists $m_0\in\ZZ_{>0}$ such that $m_0v_\gamma\in N'$ for any $\gamma \in\Gamma$. 
                Therefore, we can check that $\Delta''_{m_0}$ is specifically reduced. 
            \end{enumerate}
        \end{proof}
        From Lemma\ref{lemma: polytope and stratification}, we can compute the intersection of a hypersurface and a stratum of $Z_u$ for a general section in the linear system defined by $[u]$.  
        However, calculating this in practice is difficult in general. 
        Indeed, this is also a different point between toric varieties and mock toric varieties.
        The following proposition indicates that, under favorable conditions on the cone, the computation can be relatively straightforward.
        \begin{proposition}\label{prop: computation of reduction step}
            We use the notation in \cite[Definition 4.11]{Y23} and the following notation:
            \begin{itemize}
                \item Let $\sigma$ be a strongly convex rational polyhedral cone in $N'_\RR$. 
                We assume that $\pi_\RR(\sigma) = \{0\}$. 
                Let $\tau$ denote $p_\RR(\sigma)$. 
                \item Let $\Delta_1$ denote a fan whose cones are all faces of $\sigma$ and $W$ denote a mock toric variety induced by $Z$ along $\pi_*\colon X(\Delta_1)\rightarrow X(\Delta)$ and $s$. 
            \end{itemize}
            Then the following statements follow:
            \begin{enumerate}
                \item[(a)] The set $\tau$ is a strongly convex rational polyhedral cone in $(N_1)_\RR$. 
                \item[(b)] Let $\varphi\in\Phi$ be an element. 
                Then, all squares in the following diagram are Cartesian squares: 
                \begin{equation*}
                    \begin{tikzcd}
                        W_\sigma\ar[r]\ar[d]& O_\sigma\ar[r]\ar[d]& O_{\sigma_\varphi}\ar[r]\ar[d]& O_\tau\ar[d]\\
                        W(\sigma)\ar[r, "\iota'"]\ar[d]& X(\sigma)\ar[r, "(q'_\varphi)_*"]\ar[d, "\pi_*"]& X(\sigma_\varphi)\ar[r, "(p_\varphi)_*"]\ar[d, "(\pi_\varphi)_*"]& X(\tau)\ar[d]\\
                        Z^\circ\ar[r, "\iota"]&T_{N}\ar[r, "(q_\varphi)_*"]&T_{N/N_\varphi}\ar[r]&\Spec(k)
                    \end{tikzcd}
                \end{equation*}
                \item[(c)] Let $\epsilon^*$ denote the ring morphism $k[Z^\circ]\rightarrow k[Z_\pi]$ induced by $\epsilon\colon Z_\pi\rightarrow Z^\circ$ and $p^*_0$ denote  the ring morphism $k[M_1]\rightarrow k[Z_\pi]$ induced by $p_0\colon Z_\pi\rightarrow 
                T_{N_1}$. 
                Then $\epsilon^*\otimes \pi^*\colon k[Z^\circ]\otimes k[M_1]\rightarrow k[Z_\pi]$ is isomorphic, and all horizontal morphisms in the following diagram are isomorphic:
                \begin{equation*}
                    \begin{tikzcd}
                        k[Z^\circ]\otimes_k k[M_1]\ar[r]&k[Z_\pi]\\
                        k[Z^\circ]\otimes_k k[M_1\cap \tau^{\vee}]\ar[r]\ar[u, hook]\ar[d]&\Gamma(W(\sigma), \OO_W)\ar[u, hook]\ar[d]\\
                        k[Z^\circ]\otimes_k k[M_1\cap \tau^\perp]\ar[r]&\Gamma(W_\sigma, \OO_{\overline{W_\sigma}})\\
                    \end{tikzcd}
                \end{equation*}
                where downward morphisms are quotient morphisms and upward morphisms are localization morphisms. 
                \item[(d)] Let $(W, S, u)$ be a mock polytope of $W$.  
                From \cite[Proposition 4.12(c)]{Y23}, there exists $\eta_i\in k[Z^\circ]^*$ and $\omega_i\in M_1$ such that $u(i) = \epsilon^*(\eta_i)p^*_0(\chi^{\omega_i})$ for any $i\in S$. 
                We assume that $\sigma\in (\Delta_1)_u$.
                Then there exists $\omega\in M_1$ such that $\omega_i-\omega\in\tau^\perp$ for any $i\in S^\sigma$ and $\val^\sigma_{p^*(\chi^\omega)} = \val^\sigma_u$. 
                \item[(e)] We use the notation in (d).  
                Let $u'\colon S^\sigma\rightarrow k[Z^\circ]\otimes_k k[O_\tau]$ be a map defined as $u'(i) = \eta_i\otimes\chi^{\omega_i-\omega}$ for $i\in S^\sigma$. 
                Then $u^\sigma\sim u'$ on the identification with $k[Z^\circ]\otimes k[O_\tau]$ and $k[W_\sigma]$ from (c). 
            \end{enumerate}
        \end{proposition}
        \begin{proof}
            We prove the statements from (a) to (e) in order. 
            \begin{enumerate}
                \item[(a)] From the assumption of $\sigma$, $p_\RR|_\sigma$ is injective. 
                Thus, $\tau$ is a strongly convex cone from \cite[Lemma 8.1(a)]{Y23} and (d). 
                Because $p$ is a lattice morphism, $\tau$ is a rational polyhedral convex cone. 
                \item[(b)] From the definition of the mock toric structure of $W$, two squares on the left side of the above diagram are Cartesian squares. 
                We can identify with $\sigma$ and $\{0\}\times\tau$ by the isomorphism $\pi\times p\colon N'\rightarrow N\oplus N_1$. 
                Thus, all squares of the following diagram are Cartesian squares: 
                \begin{equation*}
                    \begin{tikzcd}
                        O_\sigma\ar[r]\ar[d]&O_{\tau}\ar[d]\\
                        X(\sigma)\ar[r, "p_*"]\ar[d, "\pi_*"]&X(\tau)\ar[d]\\
                        T_N\ar[r]&\Spec(k)
                    \end{tikzcd}
                \end{equation*}
                Moreover, We can identify with $\sigma_\varphi$ and $\{0\}\times\tau$ by the isomorphism $\pi_\varphi\times p_\varphi\colon N'/s(N_\varphi)\rightarrow N/N_\varphi\oplus N_1$. 
                Similarly, all squares of the following diagram are Cartesian squares: 
                \begin{equation*}
                    \begin{tikzcd}
                        O_{\sigma_\varphi}\ar[r]\ar[d]&O_{\tau}\ar[d]\\
                        X(\sigma_\varphi)\ar[r, "{p_\varphi}_*"]\ar[d, "{\pi_\varphi}_*"]&X(\tau)\ar[d]\\
                        T_{N/N_\varphi}\ar[r]&\Spec(k)
                    \end{tikzcd}
                \end{equation*}
                \item[(c)] The statement follows from (b). 
                \item[(d)] Because $\sigma\in (\Delta_1)_u$, there exists $\chi'\in k[Z_\pi]^*$ such that $\val^\sigma_{\chi'} = \val^\sigma_u$ from Proposition \ref{prop: VoPP2}(d). 
                From \cite[Proposition 4.12(c)]{Y23}, there exists $\eta\in k[Z^\circ]^*$ and $\omega\in M_1$ such that $\chi' = \epsilon^*(\eta)p^*_0(\chi^{\omega})$. 
                From the assumption of $\sigma$ and \cite[Proposition 4.12(b)]{Y23} and (d), $\val^\sigma_{\chi'} = \val^\sigma_{p^*_0(\chi^\omega)}$ and $\val^\tau_{\chi^\omega} = \val^\tau_{\chi^{\omega_i}}$.  
                Thus, $\val^\sigma_{p^*_0(\chi^\omega)} = \val^\sigma_u$ and $\omega_i-\omega\in \tau^\perp$ for any $i\in S^\sigma$. 
                \item[(e)] From (d) and \cite[Proposition 4.8]{Y23}, $\chi'^{-1}p^*_0(\chi^\omega) \in\Gamma(W(\sigma), \OO_W)^*$. 
                Let $p^\sigma$ denote a ring morphism $\Gamma(W(\sigma), \OO_W)\rightarrow\Gamma(W_\sigma, \OO_{\overline{W_\sigma}})$ associated with the closed immersion $W_\sigma\hookrightarrow W(\sigma)$. 
                Let $\beta$ denote $p^\sigma(\chi'^{-1}p^*_0(\chi^\omega))$. 
                Then $\beta\in\Gamma(W_\sigma, \OO_{\overline{W_\sigma}})^*$ and $u^\sigma(i) = \beta u'(i)$ for any $i\in S^\sigma$. 
                Therefore, $u^\sigma\sim u'$. 
            \end{enumerate}
        \end{proof}
        The following proposition shows that the stable birational volume of a general section in the linear system defined by the mock polytope can be computed, achieving the aim of this section. 
        \begin{proposition}\label{prop:Bertini computation}
            We use the notation in Definition \ref{def: decomposition of mock polytope}, that in Definition \ref{def: model polytope}, and that in Proposition \ref{prop: extend refinement along function}. 
            Let $(Z_\pi, S, u)$ be a mock polytope of $Z_\pi$ and $\kappa\colon S\rightarrow \ZZ$ be a map. 
            We assume that $Z$ is proper over $k$, $d_u\geq 2$, and  $(\Delta\times\Delta_!)_{u\oplus\kappa, \pi^1}$ is a strongly convex rational polyhedral fan. 
            Let $\Delta''$ denote $(\Delta\times\Delta_!)_{u\oplus\kappa, \pi^1}$. 
            We remark that $\Delta'' = \Delta''_{u\oplus\kappa}$ from Proposition \ref{prop: extend refinement along function}(h) and $\supp(\Delta'') = (\pi^1_\RR)^{-1}(\supp(\Delta\times\Delta_!))$ from Proposition \ref{prop: extend refinement along function}(f). 
            Moreover, we assume that there exists a refinement $\Delta'$ of $\Delta''$ such that $\Delta'$ is generically unimodular, specifically reduced, and compactly arranged. 
            Let $W'$ be a mock toric variety induced by $Z^1$ along $\pi^1_*\colon X(\Delta'')\rightarrow X(\Delta\times\Delta_!)$ and $s^1$. 
            Let $(a_i)_{i\in S}\in k^{|S|}$ be a family and $f$ denote the following polynomial in $k[Z^1_{\pi^1}]$: 
            \[
                f = \sum_{i\in S}a_iu(i)t^{\kappa(i)}
            \]
            Let $\sigma\in\Delta''$ be a cone and $f^\sigma$ denote the following polynomial in $k[W'_\sigma]$: 
            \[
                f^\sigma = \sum_{i\in S^\sigma}a_i(u\oplus\kappa)^\sigma(i)
            \]
            Let $\{E^{(j)}_\sigma\}_{1\leq j\leq r_\sigma}$ be all irreducible components of $H_{W'_\sigma, f^\sigma}$. 
            We remark that $r_\sigma$ is dependent on the choice of $f$. 
            Let $\{F_i\}_{1\leq i\leq r}$ be all irreducible components of $H^\circ_{W', f}\times_{{\Gm^1}_{,k}}\Spec(\Kt)$. 
            Then for general $(a_i)_{i\in S}\in k^{|S|}$, the following equation holds: 
            \[
                \sum_{1\leq i\leq r}\VOL\biggl(\bigl\{F_i\bigr\}_{\mathrm{sb}}\biggr) = \sum_{\substack{\sigma\in\Delta''_{\bdd}\\ 
                \sigma\in{\Delta''^{\geq 2}_{u\oplus\kappa}}}}(-1)^{\dim(\sigma)-1}\biggl(\sum_{1\leq j\leq r_\sigma} \bigl\{E^{(j)}_{\sigma}\bigr\}_{\mathrm{sb}}\biggr)
            \]
        \end{proposition}
        \begin{proof}
            We show that for general $(a_i)_{i\in S}\in k^{|S|}$, $f$ satisfies the conditions in Theorem \ref{thm: general computation}. 
            
            First, from Proposition \ref{prop: Bertini for fine}, for general $(a_i)_{i\in S}\in k^{|S|}$, we have $\val_{u\oplus\kappa} = \val_{f}$ and $\Delta''_{u\oplus\kappa} = \Delta''_f$. 
            In particular, we remark that $\Delta'' = \Delta''_f$. 

            Second, from Proposition \ref{prop:non-empty and Bertini}(f), for general $(a_i)_{i\in S}\in k^{|S|}$, $f$ is non-degenerate for $\Delta'' = \Delta''_{u\oplus\kappa}$. 
            In particular, from \cite[Proposition 6.14(e)]{Y23}, $f$ is non-degenerate for $\Delta'$ too. 

            Third, from Proposition \ref{prop: polytope and dominant}, for general $(a_i)_{i\in S}\in k^{|S|}$, every irreducible component of $H^\circ_{Z^1_{\pi^1}, f}$ dominates ${\Gm^1}_{,k}$.

            Thus, all conditions in Theorem \ref{thm: general computation} hold. 

            Moreover, We show that for general $(a_i)_{i\in S}\in k^{|S|}$, ${}_f\Delta'' = \Delta''^{\geq 2}_{u\oplus\kappa}$. 
            Indeed, we have already shown that for general $(a_i)_{i\in S}\in k^{|S|}$, $\Delta''_{u\oplus\kappa} = {}_f\Delta''$ in the proof of Proposition \ref{prop:non-empty and Bertini}(d). 

            We recall that $H_{W', f}\cap W'_\sigma = H^\circ_{W'_\sigma, f^\sigma}$ for any $\sigma\in\Delta''$ from \ref{lemma: polytope and stratification}. 
            Because $f$ is non-degenerate for $\Delta''$, any irreducible component of  $H_{W', f}\cap W'_\sigma$ is a connected component of $H_{W', f}\cap W'_\sigma$. 
            Therefore, the statement holds from Theorem \ref{thm: general computation}. 
        \end{proof}
        Now, while the formula has been obtained, grasping the stable rationality of each stratum in the strictly toroidal model is difficult. 
        However, when favorable conditions are established for the shape of the cone, rationality can be inferred. 
        We show it in Proposition\ref{prop:simple-slice-rational}. 
        This proposition is a generalization of \cite[Example 3.8]{NO22}. 
        The following lemma is needed in the proof of Proposition\ref{prop:simple-slice-rational}. 
        \begin{proposition}\label{prop:common-divisor}
            Let $Z$ be a mock toric variety and $(Z, S, [u])$ be an abstract mock polytope of $Z$. 
            Let $S_1$ and $S_2$ be a non-empty subset of $S$ such that $S_1\coprod S_2 = S$. 
            For $(a_i)_{i\in S}\in k^{|S|}$, $f_1$ and $f_2$ denote the following functions in $k[Z^\circ]$: 
            \begin{align*}
                f_1 &= \sum_{i\in S_1}a_iu(i)\\
                f_2 &= \sum_{i\in S_2}a_iu(i)\\
            \end{align*}
            Then for general $(a_i)_{i\in S}\in k^{|S|}$, $f_1$ and $f_2$ are co-prime. 
        \end{proposition}
        \begin{proof}
            Let $(b_i)_{i\in S_1}\in k^{|S_1|}$ be a vector such that $\sum_{i\in S_1}b_iu(i)\neq 0$ and $\theta$ denote the following linear morphism from $k^{|S_2|}$ to $k[Z^\circ]$:
            \[
                k^{|S_2|} \ni (c_i)_{i\in S_2} \mapsto\sum_{i\in S_2}c_iu(i)\in k[Z^\circ]
            \]
            Let $g$ be an irreducible component of $\sum_{i\in S_1}b_iu(i)$. 
            Then the following subset of $k^{|S_2|}$ is a linear subspace of $k^{|S_2|}$:
            \[
                \{(c_i)_{i\in S_2}\in k^{|S_2|} \mid \sum_{i\in S_2}c_iu(i)\in (g)\}
            \]
            From the definition of mock polytope, $(1, 0, \ldots, 0)_{i\in S_2}$ is not contained in the above linear subspace.
            Hence, the above linear subspace is a strict linear subspace of $k^{|S_2|}$. 
            This shows that for any $(b_i)_{i\in S_1}$ such that $\sum_{i\in S_1}b_iu(i)\neq 0$, $\sum_{i\in S_1}b_iu(i)$ and $\sum_{i\in S_2}c_iu(i)$ are co-prime for general $(c_i)_{i\in S_2}\in k^{|S_2|}$. 

            Let $X$ be a closed subscheme of $\A^{|S|}_k\times Z^\circ$ defined by the ideal which is generated by $f_1$ and $f_2\in k[\{a_i\}_{i\in S}]$, where $\{a_i\}_{i\in S}$ is a family of coordinate functions of $\A^{|S|}_k$. 
            Let $\pi$ denote the first projection $X\rightarrow \A^{|S|}_k$. 
            Then from generic flatness, there exists non-empty open subset $V$ of $\A^{|S|}_k$ such that $\pi|_{\pi^{-1}(V)}\colon \pi^{-1}(V)\rightarrow V$ is flat. 
            Thus, there exists $n$ and anon-empty open subset $V_0$ of $V$ such that $\{x\in V\mid \dim_{\kappa(x)}(\pi^{-1}(x)) = n\} = V_0$. 
            We claim that $n\leq\dim(Z)-2$ by a contradiction. 
            Because, $S_1$ and $S_2$ are non-empty, we may assume $n = \dim(Z)-1$. 
            Let $(b_i)_{i\in S_1}\in k^{|S_1|}$ be a vector such that $((b_i)_{i\in S_1}\times\A^{|S_2|})\cap V_0\neq \emptyset$ and $\sum_{i\in S_1}b_iu(i)\neq 0$. 
            Then for any $((b_i)_{i\in S_1}, (c_i)_{i\in S_2})\in V_0$, $\sum_{i\in S_1}b_iu(i)$ and $\sum_{i\in S_2}c_iu(i)$ are not co-prime from the assumption of $n$. 
            However, it is a contradiction of the argument at first. 
            Thus, $n\leq\dim(Z)-2$. 
            This inequality shows that the statement holds. 
        \end{proof}
        \begin{proposition}\label{prop:simple-slice-rational}
            We use the notation in Proposition \ref{prop: computation of reduction step}and in the following notation:
            \begin{itemize}
                \item There exists an isomorphisms $\epsilon\times p_0\colon Z^\pi\rightarrow Z\times T_{N_1}$ from Proposition \ref{prop: computation of reduction step}(c). 
                Then $\epsilon^*\otimes p^*_0\colon k[Z^\circ]\otimes_k k[M_1]\rightarrow k[Z_\pi]$ is an isomorphism. 
                We regard $k[Z^\circ]$ and $k[M_1]$ as subrings of $k[Z_\pi]$ by $\epsilon^*$ and $p^*_0$, so we omit $\epsilon^*$ and $p^*_0$.
                \item Let $\omega$ denote an element of $M_1$ such that $M_1/\ZZ\omega$ is torsion-free. 
                We fix a sub lattice $M_2$ of $M_1$ such that $M_2\oplus \ZZ\omega = M_1$. 
                \item Let $(Z_\pi, S, [u])$ be an abstract mock polytope of $Z_\pi$.  
                We assume that there exists the non-empty subsets $S_1$ and 
                $S_2$ of $S$ which satisfy the following three conditions:
                \begin{itemize}
                    \item[(1)] $S = S_1\coprod S_2$ 
                    \item[(2)] For any $i\in S_1$, there exists $v(i)\in k[Z^\circ]^*$ and $w(i)\in k[M_2]^*$ such that $u(i) = \chi^{\omega} v(i)w(i)$. 
                    \item[(3)] For any $i\in S_2$, there exists $v(i)\in k[Z^\circ]^*$ and $w(i)\in k[M_2]^*$ such that $u(i) = v(i)w(i)$. 
                \end{itemize}
                \item For $(a_i)_{i\in S}\in k^{|S|}$, let $f$ denote the following function in $k[Z_\pi]$:
                \[
                    f = \sum_{i\in S}a_iu(i)
                \]
            \end{itemize}
            Then for general $(a_i)_{i\in S}\in k^{|S|}$, $H_{Z_\pi, f}$ is rational. 
        \end{proposition}
        \begin{proof}
            Let $F_1$ denote the following function in $k[Z_\pi]$: 
            \[
                F_1= \sum_{i\in S_1}a_iv(i)w(i)
            \]
            Let $F_2$ denote the following function in $k[Z_\pi]$: 
            \[
                F_2= \sum_{i\in S_2}a_iv(i)w(i)
            \]
            We remark that $f = \chi^\omega F_1 + F_2$. 
            Hence, from Proposition \ref{prop:common-divisor}, $F_1$ and $F_2$ are co-prime for general $(a_i)_{i\in S}\in k^{|S|}$. 
            Thus, $H_{Z_\pi, f}$ is irreducible for general $(a_i)_{i\in S}\in k^{|S|}$ because $k[Z^\circ][M_2]$ is a UFD. 
            For such $(a_i)_{i\in S}\in k^{|S|}$, let $U$ denote an open subscheme $\Spec(k[Z^\circ][M_2]_{F_1})$ of $\Spec(k[Z_0][M_2])$. 
            Then $U$ and $H_{Z_\pi, f}$ are birational. 
            In particular, $\Spec(k[Z^\circ][M_2])$ is rational, so that $H_{Z_\pi, f}$ is rational. 
        \end{proof}
\section{Application: Stable rationality of hypersurfaces in Grassmaniann varieties}
		As an application of the considerations thus far, we can mention the stable rationality of some hypersurfaces in $\Gr_\C(2, n)$. 
		In this section, we use the following notation. 
    \begin{itemize}
        \item Let $n$ be a positive integer greater than 3. 
        Let $k$ be an uncountable algebraically closed field with $\charac(k) = 0$. 
        \item Let $I$ denote the following set:
        $$I = \{(i, j)\in\ZZ^2\mid 0\leq i\leq j\leq n-3\}$$
        \item Let $\{e^{i,j}\}_{(i, j)\in I}$ denote a canonical basis of $\ZZ^{|I|}$, and $\mathbf{1}\in \ZZ^{|I|}$ be $\sum_{(i, j)\in I}e^{i,j}$. 
        \item Let $N$ denote $\ZZ^{|I|}/\ZZ\mathbf{1}$, $\Pi$ denote the quotient morphism $\ZZ^{|I|}\rightarrow N$, and $e_{i,j}\in N$ denote $\Pi(e^{i,j})$ for $(i, j)\in I$. 
        \item Let $(\ZZ^{|I|})^\vee$ denote a dual lattice of $\ZZ^{|I|}$, and $\{\omega_{i, j}\}_{(i, j)\in I}$ denote a dual basis of $\{e^{i,j}\}_{(i, j)\in I}$. 
        \item Let $M$ denote a sub lattice of $(\ZZ^{|I|})^\vee$ as follows: 
        $$M = \{\sum_{(i, j)\in I}a_{i, j}\omega_{i, j}\in(\ZZ^{|I|})^\vee\mid \sum_{(i, j)\in I} a_{i, j} = 0\}$$
        We remark that we can regard $M$ as a dual lattice of $N$. 
        \item Let $N^\dagger$ be a lattice of rank $n - 1$, $\{e^\dagger_j\}_{-1\leq j\leq n-3}$ be a basis of $N^\dagger$. 
        Let $M^\dagger$ be a dual lattice of $N^\dagger$, $\{\eta_j\}_{-1\leq j\leq n-3}$ be a dual basis of $\{e^\dagger_j\}_{-1\leq j\leq n-3}$. 
        \item Let $\{Y_{i}\}_{0\leq i\leq n-3}$ be a homogeneous coordinate function of $\PP^{n-3}_k$. 
        Let $L_{i, j}\in\Gamma(\PP^{n-3}_k, \OO_{\PP^{n-3}_k}(1))$ be homogeneous functions as follows for $(i, j)\in I$: 
        \[
            L_{i,j} = \left\{
                \begin{array}{ll}
                    Y_i & (i = j)\\
                    Y_i - Y_j & (i < j)
                \end{array}
            \right.
        \]
        \item Let $\iota_0$ denote a closed immersion from $\PP^{n-3}_k$ to $\PP^{|I| - 1}_k$ defined by $\{L_{i, j}\}_{(i, j)\in I}$. 
        \item Let $\{X_{i,j}\}_{(i, j)\in I}$ be a homogeneous coordinate function of $\PP^{|I| - 1}_k$.
        \item Let $\mathcal{B}$ denote $\{L_{i, j}\}_{(i, j)\in I}$. 
        We remark that $\mathcal{B}$ generates $\Gamma(\PP^{n-3}_k, \OO(1))$. 
        Let $\mathcal{V}$, $\mathcal{C}$, $\Delta(\mathcal{B})$, $Z$, $(N, \Delta(\mathcal{B}), \iota, \Phi, \{N_\varphi\}_{\phi\in\Phi}, \{\Delta(\mathcal{B})_\varphi\}_{\phi\in\Phi})$ be data defined as those in \cite[Section 5]{Y23} from the above notation. 
        \item Let $H\in k[y_1, \ldots, y_{n-3}]$ be a polynomial defined as follows: 
        $$H = \prod_{1\leq j\leq n-3}y_j \prod_{1\leq j\leq n-3}(1- y_j)\prod_{1\leq i<j\leq n-3}(y_j - y_i)$$
        \item Let $o$ denote a ringmorphism from $k[y_1, \ldots, y_{n-3}]_H$ to $k[Z^\circ]$ defined as $o(y_j) = \frac{Y_j}{Y_0}$ for any $0\leq j\leq n-3$. 
        We can check that $o$ is isomorphic. 
        From now, we regard $k[Z^\circ]$ as $k[y_1, \ldots, y_{n-3}]_H$. 
        We remark that for any $(i, j)\in I$, the following equation holds:
        \[
            i^*(\frac{X_{i, j}}{X_{0. 0}}) = \left\{
                \begin{array}{ll}
                    y_i & (i = j)\\
                    y_i - y_j & (i < j)
                \end{array}
            \right.
        \]
        where $y_0 = 1$. 
        \item Let $e_t$ be a generator of $\ZZ$ and $\delta$ be a dual basis of $\ZZ^\vee$. 
        \item Let $\frac{X_{i, j}}{X_{l, m}}\in k[M]$ denote a torus invariant mononomial associated with $\omega_{i, j} - \omega_{l, m}\in M$, $x_j\in k[M^\dagger]$ denote a torus invariant mononomial associated with $\eta_j\in M^\dagger$, and $t\in k[\ZZ^\vee]$ is a torus invariant mononomial associated with $\delta\in\ZZ^\vee$. 
        \item Let $N'$ denote $N\oplus N^\dagger$ and $\pi$ denote a first projection $N'\rightarrow N$. 
        Let $s$ denote the section of $\pi$ such that $s(v) = (v, 0)$ for any $v\in N$.
        \item Let $J$ denote $\{(i, j)\in\ZZ^2\mid 0\leq i<j\leq n-1\}$. 
        
        \item Let $M_k(n, 2)$ denote a set of all $n\times 2$-matrices over $k$. 
        A set $M_k(n, 2)$ has a natural $\mathrm{GL}_k(2)$-action, and we can identify with $M_k(n, 2)/\mathrm{GL}_k(2)$ and $\Gr_k(2, n)(k)$. 
        \item For $A = (a_{i, j})_{0\leq i\leq n-1, 0\leq j\leq 1}\in M_k(n, 2)$, and integers $(i_1, i_2)\in J$, let $d_{i_1, i_2}(A)$ denote $a_{i_1, 0}a_{i_2, 1} - a_{i_1, 1}a_{i_2, 0}$. 
        \item Let $\{W_{i_1, i_2}\}_{(i_1, i_2)\in J}$ be a homogeneous coordinate function of $\PP^{|J|-1}_k$. 
        \item Let $\mathrm{Pl}$ denote the following Pl\"{u}cker embedding:
            \begin{align*}
                &\mathrm{Pl}\colon \Gr_k(2, n)\ni [A = (a_{i, j})_{0\leq i\leq n-1, 0\leq j\leq 1}]\mapsto[d_{i, j}(A)]\in \PP^{|J| - 1}_k
            \end{align*}
        \item Let $U^{i, j}$ denote the affine open subset of $\PP^{|J|-1}_k$ defined by $W_{i, j}\neq 0$ for $(i, j)\in J$, $U$ denote open subset $\Gr_k(2, n)\cap U^{0, 1}$ of $\Gr_k(2, n)$, and $\Gr^\circ_k(2, n)$ denote  open subset $\Gr_k(2, n)\cap \bigcap_{(i, j)\in J}U^{i, j}$ of $\Gr_k(2, n)$. 
        We remark that we can identify with $U$ and $\A^{2(n-2)}_k$ as follows: 
        \begin{multline*}
            M_k(n, 2)/\mathrm{GL}_k(2)\supset U(k) \ni
            \biggl[\begin{pmatrix}
                1 & 0 \\
                0 & 1 \\
                u_0 & v_0 \\
                u_1 & v_1 \\
                \vdots & \vdots\\
                u_{n-3} & v_{n-3} \\
            \end{pmatrix}
            \biggr]\\\longleftrightarrow
            (u_0, v_0, u_1, v_1, \ldots, u_{n-3}, v_{n-3})\in \A^{2(n-2)}_k(k)
        \end{multline*}  
        Let $\xi$ denote the closed immersion $\A^{2(n-2)}_k\hookrightarrow U^{0, 1}$ defined by the above identification with $U$ and $\A^{2(n-2)}_k$. 
        We regard $\{u_0, v_0, \ldots, u_{n-3}, v_{n-3}\}$ as coordinate functions of $\A^{2(n-2)}_k$. 
        Let $\xi^*$ denote a subjective ring morphism from $k[\{\frac{W_{i, j}}{W_{0, 1}}\}_{(i, j)\in J}]$ to $k[u_0, v_0, \ldots, u_{n-3}, v_{n-3}]$ induced from $\xi$. 
        \item We define $\{f_{i, j}\}_{(i, j)\in J}\subset k[u_0, v_0, \ldots, u_{n-3}, v_{n-3}]$ as follows: 
        \[
            f_{i,j} = \left\{
                \begin{array}{ll}
                    1 & ((i,j) = (0, 1))\\
                    v_{j - 2} & (i = 0, j > 1)\\
                    -u_{j - 2} & (i = 1)\\
                    u_{i-2}v_{j-2} - u_{j-2}v_{i-2} & (i > 1)
                \end{array}
            \right.
        \]
        We remark that $\xi^*(\frac{W_{i, j}}{W_{0, 1}}) = f_{i, j}$ for any $(i, j)\in J$ from the definition of $\mathrm{Pl}$. 
        \item Let $F\in k[u_0, v_0, \ldots, u_{n-3}, v_{n-3}]$ be a polynomial defined as $F = \prod_{(i, j)\in J}f_{i, j}$.
        We remark that the inclusion morphism $k[u_0, v_0, \ldots, u_{n-3}, v_{n-3}]\rightarrow k[u_0, v_0, \ldots, u_{n-3}, v_{n-3}]_F$ induces an open immersion $\Gr^\circ_k(2, n)\hookrightarrow U$. 
        \item Let $\zeta$ denote a ring morphism from $k[M^\dagger]\otimes_k k[Z^\circ]\rightarrow k[u_0, v_0, \ldots, u_{n-3}, v_{n-3}]_F$ as follows:
        \begin{align*}
            \zeta(x_{-1}) &= v_0u^{-1}_0\\
            \zeta(x_j) &= u_j &(0\leq j\leq n-3)\\
            \zeta(y_j) &= u_0v^{-1}_0u^{-1}_jv_j &(1\leq j\leq n-3)
        \end{align*}
        We can check that $\zeta$ is well-defined and isomorphic. 
        \item For $(i, j)\in J$, let $s_{i, j}\in k[M^\dagger]\otimes_k k[Z^\circ]$ be elements as follows: 
        \[
            s_{i,j} = \left\{
                \begin{array}{ll}
                    1 & ((i,j) = (0, 1))\\
                    x_{-1}x_0 & ((i, j) = (0, 2))\\
                    x_{-1}x_{j-2}y_{j-2} & (i = 0, j > 2)\\
                    -x_{j-2} & (i = 1)\\
                    x_{-1}x_{0}x_{j - 2}(y_{j - 2}- 1) &(i = 2)\\
                    x_{-1}x_{i - 2}x_{j - 2}(y_{j - 2}- y_{i - 2}) &(i > 2)\\
                \end{array}
            \right.
        \]
        We can check that $\zeta(s_{i, j}) = f_{i, j}$ and $s_{i, j}$ is a unit of $k[M^\dagger]\otimes_k k[Z^\circ]$ for any $(i, j)\in J$. 
        \item Let $d\geq 2$ be a positive integer. 
        Let $S_{J, d}$ denote the following set: 
        \[
            S_{J, d} = \{\alpha = (\alpha_{i, j})\in\ZZ^{|J|}_{\geq 0}\mid \sum_{(i, j)\in J}\alpha_{i, j} = d\}
        \]
        \item For $\alpha\in S_{J, d}$, let $s_\alpha$ denote $\prod_{(i, j)\in J} s^{\alpha_{i, j}}_{i, j}\in k[M^\dagger]\otimes k[Z^\circ]$. 
        \item Let $X$ be a smooth closed sub variety of $\PP^{|J| - 1}_k$. 
        \item For $\alpha\in S_{J, d}$, let $\mathbb{W}^\alpha$ denote $\prod_{(i, j)\in J}W_{i, j}^{\alpha_{i, j}}$. 
        \item Let $\{t_\alpha\}_{\alpha\in S_{J, d}}$ denote coordinate functions of $\A^{|S_{J, d}|}_k$, $\mathscr{H}^n_d$ denote a closed subvariety of $\PP^{|J| - 1}_k\times \A^{|S_{J, d}|}_k$ defined by $\sum_{\alpha\in S_{J, d}}t_\alpha\mathbb{W}^\alpha = 0$, and $\mathscr{X}_d$ denote a closed subscheme $X\times \A^{|S_{J, d}|}_k\cap \mathscr{H}^n_d$ of $\PP^{|J| - 1}_k\times \A^{|S_{J, d}|}_k$. 
        \item Let $\theta$ be a composition of a closed immersion $\mathscr{X}_d\hookrightarrow \PP^{|J| - 1}_k\times \A^{|S_{J, d}|}_k$ and the second projection $\PP^{|J| - 1}_k\times \A^{|S_{J, d}|}_k\rightarrow \A^{|S_{J, d}|}_k$. 
        \item Let $u\colon S_{J, d}\rightarrow k[M^\dagger]\otimes_k k[Z^\circ]$ be a map such that $u(\alpha) = s_\alpha$  for $\alpha\in S_{J, d}$. 
        Then $(Z_{\pi}, S_{J, d}, u)$ be a mock polytope of $Z_{\pi}$.
        \item Let $J_0, J_1$, and $J_2$ denote the following subset of $J$: 
        \begin{align*}
            J_0 &= \{(0, 1)\}\\
            J_1 &= \{(i, j)\in J\mid i < 2, j > 1\}\\
            J_2 &= \{(i, j)\in J\mid i > 1\}
        \end{align*}
        We remark that $J = \coprod_{0\leq i\leq 2} J_i$. 
        \item For $\alpha\in S_{J, d}$, we define an integer $c_0(\alpha), c_1(\alpha)$, and $c_2(\alpha)$ as follows:
            \begin{align*}
                c_0(\alpha) &= \alpha_{0, 1}\\
                c_1(\alpha) &= \sum_{(i, j)\in J_1}\alpha_{i, j}\\
                c_2(\alpha) &= \sum_{(i, j)\in J_2}\alpha_{i, j}\\
            \end{align*}
        \item Let $\kappa$ denote a map $S_{J, d}\rightarrow \ZZ$ defined as follows: 
             \[
                \kappa(\alpha) = \left\{
                    \begin{array}{ll}
                        0 & (c_1(\alpha) = d)\\
                        2(d - c_1(\alpha)) - 1 & (c_1(\alpha) < d)\\
                    \end{array}
                \right.
             \]
        \item For non-negative integers $a, b$, and $c$, let $S_{d_0, d_1, d_2}$ denote the following subset of $S_{J, d}$: 
            \[
                S_{d_0, d_1, d_2} = \{\alpha\in S_{J, d}\mid (c_0(\alpha), c_1(\alpha), c_2(\alpha)) = (d_0, d_1, d_2)\}
            \]
        \item Let $\pi^1$ denote the morphism $N'\oplus\ZZ\rightarrow N\oplus\ZZ$ defined by $\pi^1 = (\pi, \id_\ZZ)$. Let $s^1$ denote the section of $\pi^1$ such that $s^1((v, r)) = (s(v), r)$ for  $(v, r)\in N\oplus\ZZ$. 
        \item Let $j^\circ$ denote the section of $\pr_1\colon N\oplus\ZZ\rightarrow N$ such that $\pr_2\circ j^\circ\equiv 0$. 
        \item Let $j'$ denote the section of $\pr_1\colon N'\oplus\ZZ\rightarrow N'$ such that $\pr_2\circ j'\equiv 0$. 
        \item Let $Z^1$ be a mock toric variety induced by $Z$ along $(\pr_1)_*\colon X(\Delta(\mathcal{B})\times\Delta_!)\rightarrow X(\Delta(\mathcal{B}))$ and $j^\circ$. 
        \item Let $\Kt$ denote the field of the puise\"{u}x functions over $\C$. 
    \end{itemize}
    Fundamentally, the following proposition shows the reason why finding one nonstably rational hypersurface over $\mathcal{K}$ implies the existence of a very general nonstably rational hypersurface over $k$.

    \begin{proposition}\label{prop: Bertini-stably rational}
        The following statements follow:
        \begin{enumerate}
            \item[(a)] There exists a non-empty open subset $U_{X, d}$ of $\A^{|S_{J, d}|}_k$ such that the restriction $\theta|_{\theta^{-1}(U_{X, d})}\colon \theta^{-1}(U_{X, d})\rightarrow U_{X, d}$ is smooth and projective. 
            \item[(b)] Let $K/k$ be a field extension. 
            We assume that $K$ is algebraically closed. 
            Let $x \in (U_{X, d})_K(K)$ be a point. 
            We assume that $(\mathscr{X}_{d})_{x}$ is not stably rational over $K$. 
            Then, a very general hypersurface of degree $d$ in $X$ is not stably rational over $k$. 
        \end{enumerate}
    \end{proposition}
    \begin{proof}
         \item[(a)] From Bertini's theorem, the statement holds. 
        \item[(b)] For any $y\in U_{X, d}$, we fix a geometric point $\overline{y}$ of $y$. 
        Let $A$ denote the following subset of $U_{X, d}$. 
        \[
            A = \{y\in U_{X, d}\mid (\mathscr{X}_d)_{\overline{y}} \mathrm{\ is\ stably \ rational.}\}
        \]
        Let $h\colon (U_{X, d})_K\rightarrow U_{X, d}$ be a canonical map. 
        Then $h(x)\notin A$  because $K$ is algebraically closed. 
        Thus, from \cite[Cor.4.1.2]{NO21}, $A$ is countable unions of strict closed subsets of $U_{X, d}$. 
        Because $k$ is an uncountable field, a very general hypersurface of degree $d$ in $X$ is not stably rational over $k$.         
    \end{proof}
    We have not found a mock toric structure of $\Gr_\C(2, n)$. 
    Instead, we substitute the compactification of $T_{N^\dagger}\times Z^\circ$.  
    Moreover, we see the relationship with hypersurfaces in $\Gr_\C(2, n)$ and some hypersurfaces in this mock toric variety.

    \begin{proposition}\label{prop:grassman-mock}
        Let $D = (d_\alpha)_{\alpha\in S_{J, d}}\in k^{|S_{J, d}|}$ be a vector, $X_D\subset \PP^{|J| - 1}_k$ be the scheme theoretic intersection of $\Gr_k(2, n)$ and a hypersurface of degree $d$ defined by $\sum_{\alpha\in S_{J, d}}d_\alpha\mathbb{W}^\alpha = 0$, and $s_D\in k[M^\dagger]\otimes_k k[Z^\circ]$ be the following polynomial: 
        \[
            s_D = \sum_{\alpha\in S_{J, d}}d_\alpha s_\alpha
        \]
        Let $Y_D$ be a closed subscheme of $T_{N^\dagger}\times Z^\circ$ defined by $s_D = 0$.
        Then for general $D\in k^{|S_{J, d}|}$, $X_D$ and $Y_D$ are irreducible and birational. 
        \end{proposition}
        \begin{proof}  
           Let $\Theta$ denote the following composition of morphisms: 
           \[
                T_{N^\dagger}\times Z^\circ\xrightarrow{\sim} \Spec(k[u_0, v_0, \ldots, u_{n-3}, v_{n-3}]_F)\xrightarrow{\sim}\Gr^\circ_k(2, n)\xrightarrow{\mathrm{Pl}}\PP^{|J| - 1}_k
           \]
           where first morphism is induced by $\zeta^{-1}$. 
           We can check that $\Theta^*(\frac{W_{i, j}}{W_{0, 1}}) = s_{i, j}$ for any $(i, j)\in J$ be the definition of $\xi^*$ and $\zeta$. 
           In particular, $\Theta^*(\frac{\mathbb{W}^\alpha}{W^d_{0, 1}}) = s_\alpha$ for any $\alpha\in S_{J, \alpha}$. 
           We recall that $\Gr^\circ_k(2, n)$ is an open subscheme of $\Gr_k(2, n)$. 
           Thus, for a general $D$, $X_D$ and $Y_D$ are irreducible and birational. 
        \end{proof} 
	Now, we proceed with the calculations using the discussions from Section 2. 
	
    \begin{proposition}\label{prop: computation result} 
        Let $S$ denote $S_{J, d}$,  $(a_\alpha)_{\alpha\in S}\in k^{|S|}$ be a vector, $l\in\ZZ_{>0}$ be a positive integer, and $f$ be the following function in $k[(Z^1)_{\pi^1}]$: 
        \[
            f = \sum_{\alpha\in S}a_\alpha(u\oplus l\kappa)(\alpha)
        \]
        Then the following statements hold:
        \begin{enumerate}
            \item[(a)] The rational polyhedral convex fan $(\Delta(\mathcal{B})\times\Delta_!)_{u\oplus\kappa, \pi^1}$ is strongly convex. 
            \item[(b)] Let $\Delta'_{(l)}$ denote $(\Delta(\mathcal{B})\times\Delta_!)_{u\oplus l\kappa, \pi^1}$. 
            Then $(\Delta'_{(l)})_{\bdd}$ consists of the following 7 cones: 
            \begin{align*}
                \tau_0 &= \RR_{\geq 0}(e_t -2l\sum_{0\leq j\leq n-3}e^\dagger_j)\\
                \tau_1 &= \RR_{\geq 0}(e_t -l\sum_{0\leq j\leq n-3}e^\dagger_j)\\
                \tau_2 &= \RR_{\geq 0}(e_t +l\sum_{0\leq j\leq n-3}e^\dagger_j)\\
                \tau_3 &= \RR_{\geq 0}(e_t +2l\sum_{0\leq j\leq n-3}e^\dagger_j)\\
                \sigma_0 &= \tau_0 + \tau_1\\                    \sigma_1 &= \tau_1 + \tau_2\\
                \sigma_2 &= \tau_2 + \tau_3
            \end{align*}
            \item[(c)] The following equations hold: 
            \begin{align*}
                S^{\tau_0} &= \bigcup_{1\leq i\leq d}S_{0, d-i, i}\\
                S^{\tau_1} &= S_{0, d-1, 1} \cup S_{0, d, 0}\\
                S^{\tau_2} &= S_{0, d, 0} \cup S_{1, d-1, 0}\\
                S^{\tau_3} &= \bigcup_{1\leq i\leq d}S_{i, d-i, 0}\\
                S^{\sigma_0} &= S_{0, d-1, 1}\\
                S^{\sigma_1} &= S_{0, d, 0}\\
                S^{\sigma_2} &= S_{1, d-1, 0}
            \end{align*}
            \item[(d)] Let $W_{(l)}$ denote the mock toric variety induced by $Z^1$ along $(\pi^1)_*\colon X(\Delta'_{(l)})\rightarrow X(\Delta(\mathcal{B})\times\Delta_!)$ and $s^1$. 
            Then $H_{W_{(l)}, f}\cap (W_{(l)})_{\sigma_1}$ is birational to a general hypersurface of degree $d$ in $\PP^{2n-5}_k$ for general $(a_\alpha)_{\alpha\in S}\in k^{|S|}$. 
            \item[(e)] For general $(a_\alpha)_{\alpha\in S}\in k^{|S|}$, both $H_{W_{(l)}, f}\cap (W_{(l)})_{\tau_1}$ and $H_{W_{(l)}, f}\cap (W_{(l)})_{\tau_2}$ are irreducible and rational. 
            \end{enumerate}
        \end{proposition}
        \begin{proof}
            We prove the proposition from (a) to (e) in order:
            \begin{enumerate}
                \item[(a)] 
                Let $\sigma = \{0\}\in\Delta(\mathcal{B})\times\Delta_!$ be a cone. 
                From the definition of $(\Delta(\mathcal{B})\times\Delta_!)_{u\oplus\kappa, \pi^1}$, the minimal cone of $(\Delta(\mathcal{B})\times\Delta_!)_{u\oplus\kappa, \pi^1}$ is contained in ($\Delta(\mathcal{B})\times\Delta_!)_{u\oplus\kappa, \pi^1}^\sigma$. 
                Thus, from Proposition \ref{prop: actual-configuration}(h), it is enough to show that $D^\sigma_{u\oplus\kappa, \pi^1}$ is full cone in $(M'\oplus\ZZ^\vee)_\RR$. 
                For $(i, j)\in J$, let denote $\varpi_{i, j}\in M' = M\oplus M^\dagger$ as follows: 
                \[
                \varpi_{i, j} = \left\{
                    \begin{array}{ll}
                        0 & ((i, j) = (0, 1))\\
                        \eta_{-1} + \eta_0 & ((i, j) = (0, 2))\\
                        \eta_{-1} + \eta_{j - 2}  + (\omega_{j-2, j-2} - \omega_{0, 0}) & (i = 0, j > 2)\\
                        \eta_{j - 2} & (i = 1)\\
                        \eta_{-1} + \eta_{i - 2} + \eta_{j - 2} + (\omega_{i - 2, j-2} - \omega_{0, 0}) & (i > 1)
                    \end{array}
                \right.
             \]
            We remark that $s_{i, j} = \pm{\iota'}^*(\chi^{\varpi_{i, j}})$ for any $(i, j)\in J$. 
            For $\alpha\in S$, let $\varpi_{\alpha}\in M'$ denote as follows:
            \[
                \varpi_{\alpha} = \sum_{(i, j)\in J}\alpha_{i, j}\varpi_{i, j}
            \]
            Thus, from Corollary \ref{lem: supplement-of-valuation3}, $D^\sigma_{u\oplus\kappa, \pi^1}$ is generated by $\{(\omega, 0)\mid \omega\in M\oplus\ZZ^\vee\}\cup\{(\varpi_\alpha + \kappa(\alpha)\delta, 1)\mid \alpha\in S\}$. 
            From now, we will show that $\langle D^\sigma_{u\oplus\kappa, \pi^1}\rangle = (M'\oplus\ZZ^\vee\oplus\ZZ)_\RR$. 
            Let $\alpha_0$ denote $((0, 1), \ldots, (0, 1))\in S$. 
            Then $(\varpi_{\alpha_0}+\kappa(\alpha_0)\delta, 1)=((2d -1)\delta, 1)\in D^\sigma_{u\oplus\kappa, \pi^1}$. 
            This shows that $(0, 1)\in \langle D^\sigma_{u\oplus\kappa, \pi^1}\rangle$. 
            Similary, for any $0\leq j\leq n - 3$, we can check that $(d\eta_{j}, 1)\in D^\sigma_{u\oplus\kappa, \pi^1}$. 
            In particular, $(\eta_j, 0)\in \langle D^\sigma_{u\oplus\kappa, \pi^1}\rangle$ for any $0\leq j\leq n-3$. 
            Finally, we can check that $(d\eta_{-1}+d\eta_0, 1)\in D^\sigma_{u\oplus\kappa, \pi^1}$. 
            Therefore, $(\eta_{-1}, 0)\in \langle D^\sigma_{u\oplus\kappa, \pi^1}\rangle$. 
            Thus, these shows that $\langle D^\sigma_{u\oplus\kappa, \pi^1}\rangle = (M'\oplus\ZZ^\vee\oplus\ZZ)_\RR$.
            \item[(b)] We may assume that $l = 1$. 
            Let $\tau\in (\Delta'_{(1)})_{\bdd}$ be a cone. 
            From Proposition \ref{prop: actual-configuration}(f), there exists $\sigma\in\Delta(\mathcal{B})\times\Delta_!$ and $\tau'\prec C^\sigma_{u\oplus\kappa, \pi^1}$ such that $\pr_{1, \RR}(\tau') = \tau$ and $(0, 1)\notin\tau'$. 
            From the definition of $\sigma\in\Delta(\mathcal{B})\times\Delta_!$, there exists $c\in \mathcal{C}$ such that $\sigma = \sigma_c\times\{0\}$ or $\sigma_c\times[0, \infty)$. 
            Because $\tau\subset (\pi^1_\RR)^{-1}(\sigma)$ and $\tau\in (\Delta'_{(1)})_{\bdd}$, we may assume that $\sigma = \sigma_c\times[0, \infty)$. 
            Let $\gamma\prec D^\sigma_{u\oplus\kappa, \pi^1}$ be the dual face of $\tau'$, i.e. $\gamma = (\tau')^\perp\cap D^{\sigma}_{u\oplus\kappa, \pi^1}$. 
            In particular, $\tau' = \gamma^\perp\cap C^\sigma_{u\oplus\kappa, \pi^1}$. 

            Now, we will classify all such $\gamma$ from Step. 1 to Step. 9. 

            Step. 1. In this step, we show that $(\delta, 0)\notin\gamma$ and there exists $\alpha\in S$, such $(\varpi_\alpha+\kappa(\alpha)\delta, 1)\in \gamma$. 
            Because $\tau\not\subset M'_\RR\times\{0_{(\ZZ^\vee)_\RR}\}$, $(\delta, 0)\notin \gamma$. 
            In particular, for any $\omega\in\sigma^\vee_c$ and $r>0$, $(\omega+r\delta, 0)\notin\gamma$ because $(\omega, 0), (\delta, 0)\in D^\sigma_{u\oplus\kappa, \pi^1}$ and $\gamma\prec D^\sigma_{u\oplus\kappa, \pi^1}$. 
            Because $\gamma\prec D^{\sigma}_{u\oplus\kappa, \pi^1}$, $\gamma$ is generated by some generaters of $D^{\sigma}_{u\oplus\kappa, \pi^1}$. 
            If $\gamma$ doesn't contain $(\varpi_\alpha+\kappa(\alpha)\delta, 1)$ for any $\alpha\in S$, then $\gamma\subset (M'\oplus\ZZ^\vee)_\RR\times\{0\}$. 
            This inclusion indicates $(0, 1)\in\tau'$, and it is a contradiction. 

            Step. 2. Let $(i, j)\in J$ be an element with $i > 1$. 
            In this step, we show that there exists $\omega_0\in \sigma^\vee_c\cap M$ such that either following equation holds in $M'$: 
            \begin{align*}
                 \varpi_{0, 1} + \varpi_{i, j} &= \varpi_{0, i} + \varpi_{1, j} + \omega_0\\
                 \varpi_{0, 1} + \varpi_{i, j} &= \varpi_{0, j} + \varpi_{1, i} + \omega_0
            \end{align*}
            We write down $c$ explicitly as $c = (V_1, V_2, \ldots, V_s)$. 
            Let $V_{s+1}$ denote $\Gamma(\PP^{n-3}_k, \OO(1))$. 
            Let $k_1, k_2,$ and $k_3$ be positive integers such that $L_{i-2, i-2}\in V_{k_1}, L_{j-2, j-2}\in V_{k_2}, $and $L_{i-2, j-2}\in V_{k_3}$. 
            Because $L_{i-2, j-2} = L_{i-2, i-2} - L_{j-2, j-2}$, $k_3\leq \max\{k_1, k_2\}$. 
            If $k_3\leq k_1$, then $\omega_{i-2, j-2} - \omega_{i-2, i-2}\in \sigma^\vee_c$. 
            Similary, if $k_3\leq k_2$, then $\omega_{i-2, j-2} - \omega_{j-2, j-2}\in \sigma^\vee_c$. 
            Therefore, from the definition of $\varpi_{\cdot, \cdot}$, there exists $\omega_0\in\sigma^\vee_c\cap M$ such that the either equation above holds. 

            Step. 3. Let $d_0, d_1,$ and $d_2$ be integers and $\alpha\in S_{d_0, d_1, d_2}$ be an element. 
            We assume that $d_0d_2 > 0$. 
            In this step, We show that $(\varpi_{\alpha} + \kappa(\alpha)\delta, 1)\notin \gamma$ holds. 
            Indeed, from the assumption of $d_0$ and $d_2$ and Step. 2, there exists $\beta\in S_{(d_0 - 1, d_1 + 2, d_0 -1)}$, a positive integer $m$ and $\omega_0\in\sigma^\vee_c$ such that $(\varpi_\alpha + \kappa(\alpha)\delta, 1) = (\varpi_\beta + \kappa(\beta)\delta, 1) + (m\delta, 0) + (\omega_0, 0)$. 
            From the definition of $D^\sigma_{u\oplus\kappa, \pi^1}$, $(\varpi_\beta + \kappa(\beta)\delta, 1), (m\delta, 0), (\omega_0, 0)\in D^\sigma_{u\oplus\kappa, \pi^1}$. 
            Because $\gamma\prec D^\sigma_{u\oplus\kappa, \pi^1}$, $(\delta, 0)\in\gamma$. 
            However, it is a contradiction to Step. 1. 

            Step. 4. Let $d_0, d_2$ be positive integers, $\alpha\in S_{(d_0, d-d_0, 0)}$ and $\alpha'\in S_{(0, d-d_2, d_2)}$ be elements. 
            In this step, we show that $\{(\varpi_\alpha + \kappa(\alpha)\delta, 1), (\varpi_{\alpha'} + \kappa(\alpha')\delta, 1)\}\not\subset \gamma$. 
            There exist $\beta\in S_{(d_0 - 1, d - d_0, 1)}$, $\beta'\in S_{(1, d - d_2, d_2 - 1)}$ such that $\varpi_{\alpha} + \varpi_{\alpha'} = \varpi_{\beta} + \varpi_{\beta'}$ holds. 
            We remark that $\kappa(\alpha) = \kappa(\beta)$ and $\kappa(\alpha') = \kappa(\beta')$ holds. 
            Thus, if $\{(\varpi_\alpha + \kappa(\alpha)\delta, 1), (\varpi_{\alpha'} + \kappa(\alpha')\delta, 1)\}\subset \gamma$, then $(\varpi_\beta + \kappa(\beta)\delta, 1)\in \gamma$. 
            It is a contradiction to Step. 3. 

            Step. 5. Let $i > 1$ be an integer, $\alpha\in S_{0, d, 0}$ and $\alpha'\in S_{i, d - i, 0}$. 
            In this step, we show that $\{(\varpi_\alpha + \kappa(\alpha)\delta, 1), (\varpi_{\alpha'} + \kappa(\alpha')\delta, 1)\}\not\subset \gamma$. 
            There exist $\beta\in S_{(1, d - 1, 0)}$, $\beta'\in S_{(i - 1, d - i + 1, 0)}$, such that $\varpi_{\alpha} + \varpi_{\alpha'} = \varpi_{\beta} + \varpi_{\beta'}$ holds. 
            We remark that $\kappa(\alpha) + \kappa(\alpha') =\kappa(\beta) + \kappa(\beta') + 1$ holds. 
            Thus, if $\{(\varpi_\alpha + \kappa(\alpha)\delta, 1), (\varpi_{\alpha'} + \kappa(\alpha')\delta, 1)\}\subset \gamma$, we have that $(\delta, 0)\in \gamma$. 
            It is a contradiction to Step. 1. 

            Step. 6. Let $i > 1$ be an integer, $\alpha\in S_{0, d, 0}$ and $\alpha'\in S_{0, d - i, i}$. 
            In this step, we show that $\{(\varpi_\alpha + \kappa(\alpha)\delta, 1), (\varpi_{\alpha'} + \kappa(\alpha')\delta, 1)\}\not\subset \gamma$. 
            Indeed, we can check it as Step. 5. 

            Step. 7. Let $\tau'_0, \tau'_1, \tau'_2$, and $\tau'_3$ be cones in $(N'\oplus\ZZ\oplus\ZZ)_\RR$ as follows: 
            \begin{align*}
                \tau'_0 &= \RR_{\geq 0}(e_t -2\sum_{0\leq j\leq n-3}e^\dagger_j, 2d + 1)\\
                \tau'_1 &= \RR_{\geq 0}(e_t -\sum_{0\leq j\leq n-3}e^\dagger_j, d)\\
                \tau'_2 &= \RR_{\geq 0}(e_t +\sum_{0\leq j\leq n-3}e^\dagger_j, -d)\\
                \tau'_3 &= \RR_{\geq 0}(e_t +2\sum_{0\leq j\leq n-3}e^\dagger_j, -2d + 1)\\
            \end{align*}
            In this step, we show that the above cones are rays of $C^\sigma_{u\oplus\kappa, \pi^1}$. 
            We can check that the above cones are contained in $C^\sigma_{u\oplus\kappa, \pi^1}$. 
            It is enough to show that $(\tau'_j)^\perp \cap D^\sigma_{u\oplus\kappa, \pi^1}$ is a facet of $D^\sigma_{u\oplus\kappa, \pi^1}$.
            Let $\gamma_j$ denote $(\tau'_j)^\perp\cap D^\sigma_{u\oplus\kappa, \pi^1}$ for $0\leq j\leq 3$. 
            We can check that $\{(\omega, 0)|\omega\in M_\RR\}\subset\langle\gamma_j\rangle$ for any $0\leq j\leq 3$ because of the strong convexity of $\sigma_c$ and the definition of $\tau'_j$. 
            For $0\leq j\leq 3$, let $S(j)$ denote $\{\alpha\in S\mid (\varpi_\alpha+\kappa(\alpha)\delta, 1)\in\gamma_j\}$. 
            We can check the following equations:
                \begin{align*}
                    S(0) &= \bigcup_{1\leq i\leq d}S_{0, d-i, i}\\
                    S(1) &= S_{0, d-1, 1} \cup S_{0, d, 0}\\
                    S(2) &= S_{0, d, 0} \cup S_{1, d-1, 0}\\
                    S(3) &= \bigcup_{1\leq i\leq d}S_{i, d-i, 0}\\
                \end{align*}
            Thus, as the proof of (a), we can check that $\langle\gamma_j\rangle$ is a codimension $1$ linear subspace of $(M'\oplus\ZZ^\vee\oplus\ZZ)_\RR$. 

            Step. 8. We assume that $\tau'$ is a ray of $C^\sigma_{u\oplus\kappa, \pi^1}$. 
            In this step, we show that there exists $0\leq j\leq 3$ such that $\tau'= \tau'_j$. 
            Let $S^*$ denote $\{\alpha\in S\mid (\varpi_\alpha+\kappa(\alpha)\delta, 1)\in\gamma\}$. 
            Then, from Step. 3 to Step. 7, there exists $0\leq j\leq 3$ such that $S^*\subset S(j)$. 
            Moreover, we can check that $\{(\omega, 0)\mid \omega\in\sigma^\vee_c\}\subset \gamma_j$ for any $0\leq j\leq 3$. 
            Furthermore, $(\omega+r\delta, 0)\notin\gamma$ for any $\omega\in \sigma^\vee_c$ and $r > 0$ from Step. 1. 
            Thus, $\gamma\subset\gamma_j$. 
            Because $\tau'$ is a ray of $C^\sigma_{u\oplus\kappa, \pi^1}$, $\gamma = \gamma_j$. 
            In particular, $\tau' = \tau'_j$. 

            Step. 9. In this step, we classify $\tau'$. 
            From Step. 8, all rays of $\tau'$ are conteind in $\{\tau'_0, \tau'_1, \tau'_2, \tau'_3\}$. 
            We have already known $S(j)$ for $0\leq j\leq 3$. 
            Thus, from Step. 1, there are 7 possible forms for $\tau'$, as follows: $\tau'_0, \tau'_1, \tau'_2, \tau'_3, \tau'_{0, 1} = \tau'_0 + \tau'_1, \tau'_{1, 2} = \tau'_1 + \tau'_2, \tau'_{2, 3} = \tau'_2 + \tau'_3$. 
            Let $\gamma_{i, j}$ denote $(\tau'_{i, j})^\perp\cap D^\sigma_{u\oplus\kappa, \pi^1}$. 
            We can check that all $\langle\gamma_{i, j}\rangle$ is a codimention $2$ subspace of $(M'\oplus\ZZ^\vee\oplus\ZZ)_\RR$. 
            Thus, all $\tau'_{i, j}$ are faces of $C^\sigma_{u\oplus\kappa, \pi^1}$. 
            
            \item[(c)] From Step. 7 in the proof of (b), we can check the statement. 
            \item[(d)] From the definition of $\sigma_1$, $(\pr_1)_\RR\circ \pi^1_\RR (\sigma_1) = \{0\}$. 
            Thus, we can compute $(u\oplus{l\kappa})^{\sigma_1}$ from Proposition \ref{prop: computation of reduction step}.  
            Let $p\colon N\oplus N^\dagger\oplus\ZZ\rightarrow N^\dagger\oplus\ZZ$ be the projection and let $\sigma'_1$ denote $p_\RR(\sigma_1)$. 
            Then we can check that $(\sigma'_1)^{\perp}\cap (M^\dagger\oplus\ZZ^\vee)$ is generated by $\eta_{-1}$ and $\{\eta_j - \eta_0\}_{1\leq j\leq n-3}$. 
            Let $\alpha_1$ denote $((1, 2), (1, 2), \ldots, (1, 2))$. 
            Then $\alpha_1\in S^{\sigma_1}$. 
            Thus, we have $\val^{\sigma_1}_{x_0^d} = \val^{\sigma_1}_{u\oplus{l\kappa}}$. 
            For $(i, j)\in J_1$, let $s'_{i, j}$ denote the following elements in $k[Z^\circ]\otimes_k k[(\sigma'_1)^{\perp}\cap (M^\dagger\oplus\ZZ^\vee)]$:
            \[
                s'_{i, j} = \left\{
                    \begin{array}{ll}
                        x_{-1} & ((i, j) = (0, 2))\\
                        x_{-1}(x^{-1}_0x_{j - 2})y_{j - 2}& (i = 0, j > 2)\\
                        -1 & ((i, j) = (1, 2))\\
                        -(x^{-1}_0x_{j - 2}) & (i = 1, j>2)
                    \end{array}
                \right.
            \]
            Let $u'$ denote the map $S^{\sigma_1}\rightarrow k[Z^\circ]\otimes k[(\sigma'_1)^{\perp}\cap (M^\dagger\oplus\ZZ^\vee)]$ such that $u'(\alpha) = \prod_{(i, j)\in J_1}{s'_{i, j}}^{\alpha_{i, j}}$ for $\alpha\in S^{\sigma_1}$. 
            Then from Proposition \ref{prop: computation of reduction step}(f), $(u\oplus{l\kappa})^{\sigma_1}\sim u'$. 
            Let $f^{\sigma_1}$ denote $\sum_{\alpha\in S^{\sigma_1}}a_\alpha(u\oplus{l\kappa})^{\sigma_1}(\alpha)\in k[(W_{(l)})_{\sigma_1}]$ and $g$ denote $\sum_{\alpha\in S^{\sigma_1}}a_\alpha u'(\alpha)\in k[(W_{(l)})_{\sigma_1}]$. 
            Then $H^\circ_{\overline{(W_{(l)})_{\sigma_1}}, f} = H^\circ_{\overline{(W_{(l)})_{\sigma_1}}, g}$. 
                
            On the other hand, $\{s'_{i, j}\}_{(i, j)\in J_1\setminus\{(1, 2)\}}$ is a transcendence basis and a generator of the fraction field of $k[Z^\circ]\otimes_k k[(\sigma'_1)^{\perp}\cap (M^\dagger\oplus\ZZ^\vee)]$. 
            Moreover, $\Spec(k[Z^\circ]\otimes_k k[(\sigma'_1)^{\perp}\cap (M^\dagger\oplus\ZZ^\vee)])\rightarrow \Spec(k[\{s'_{i, j}\}_{(i, j)\in J_1}]) \cong {\A^{|J_1| - 1}_k}$ is an open immersion. 
            Thus, $H^\circ_{\overline{(W_{(l)})_{\sigma_1}}, g}$ is birational to a general hypersurface of degree $d$ in $\PP^{|J_1| - 1}_k$. 
            In particular, for general $(a_\alpha)_{\alpha\in S}\in k^{|S|}$, $H_{W_{(l)}, f}\cap (W_{(l)})_{\sigma_1} = H^\circ_{\overline{(W_{(l)})_{\sigma_1}}, f^{\sigma_1}}$ is birational to a general hypersurface of degree $d$ in $\PP^{|J_1| - 1}_k$ from Proposition \ref{prop: Bertini for fine} and Lemma \ref{lemma: polytope and stratification}. 
            We remark that $|J_1| = 2n - 4$. 
            \item[(e)] From the definition of $\tau_1$, $(\pr_1)_\RR\circ \pi^1_\RR(\tau_1) = \{0\}$. 
            Thus, we can compute $(u\oplus{l\kappa})^{\tau_1}$ from Proposition \ref{prop: computation of reduction step}(e). 
            Let $\tau'_1$ denote $p_\RR(\tau_1)$. 
            Then we can check that $(\tau'_1)^{\perp}\cap (M^\dagger\oplus\ZZ^\vee)$ is generated by $\eta_{-1}$, $\eta_0 + l\delta$, and $\{\eta_j - \eta_0\}_{1\leq j\leq n-3}$. 
            Let $\alpha_1$ denote $((1, 2), (1, 2), \ldots, (1, 2))$. 
            Then $\alpha_1\in S^{\tau_1}$. 
            Thus, we have $\val^{\tau_1}_{x_0^d} = \val^{\tau_1}_{u\oplus{l\kappa}}$. 
            For $(i, j)\in J_1\cup J_2$, $s''_{i, j}$ denote the following elements in $k[Z^\circ]\otimes_k k[(\tau'_1)^{\perp}\cap (M^\dagger\oplus\ZZ^\vee)]$:
            \[
                s''_{i, j} = \left\{
                    \begin{array}{ll}
                        x_{-1} & ((i, j) = (0, 2))\\
                        x_{-1}(x^{-1}_0x_{j - 2})y_{j - 2}& (i = 0, j > 2)\\
                        -1 & ((i, j) = (1, 2))\\
                        -(x^{-1}_0x_{j - 2}) & (i = 1)\\
                        x_{-1}(x^{-1}_0x_{j-2})(y_{j-2} - 1)(x_0t^l) & (i = 2)\\
                        x_{-1}(x^{-1}_0x_{j-2})(x^{-1}_0x_{i-2})(y_{j-2} - y_{i-2})(x_0t^l) & (i > 2)
                    \end{array}
                \right.
            \]
            Let $u''$ denote the map $S^{\tau_1}\rightarrow k[Z^\circ]\otimes k[(\tau'_1)^{\perp}\cap (M^\dagger\oplus\ZZ^\vee)]$ such that  $u''(\alpha) = \prod_{(i, j)\in J_1\cup J_2}{s''_{i, j}}^{\alpha_{i, j}}$ for $\alpha\in S^{\tau_1}$. 
            Then from Proposition \ref{prop: computation of reduction step}(e), $(u\oplus{l\kappa})^{\tau_1}\sim u''$. 
            Let $f^{\tau_1}$ denote $\sum_{\alpha\in S^{\tau_1}}a_\alpha(u\oplus{l\kappa})^{\tau_1}(\alpha)\in k[(W_{(l)})_{\tau_1}]$ and $g$ denote $\sum_{\alpha\in S^{\tau_1}}a_\alpha u''(\alpha)\in k[(W_{(l)})_{\tau_1}]$. 
            Then $H^\circ_{\overline{(W_{(l)})_{\tau_1}}, f} = H^\circ_{\overline{(W_{(l)})_{\tau_1}}, g}$. 

            We remark that $k[(\tau'_1)^{\perp}\cap (M^\dagger\oplus\ZZ^\vee)] \cong k[(\sigma'_1)^{\perp}\cap (M^\dagger\oplus\ZZ^\vee)]\otimes_k k[(x_0t^l)^{\pm}]$. 
            Moreover, we can check that the following three conditions hold:                 
            \begin{itemize}
                \item[(1)] The equation $S^{\tau_1} = S^{\sigma_0}\coprod S^{\sigma_1}$ from (c). 
                \item[(2)] For any $\alpha\in S^{\sigma_0}$, there exists $v(\alpha)\in k[Z^\circ]^*$ and $w(\alpha)\in k[(\sigma'_1)^{\perp}\cap (M^\dagger\oplus\ZZ^\vee)]^*$ such that $u''(\alpha) = (x_0t^l) v(\alpha)w(\alpha)$. 
                \item[(3)] For any $\beta\in S^{\sigma_1}$, there exists $v(\beta)\in k[Z^\circ]^*$ and $w(\beta)\in k[(\sigma'_1)^{\perp}\cap (M^\dagger\oplus\ZZ^\vee)]^*$ such that $u''(\beta) = v(\beta)w(\beta)$.
            \end{itemize}
            Thus, for general $(a_i)_{i\in S}\in k^{|S|}$, $H_{W_{(l)}, f}\cap (W_{(l)})_{\tau_1} = H^\circ_{\overline{(W_{(l)})_{\tau_1}}, f^{\tau_1}}$ from Proposition \ref{prop: Bertini for fine} and Lemma \ref{lemma: polytope and stratification}.                 
            Moreover, by taking a more general element,  $H^\circ_{\overline{(W_{(l)})_{\tau_1}}, g}$ is irreducible and rational from Proposition \ref{prop:simple-slice-rational}.                 
            Therefore, for general $(a_\alpha)_{\alpha\in S}\in k^{|S|}$, $H_{W_{(l)}, f}\cap (W_{(l)})_{\tau_1}$ is irreducible and rational. 
            Similarly, we can check that the statement holds for $H_{W_{(l)}, f}\cap (W_{(l)})_{\tau_2}$. 
        \end{enumerate}       
        \end{proof}
        The following theorem is the main theorem of this paper. 
        
        \begin{theorem}\label{thm: d in Grassmanian}
            If a very general hypersurface of degree $d$ in $\PP_\C^{2n-5}$ is not stably rational, then a very general hypersurface of degree $d$ in $\Gr_\C(2, n)$ is not stably rational.
        \end{theorem}
        \begin{proof}
             We use the notations in Proposition \ref{prop: Bertini-stably rational}, Proposition \ref{prop:grassman-mock}, and Proposition \ref{prop: computation result}. 
            Let $k$ denote $\C$ and $X$ denote $\Gr_\C(2, n)$. 
            We define the open subsets $U_0, U_1, U_2$ and $U_3$ of $\A^{|S|}_\C$ as follows: 
            \begin{enumerate} 
                \item[(0)] 
                From Lemma \ref{lem: modified fan}(e), there exists a positive integer $l$ and a refinement $\Delta'$ of $(\Delta(\mathcal{B})\times\Delta_!)_{u\oplus{l\kappa}, \pi^1}$ such that $\Delta'$ is generically unimodular, specifically reduced, and compactly arranged. 
                Let $\Delta''$ denote $(\Delta(\mathcal{B})\times\Delta_!)_{u\oplus{l\kappa}, \pi^1}$. 
                From Proposition\ref{prop: computation result}(a) and Lemma \ref{lem: modified fan}, $\Delta''$ is strongly convex. 
                Let $W'$ denote the mock toric variety induced by $Z^1$ along $(\pi^1)_*\colon X(\Delta'')\rightarrow X(\Delta(\mathcal{B})\times\Delta_!)$ and $s^1$. 
                We can check that $\Delta''_{\bdd}\subset \Delta''^{\geq 2}_{u\oplus l\kappa}$ from Proposition \ref{prop: dimension of polytope}, Proposition \ref{prop: computation of reduction step}(e), and Proposition \ref{prop: computation result}(c). 
                
                Let $U_0$ denote an open subset of $\A^{|S|}_\C$ in Proposition \ref{prop:Bertini computation}. 
                For the notation in Proposition \ref{prop: computation result} (b), let $\{E^{(j)}_{\tau_i}\}_{1\leq j\leq r_{\tau_i}}$ be all irreducible components of $H_{W', f}\cap W'_{\tau_i}$ for $0\leq i\leq 3$ and let $\{E^{(j)}_{\sigma_i}\}_{1\leq j\leq r_{\sigma_i}}$ be all irreducible components of $H_{W', f}\cap W'_{\sigma_i}$ for $0\leq i\leq 2$. 
                \item[(1)] Let $\lambda$ denote the automorphism of $\A^{|S|}_\Kt$ defined as $\lambda((c_\alpha)_{\alpha\in S}) = (t^{-l\kappa(\alpha)}c_\alpha)_{\alpha\in S}$ for $(c_\alpha)_{\alpha\in S}\in \Kt^{|S|}$. 
                Let $h$ denote a canonical map $\A^{|S|}_\Kt\rightarrow \A^{|S|}_\C$. 
                From Proposition \ref{prop:grassman-mock}(e), there exists an open subset $V$ of $\A^{|S|}_\Kt$ such that $X_D$ and $Y_D$ are irreducible and birational for any $D\in V(\Kt)$. 
                Let $U_1$ denote an open subset $h\circ\lambda(h^{-1}(U_{X, d})\cap V)$
                of $\A^{|S|}_\C$. 
                For $(a_\alpha)_{\alpha\in S}\in U_1(\C)$, let $D_0$ denote $(a_\alpha t^{l\kappa(\alpha)})_{\alpha\in S}\in \Kt^{|S|}$. 
                Then $D_0\in (U_{X, d})_\Kt(\Kt)$. 
                Moreover, let $f$ denote $\sum_{\alpha\in S}a_\alpha u(\alpha)t^{l\kappa(\alpha)}\in k[(Z^1)_{\pi^1}]$. 
                Then $X_{D_0}$ and $Y_{D_0} = H_{(Z^1)_{\pi^1}, f}\times_{\Gm^1}\Spec(\Kt)$ are birational. 
                In particular, $H_{(Z^1)_{\pi^1}, f}\times_{\Gm^1}\Spec(\Kt)$ is irreducible.
                \item[(2)] Let $U_2$ denote an open subset of $\A^{|S|}_\C$ in Proposition \ref{prop: computation result}(d). 
                In particular, $r_{\sigma_1} = 1$ and $E^{(1)}_{\sigma_1}$ is birational to a general hypersurface of degree $d$ in $\PP^{2n-5}_{\C}$. 
                \item[(3)]
                Let $U_3$ denote an open subset of $\A^{|S|}_\C$ in Proposition \ref{prop: computation result}(e). 
                In particular, $r_{\tau_1} = r_{\tau_2} = 1$ and $\{E^{(1)}_{\tau_1}\}_{\mathrm{sb}} = \{E^{(1)}_{\tau_2}\}_{\mathrm{sb}} = \{\Spec(\C)\}_{\mathrm{sb}}$. 
            \end{enumerate}
            Let $U$ denote $U_0\cap U_1\cap U_2 \cap U_3$. 
            Let $(a_\alpha)_{\alpha\in S}\in U(\C)$ be a point. 
            Then from Proposition \ref{prop:Bertini computation}, the following equation holds: 
            \begin{align*}
                \VOL(H^\circ_{W', f}\times_{\Gm^1}\Spec(\Kt)) &= \sum_{i = 0, 3}\biggl(\sum_{1\leq j\leq r_{\tau_i}} \bigl\{E^{(j)}_{\tau_i}\bigr\}_{\mathrm{sb}}\biggr)-\sum_{i = 0, 2}\biggl(\sum_{1\leq j\leq r_{\sigma_i}} \bigl\{E^{(j)}_{\sigma_i}\bigr\}_{\mathrm{sb}}\biggr)\\
                &+ 2\{\Spec(\C)\}_{\mathrm{sb}}-\{E^{(1)}_{\sigma_1}\}_{\mathrm{sb}}
            \end{align*}

            On the other hand, let $(b_\alpha)_{\alpha\in S^{\sigma_1}}\in \C^{|S^{\sigma_1}|}$ be a vector and $(a'_\alpha)_{\alpha\in S}$ denote the following element in $\C^{|S|}$: 
            \[
                a'_\alpha = 
                \left\{
                    \begin{array}{ll}
                        b_\alpha & 
                        (\alpha\in S^{\sigma_1})\\
                        a_\alpha & 
                        (\alpha\notin S^{\sigma_1})\\
                    \end{array}
                \right.
            \]
            Because $U\cap ((a_\alpha)_{\alpha\in S\setminus S^{\sigma_1}}\times \A^{|S^{\sigma_1}|}_\C)$ is a non-empty open subset of $(a_\alpha)_{\alpha\in S\setminus S^{\sigma_1}}\times \A^{|S^{\sigma_1}|}_\C$, $(a'_\alpha)_{\alpha\in S} \in U$ for general $(b_\alpha)_{\alpha\in S^{\sigma_1}}\in \C^{|S^{\sigma_1}|}$. 
            We assume that $(a'_\alpha)_{\alpha\in S} \in U$. 
            Let $f'$ denote $\sum_{\alpha\in S}a'_\alpha u(\alpha)t^{l\kappa(\alpha)}\in k[(Z^1)_{\pi^1}]$, $\{F^{(j)}_{\tau_i}\}_{1\leq j\leq r'_{\tau_i}}$ be all irreducible components of $H_{W', f'}\cap W'_{\tau_i}$ for $0\leq i\leq 3$, and $\{F^{(j)}_{\sigma_i}\}_{1\leq j\leq r'_{\sigma_i}}$ be all irreducible components of $H_{W', f'}\cap W'_{\sigma_i}$ for $0\leq i\leq 2$. 
            Let $\{f^{\tau_i}\}_{0\leq i\leq 3}$, $\{f^{\sigma_i}\}_{0\leq i\leq 2}$, $\{f'^{\tau_i}\}_{0\leq i\leq 3}$, and $\{f'^{\sigma_i}\}_{0\leq i\leq 2}$ denote the following elements: 
            \begin{itemize}
                \item $f^{\tau_i} = \sum_{\alpha\in S^{\tau_i}}a_\alpha(u\oplus{l\kappa})^{\tau_i}(\alpha)\in k[W'_{\tau_i}]$
                \item $f^{\sigma_i} = \sum_{\alpha\in S^{\sigma_i}}a_\alpha(u\oplus{l\kappa})^{\sigma_i}(\alpha)\in k[W'_{\sigma_i}]$
                \item $f'^{\tau_i} = \sum_{\alpha\in S^{\tau_i}}a'_\alpha(u\oplus{l\kappa})^{\tau_i}(\alpha)\in k[W'_{\tau_i}]$
                \item $f'^{\sigma_i} = \sum_{\alpha\in S^{\sigma_i}}a'_\alpha(u\oplus{l\kappa})^{\sigma_i}(\alpha)\in k[W'_{\sigma_i}]$
            \end{itemize}            
            From Proposition \ref{prop: computation result}(c), $S^{\sigma_1}\cap S^{\tau_i} = \emptyset$ for any $i\in\{0, 3\}$ and $S^{\sigma_1}\cap S^{\sigma_i} = \emptyset$ for any $i\in\{0, 2\}$. 
            This shows that $f^{\tau_i} = f'^{\tau_i}$ for any $i = 0, 3$ and $f^{\sigma_i} = f'^{\sigma_i}$  for any $i = 0, 2$. 
            Thus, from Lemma \ref{lemma: polytope and stratification}, $r_{\tau_i} = r'_{\tau_i}$ for $i = 0, 3$ and if it is necessary, we can replace the index such that $E^{(j)}_{\tau_i} = F^{(j)}_{\tau_i}$ for any $1\leq j\leq r_{\tau_i}$. 
            Similarly, $r_{\sigma_i} = r'_{\sigma_i}$ for $i = 0, 2$ and if it is necessary, we can replace the index such that $E^{(j)}_{\sigma_i} = F^{(j)}_{\sigma_i}$ for any $1\leq j\leq r_{\sigma_i}$. 
            Therefore, the following equation holds: 
            \[
                \VOL(H^\circ_{W', f}\times_{\Gm^1}\Spec(\Kt)) - \VOL(H^\circ_{W', f'}\times_{\Gm^1}\Spec(\Kt)) = \{E^{(1)}_{\sigma_1}\}_{\mathrm{sb}} - \{F^{(1)}_{\sigma_1}\}_{\mathrm{sb}}
            \]
            From the assumption, a very general hypersurface of degree $d$ in $\PP^{2n-5}_\C$ is not stably rational. 
            Thus, from \cite[Corollary.4.2]{NO22}, a very general hypersurface of degree $d$ in $\PP^{2n-5}_\C$ is not stably rational to $E^{(1)}_{\sigma_1}$. 
            Because $\C$ is uncountable, there exists $(a'_\alpha)_{\alpha\in S}\in U$ such that $\{E^{(1)}_{\sigma_1}\}_{\mathrm{sb}} \neq\{F^{(1)}_{\sigma_1}\}_{\mathrm{sb}}$. 
            Thus, for such $(a'_\alpha)_{\alpha\in S}$, either of $H^\circ_{W', f}\times_{\Gm^1}\Spec(\Kt)$ or $H^\circ_{W', f'}\times_{\Gm^1}\Spec(\Kt)$is not stably rational. 
            Because $(a_\alpha)_{\alpha\in S}, (a'_\alpha)_{\alpha\in S}\in U_1(\C)$, a very general hypersurface of degree $d$ in $\Gr_\C(2, n)$ is not stably rational from Proposition \ref{prop: Bertini-stably rational}(b). 
        \end{proof}
        The following corollary holds from \cite[Corollary 1.2]{S19} immediately. 
        \begin{corollary}\label{cor: log bound}
            If $n\geq 5$ and $d \geq 3 + \log_2(n-3)$, then a very general hypersurface of degree $d$ of $\Gr_\C(2, n)$ is not stably rational. 
        \end{corollary}
        \begin{proof}
            From \cite[Corollary 1.2]{S19}, a very general hypersurface of degree $d$ is not stably rational. 
            Thus, the statement follows from Theorem \ref{thm: d in Grassmanian}. 
        \end{proof}
    \bibliographystyle{amsplain}
    \bibliography{yoshino-bib}
\end{document}